\documentclass[a4paper,8pt]{article}
\usepackage[francais]{babel}
\usepackage[T1]{fontenc}
\usepackage{geometry}
\usepackage[latin1]{inputenc}
\usepackage{amsmath,amssymb,mathrsfs,amsthm}
\usepackage{soul}
\usepackage{epsfig}
\usepackage{hyperref}
\usepackage{graphicx}
\usepackage{xypic}
\usepackage{datetime}
\usepackage{fancyhdr}
\fancypagestyle{plain}{
\lhead{}
\rhead{}
}

\newtheorem{theorem}{Théorème}[section]
\newtheorem{proposition}[theorem]{Proposition}
\newtheorem{lemma}[theorem]{Lemme}
\newtheorem{Corollaire}[theorem]{Corollaire}
\newtheorem{definition}[theorem]{Définition}
\newtheorem{example}[theorem]{Exemple}
\newtheorem{remarque}[theorem]{Remarque}
\newcommand{\vc}{\|\cdot\|}
\newcommand{\C}{\mathbb{C}}

\newcommand{\Z}{\mathbb{Z}}

\newcommand{\lra}{\longrightarrow}
\newcommand{\al}{\alpha}
\newcommand{\la}{\lambda}

\newcommand{\R}{\mathbb{R}}
\newcommand{\cl}{\mathcal{C}^\infty}
\newcommand{\p}{\mathbb{P}}
\newcommand{\eps}{\varepsilon}

\newcommand{\vf}{\varphi}

\newcommand{\Q}{\mathbb{Q}}
\newcommand{\N}{\mathbb{N}}
\newcommand{\z}{\overline{z}}
\newcommand{\pt}{\partial}

\title{Sur  la fonction Zêta associée au Laplacien singulier $\Delta_{\overline{\mathcal{O}(m)}_\infty}$}
\date{}
\author{Mounir Hajli}
\begin{document}

\maketitle

\begin{abstract}
Dans \cite{Mounir1}, on a déterminé le spectre du  Laplacien singulier attaché aux métriques canoniques sur $\p^1$. Le but de cet
 article
est l'étude  de  $\zeta_{\Delta_{\overline{\mathcal{O}(m)}_\infty}}$, la fonction Zêta associée à ce spectre. On montre qu'elle est
prolongeable en une fonction holomorphe en zéro et on détermine les valeurs de  $\zeta_{\Delta_{\overline{\mathcal{O}(m)}_\infty}}(0)$ et $\zeta'_{\Delta_{\overline{\mathcal{O}(m)}_\infty}}(0)$ pour tout $m\in
\N$:
\begin{align*}
\zeta_{\Delta_{\overline{\mathcal{O}(m)}_\infty}}(0)&=-\frac{2}{3}-\frac{m}{2},\\
\zeta'_{\Delta_{\overline{\mathcal{O}(m)}_\infty}}(0)&=4\zeta'_\Q(-1)-\frac{1}{6}-\log \frac{\,{}(m+2)^{{m+1}}}{((m+1)!)^2}.
 \end{align*}
\end{abstract}

\tableofcontents

\section{Introduction}
Dans \cite{Mounir1}, on a
construit un opérateur différentiel singulier associé aux métriques canoniques
sur
$\p^1$, qu'on noté par $\Delta_{\overline{\mathcal{O}(m)}_\infty}$,  et on a montré qu'il possède un
spectre discret, positif et infini, qu'on a calculé explicitement.
 Un des objectifs de cet article est de proposer une alternative
aux calculs de Gillet, Soulé et Zagier, cf. \cite{GSZ}, en étudiant le spectre du Laplacien
$\Delta_{\overline{\mathcal{O}(m)}_\infty}$ associé aux métriques canoniques sur $\p^1$.  On souhaite étudier
$\zeta_{\Delta_{\overline{\mathcal{O}(m)}_\infty}} $ la fonction
Zêta associée à ce spectre. Comme les métriques ne sont pas $\cl$ alors la théorie spectrale des opérateurs Laplaciens, voir
par exemple \cite{heat}, n'est plus
applicable à  cette
situation. On développe  une méthode de régularisation de fonctions Zêta, qui va nous permettre dans la suite de déduire
 que
$\zeta_{\Delta_{\overline{\mathcal{O}(m)}_\infty}} $  est prolongeable  en une fonction holomorphe au voisinage de zéro, en
plus on calcule explicitement sa valeur, ainsi que sa dérivée en ce point.\\

Commençons par rappeler quelques éléments de la théorie spectrale des opérateurs Laplaciens, voir par exemple \cite{heat}. Soit $(X,h_X)$ une variété kählérienne compacte et $h_X$ une métrique hermitienne de classe $\cl$, et $(L,h)$ un
fibré en droites holomorphe muni d'une métrique hermitienne $h$ de classe $\cl$ sur $X$. Dans \cite{RaySinger}, Ray et
et Singer associent à la donnée $\bigl((TX,h_X);(L,h)\bigr)$, un réel appelé la torsion analytique holomorphe, qu'on note par $T\bigl((TX,h_X);\overline{L}\bigr)$ est définie par un procédé de régularisation du produit infini des valeurs propres non nuls du  Laplacien $\Delta_{\overline{L}}^q$ agissant sur $A^{(0,q)}(X,L)$ pour $q=0,1,\ldots,\dim_\C(X)$. Si l'on note par $\al_{q,1}\leq \al_{q,2}\leq \ldots$ les valeurs propres non nulles de $\Delta_{\overline{L}}^q$ comptées avec multiplicité, et par
\[
 \zeta_{\Delta_{\overline{L}}^q}(s)=\sum_{k\geq 1}\frac{1}{\al_{q,k}^s},
\]
la fonction Zêta associé, on montre  par des méthodes de la théorie du noyau de la chaleur, voir  \cite{heat}, que  $\zeta_{\Delta_{\overline{L}}}$ converge pour   $\mathrm{Re}(s)>1$ avec un pôle en $s=1$ et s'étend en une fonction méromorphe sur $\C$ entier et qu'elle est holomorphe au voisinage de $s=0$, voir \cite[§ 9.6]{heat}. On pose alors
\[
 T\bigl((TX,h_X);\overline{L}\bigr)=
\sum_{q=0}^n(-1)^{q+1}q \zeta_{\Delta_{\overline{L}}^q}'(0).
\]

Dans \cite{Voros}, Voros étudie la  notion du produit régularisé associée à une suite abstraite de réels
$\{\al_k\}_{k\in \N}$ qui vérifie  les  hypothèses suivantes:
\begin{enumerate}
 \item La suite $\{\al_k\}_{k\in \N}$ est croissante et non bornée.
\item La fonction $\theta(t):=\sum_{k=1}^\infty e^{-t\al_k}$ converge pour tout $t>0$ et elle  admet un
développement en série de Laurent pour $t$ assez petit de la forme suivante:
\[
 \theta(t)=\sum_{n=0}^\infty c_{i_n}t^{i_n}
\]
avec $\{ i_n\}_{n\geq 0}$ une suite de réels croissante non bornée et telle que $i_0<0$.\\
\end{enumerate}

Alors si l'on pose
\[
 Z(s,a):=\sum_{k\geq 1}\frac{1}{(\al_k+a)^s}\quad \text{pour}\; \mathrm{Re}(s)\gg 1,
\]
appelée la fonction Zêta généralisée, voir \cite[p. 444]{Voros},  on  vérifie  qu'elle  s'étend en une
fonction méromorphe sur $\C$  et elle est holomorphe en zéro. On définit donc, le produit régularisé des
réels $\al_1,\al_2,\ldots$ plus généralement celui de $\al_1-\al,\al_2-\al,\ldots$ (pour tout $\al\leq 0$)
comme étant le réel:
\[
 D(\al):=\exp(-\zeta'(0,-\al))
\]
D'après Voros, on appelle  $D(0)$ le déterminant régularisé de la suite $\{\al_k\}_{k\in \N}$. Lorsque
$\{ \al_k\}_{k\geq 1}$ sont les valeurs propres d'un opérateur Laplacien alors, par définition,
on vérifie
que le déterminant régularisé correspond à l'exponentielle de la torsion analytique de
Ray-Singer. Voros
 montre aussi un
résultat important qui décrit explicitement le  comportement asymptotique  de $\log D$ pour
$\al\mapsto
-\infty$ en  fonction des $\{c_{i_n}\}_{n\in \N}$, voir \cite[p. 448 (5.1)]{Voros}.
Réciproquement, si l'on
donne une fonction $\Delta$ dont tous ses zéros sont positifs et admettant un développement
asymptotique
du même type, alors on montre que $\zeta$, la fonction zêta associée à ces zéros s'étend en une
fonction
méromorphe sur $\C$  et elle est holomorphe au voisinage de zéro, en plus on montre que
$\zeta(0)$ et
$\zeta'(0)$ se déduisent du développement asymptotique du $\log \Delta$.\\

La torsion analytique est un ingrédient fondamental de la géométrie d'Arakelov. Dans le cas de
fibré en droites trivial sur l'espace projectif muni de la métrique de Fubini-Study, Ikeda et
Taniguchi  déterminent explicitement le spectre du Laplacien associé à ces métriques, voir
\cite{Spectra}. En se basant sur leurs résultats, Gillet, Soulé et Zagier  calculent
explicitement
la torsion analytique dans ce cas \cite{GSZ}. Gillet et Soulé utilisent ce calcul pour démontrer
leur théorème de Riemann-Roch arithmétique
pour le fibré en droites trivial sur l'espace projectif muni de la métrique de Fubini-Study, voir
\cite[Théorème 2.1.1]{GSZ}.\\

%rappelons que le spectre a été completement calculé pour tout $m\in \N$ dans
%\cite{Mounir1}.% Nous proposerons deux approches différentes et nous montrons que que Pour cela nous allons introduire la notion de famille de fonctions Zêta
%régularisable, nous donnerons une formule  permettant de  calculer explicitement le déterminant
%régularisé pour ce type de fonctions Zêta. Nous montrerons que
 %général permettant d'étudier la fonction Zêta
%associée à  $\Delta_{\overline{\mathcal{O}(m)}_\infty}$ et de calculer explicitement sa valeur

%Dans cet article nous proposons une alternative aux calculs de Gillet, Soulé et Zagier dans
%\cite{GSZ} basée sur nos
%précédents calculs du spectre du Laplacien $\Delta_{\overline{\mathcal{O}(m)}_\infty}$, voir
%\cite{Mounir1} et la
%notion de familles de fonctions Zêta régularisable qu'on introduit dans ce texte.\\

Ce texte  est constitué de deux parties logiquement indépendantes:\\

Dans la section \eqref{tttt}, on étend la notion de la torsion analytique holomorphe aux fibrés
en droites munis de
métriques admissibles sur $\p^1$. Rappelons qu'une métrique est dite admissible si elle est
limite uniforme d'une suite de métriques positives et de classe $\cl$. Notons qu'une métrique
admissible   peut être singulière, dans ce cas  la théorie de noyau de
chaleur développée dans \cite{heat} n'est plus valable et donc on ne peut pas associer
directement une  torsion analytique  à ce genre de métrique. Notre idée  consiste en premier
temps à approximer une métrique
admissible par une suite de métriques positives de classe $\cl$ et d'étudier la variation de la
torsion analytique associée à cette suite, moyennant  la formule des anomalies. On peut
aussi considérer une suite de métriques positives de classe $\cl$ qui converge uniformément vers
une métrique  admissible sur $T\p^1$ et d'étudier le comportement de la torsion analytique dans
ce cas. En résumé, lorsque  $\overline{L}$ et $\overline{T\p^1}$ sont munis de métriques
admissibles, nous établissons qu'on peut étendre la notion de torsion analytique à cette situation, voir (théorème
\eqref{torsiongeneralisée}), qu'on
appellera  torsion analytique holomorphe
généralisée associée à $\overline{L}$ et $\overline{T\p^1}$ et  on la  notera par $T_g\bigl(\overline{T\p^1};\overline{L}\bigr)$. Nous nous intéresserons ensuite à une classe de métriques admissibles non nécessairement $\cl$ à savoir les   métriques canoniques sur
$\p^1$. Ces dernières ont la particularité
d'être  déterminées uniquement par la combinatoire de $\p^1$, vue comme variété torique. Nous
appliquerons la théorie developpée ci-dessus pour   calculer $T_g\bigl(\overline{T\p^1}_\infty;
\overline{\mathcal{O}(m)}_\infty \bigr)$; la valeur de la  torsion analytique généralisée associée
à $\overline{T\p^1}_\infty$ et $
\overline{\mathcal{O}(m)}_\infty $ munis de leur métriques canoniques. On obtient:
 \[
T_g\bigl(\overline{T\p^1}_\infty;\overline{\mathcal{O}(m)}_\infty\bigr)=4\zeta'_\Q(-1)-\frac{1}{6}
-\log
\frac{\,{}(m+2)^{{m+1}}}{\bigl((m+1)!\bigr)^2}\quad \forall m\in \N.
 \]

 La section \eqref{partie222} constitue une approche directe
pour définir une torsion analytique
associée $(\overline{T\p^1}_\infty;
\overline{\mathcal{O}(m)}_\infty)$. Rappelons que dans \cite{Mounir1}, nous avons calculé
explicitement le spectre de
$\Delta_{\overline{\mathcal{O}(m)}_\infty}$. Il est donc naturel  d'associer   à ce spectre une
fonction $\zeta_{\Delta_{\overline{\mathcal{O}(m)}_\infty}}$, qu'on appellera la fonction Zêta
associée à l'opérateur singulier $\Delta_{\overline{\mathcal{O}(m)}_\infty}$.  Afin d'étudier
cette fonction Zêta, nous introduisons la notion de familles de fonctions Zêta, voir
\eqref{paragrapheG}, et nous généralisons
la théorie de Voros à cette classe de fonctions Zêta en établissant un résultat  qui généralise ceux de \cite{Voros} et de
\cite{Spreafico}, voir (théorème \eqref{regularisation}). Grâce à la théorie développée dans cette section, nous sommes capables d'établir
que   $\zeta_{\Delta_{\overline{\mathcal{O}(m)}_\infty}}$ possède exactement les mêmes propriétés qu'une fonction Zêta associée
à un opérateur Laplacien vérifiant les hypothèses de \cite{heat}, plus précisément on démontre le résultat suivant (voir théorème
\eqref{lavaleurdezeta}):
\begin{theorem}
Pour tout $m\in \N$, la fonction $\zeta_{\Delta_{\overline{\mathcal{O}(m)}_\infty}}$ converge pour tout $\mathrm{Re}(s)>1$, avec un pôle en $1$ et  admet un prolongement analytique au voisinage de $s=0$, de plus, on a

\[
\zeta_{\Delta_{\overline{\mathcal{O}(m)}_\infty}}(0)=-\frac{2}{3}-\frac{m}{2},
\]
et

\[\zeta_{\Delta_{\overline{\mathcal{O}(m)}_\infty}}'(0)=4\zeta_\Q'(-1)-\frac{1}{6}-\log \frac{\,{}(m+2)^{{m+1}}}{(m+1)!^2}.\]
\end{theorem}

Rappelons que lorsqu'on considère des métriques  $\cl$ sur $T\p^1$ et sur $\mathcal{O}(m)$, alors
par définition:
\[
T\bigl(\overline{T\p^1};\overline{\mathcal{O}(m)}\bigr)=\zeta'_{\Delta_{\overline{\mathcal{O}(m)}}
}(0)\quad \forall \,m\in \N
\]

A priori la théorie classique du noyau de chaleur ne permet pas de déduire une relation entre
la torsion
analytique généralisée $T_g\bigl(\overline{T\p^1}_\infty;\overline{\mathcal{O}(m)}_\infty\bigr)$
et le déterminant régularisé canonique $\zeta'_{\Delta_{\overline{\mathcal{O}(m)}_\infty}}(0) $.
Mais d'après nos calculs, nous avons établit
  que les deux différentes approches   donnent le même résultat, c'est à dire:
\[
T_g\bigl(\overline{T\p^1}_\infty;\overline{\mathcal{O}(m)}_\infty\bigr)=\zeta'_{\Delta_{\overline{
\mathcal{O}(m)}_\infty}}(0)\quad \forall \,m\in \N.
\]
\vspace{1cm}

\noindent\textbf{Remerciements}: Cet article est une partie de ma thèse (voir \cite{these})
 sous la direction de Vincent Maillot. Je le remercie pour
ses conseils et son aide  lors de la préparation de ce travail.

\section{La torsion analytique holomorphe généralisée  sur $\p^1$}\label{tttt}

Le but de ce paragraphe est d'étendre la notion de torsion analytique holomorphe aux fibrés en droites sur $\p^1$ munis d'une métrique admissible.

\subsection{Métriques admissibles}
Soit $X$ une variété complexe analytique et $\overline{L}=(L,\vc)$ un fibré en droites hermitien muni d'une métrique continue sur $L$.
\begin{definition}
On appelle premier courant de Chern de $\overline{L}$ et on note $c_1\bigl( \overline{L}\bigr)\in D^{(1,1)}(X)$ le courant défini localement par l'égalité:
\[
c_1\bigl(\overline{L}\bigr)=dd^c\bigl( -\log \|s\|^2\bigr),
\]
où $s$ est une section holomorphe locale et ne s'annulant pas du fibré  $L$.
\end{definition}

\begin{definition}
La métrique $\vc$ est dite positive si $c_1\bigl(L,\vc\bigr)\geq 0$.
\end{definition}
\begin{definition}
 La métrique $\vc$ est dite admissible s'il existe une famille $\bigl(\vc_n \bigr)_{n\in \N}$ de métriques positives de classe $\cl$  convergeant uniformément vers $\vc$ sur $L$. On appelle fibré admissible sur $X$ un fibré en droites holomorphe muni d'une métrique admissible sur $X$.
\end{definition}
On dira que $\overline{L}$ est un fibré en droites intégrable s'il existe $\overline{L}_1$ et $\overline{L}_2$ admissibles tels que
\[
 \overline{L}=\overline{L}_1\otimes \overline{L}_2^{-1}.
\]

\begin{example}
Soit $n\in \N^\ast$. On note par $\mathcal{O}(1)$ le fibré de Serre sur $\p^n$ et on le munit de la métrique définie pour toute section méromorphe de $\mathcal{O}(1)$ par:
\[
 \|s(x)\|_\infty=\frac{|s(x)|}{\max(|x_0|,\ldots,|x_n|)}.
\]
Cette  métrique est admissible.
\end{example}
En fait, c'est un cas particulier d'un résultat plus général combinant   la construction Batyrev
et Tschinkel  sur une variété torique projective et  la construction de Zhang. Dans la première
construction permet d'associer canoniquement à tout fibré en droites sur une variété torique
projective complexe une métrique continue notée $\vc_{BT}$ et déterminée uniquement par la
combinatoire de la variété, voir \cite[proposition 3.3.1]{Maillot} et \cite[proposition
3.4.1]{Maillot}. L'approche de Zhang est moins directe, elle utilise un endomorphisme équivariant
(correspondant à la multiplication par $p$, un entier supérieur à 2) afin de construire par
récurrence une suite de métriques qui converge uniformément vers une limite notée $\vc_{Zh,p}$ et
qui, en plus, ne dépend pas  du choix de la métrique de départ, voir \cite{Zhang} ainsi que
\cite[théorème 3.3.3]{Maillot}. Mais d'après \cite[théorème 3.3.5]{Maillot} on montre que
\[
 \vc_{BT}=\vc_{Zh,p}.
\]
Que l'on appelle la métrique canonique associée à $L$. Notons que lorsque $L$ n'est pas trivial,
alors cette métrique  n'est pas $\cl$.

\subsection{La métrique de Quillen et la torsion analytique holomorphe, un rappel}
Commençons tout d'abord par faire des rappels sur la métrique de Quillen. Soit $\bigl(X,\omega_X\bigr)$ une variété compacte kählérienne et $\overline{E}$ un fibré vectoriel muni d'une métrique hermitienne $\cl$.  Ces données définissent un produit hermitien $L^2$, sur $A^{(0,q)}\bigl(X,E\bigr)$ pour $q=0,1\ldots,\dim_\C(X)$.\\

Soit $\Delta_q$ l'opérateur Laplacien agissant sur $A^{(0,q)}\bigl(X,E\bigr)$ et $\mathcal{H}^{(0,q)}=\ker \Delta_q$ l'ensemble des formes harmoniques. On note par
\[
 \zeta_q(X,E,s)=\sum_{n\in \N_{\geq 1}}\frac{1}{\la_{q,n} ^{s}}
\]
la fonction Zêta associée au spectre de $\Delta_q$, cette série converge absolument sur $\mathrm{Re}(s)>\dim_\C\bigl( X\bigr)$ et admet un prolongement méromorphe sur $\C$ entier et  son prolongement  est holomorphe en $s=0$. D'après Ray et Singer \cite{RaySinger} on définit la torsion analytique comme étant le réel:
\[
T\Bigl( \bigl(X,\omega_X\bigr); \overline{E}\Bigr)=\sum_{q\in \N}(-1)^{q+1}\,q\,\zeta'_q(X,E,0).
\]
On considère l'espace vectoriel complexe de dimension $1$ suivant:
\[
 \la\bigl(E\bigr)=\otimes_{q\geq 0} \det\bigl(H^q(X,E) \bigr)^{(-1)^q},
\]
où $H^q(X,E)$ est le $q$-ème groupe de cohomologie de $X$ à coefficients dans $E$. Comme
$H^q(X,E)$ est canoniquement  isomorphe à $\ker\bigl(\Delta_q\bigr)$, le produit hermitien
$L^2$ induit une métrique $h_{L^2}$ sur $\la\bigl(E\bigr)$. Dans \cite{Quillen}, Quillen  définit
la métrique $h_Q$ sur $\la\bigl(E\bigr)$  par la formule:
\[
 h_Q=h_{L^2}\exp\Bigl( T\Bigl( \bigl(X,\omega_X\bigr); \overline{E}\Bigr)\Bigr).
\]

Le résultat suivant permet d'étendre, par approximation, la notion de torsion analytique holomorphe aux fibrés en
droites intégrables sur $\p^1$.

\begin{theorem}\label{torsiongeneralisée}
On munit $T\p^1$ d'une métrique intégrable, qu'on note par $h_{\infty,\p^1}$. Soit $\overline{L}_\infty:=\bigl(L,h_{\infty,L} \bigr)$ un fibré en droites admissible sur $\p^1$.\\

On considère $\bigl( h_{n,L}\bigr)_{n\in\N}$  et $\bigl(h_{n,\p^1} \bigr)_{n\in \N}$, deux  suites de métriques  $L$ et $T\p^1$ respectivement vérifiant:

\begin{enumerate}
 \item
Il existe $H_1$ et $H_2$ deux fibrés en droites sur $\p^1$ et deux suites de métriques $\bigl(h_{n,1} \bigr)$ et $\bigl(h_{n,2} \bigr)$ de métriques positives de classe $\cl$ qui convergent uniformément  vers deux métriques $h_{\infty,1,\p^1}$ et $h_{\infty,2,\p^1}$ respectivement sur $H_1$ et $H_2$ telles que

\[
 T\p^1=H_1\otimes H_2^{-1},
\]

\[h_{\infty,\p^1}=h_{\infty,1}\otimes h_{\infty, 2}^{-1},\]
et
 \[
 h_{n,\p^1}=h_{n,1}\otimes h_{n,2}^{-1}\quad \forall \, n\in \N.
\]
\item
$\bigl(h_{n,L}\bigr)_{n\in \N}$ est une suite de métriques positives de classe $\cl$ sur $L$ qui converge  uniformément vers $h_{\infty,L}$.
\end{enumerate}
On note pour tout $n\in \N$, $n'\in \N$:
\[
 \Bigl(\bigl(T\p^1,h_{n,T\p^1}\bigr);\bigl(L,h_{n',L}\bigr) \Bigr)
\]
 la torsion analytique holomorphe de Ray-Singer associée à
$\Bigl(\bigl(T\p^1,h_{n,T\p^1}\bigr);\bigl(L,h_{n',L}\bigr) \Bigr)$.\\
 Alors la suite double\footnote{Dans ce texte, on fera la convention suivante: On dira
qu'une suite double $(x_{n,n'})_{n\in \N,n'\in\N}$ converge vers une limite $l$, si
pour tout $\eps>0$, il existe $A\in \R$ tel que $|x_{n,n'}-l|<\eps$ pour  $n,n'>A$.}suivante:
\[
\biggl(T\Bigl(\bigl(T\p^1,h_{n,T\p^1}\bigr);\bigl(L,h_{n',L}\bigr) \Bigr)\biggr)_{n\in \N,\,n'\in
\N}
\]
converge vers une limite finie qui ne dépend pas du choix des suites ci-dessus.
\end{theorem}

\begin{proof}
On considère $\bigl(h_{n,\p^1}\bigr)_{n\in \N}$ et $\bigl(h_{n,L}\bigr)_{n\in\N} $ comme dans le
théorème. \\

On fixe la notation suivante:
\[
 h_{Q,\overline{T\p^1},\overline{L}},
\]
la métrique de Quillen associée à la donnée $\Bigl(\overline{T\p^1};\overline{L} \Bigr)$ où toutes les métriques sont $\cl$.\\

Par les formules des anomalies \cite[théorèmes 0.1, 0.2]{BGS1} on a pour tout $k,k'\in \N$ et $j,j'\in \N$:
\[
\log h_{Q,{\footnotesize \overline{T\p^1}_k; \overline{L}_j}}-\log h_{Q,{\footnotesize \overline{T\p^1}_k; \overline{L}_{j'}}}=
\int_{\p^1}\widetilde{ch}\bigl(L,h_j,h_{j'}\bigr)\mathrm{Td}(\overline{T\p^1}_k)
\]
et
\[
\log h_{Q,{\footnotesize \overline{T\p^1}_{k}; \overline{L}_j}}-\log h_{Q,{\footnotesize \overline{T\p^1}_{k'}; \overline{L}_{j}}}=\int_{\p^1}ch(\overline{L}_j)\widetilde{\mathrm{Td}}\bigl(T\p^1,h_k,h_{k'}\bigr).
\]

On montre que, voir \eqref{ch}:
\[
\widetilde{ch}\bigl(L,h_{j},h_{j'}\bigr)=-\log\biggl(\frac{h_j}{h_{j'}}\biggr)-\frac{1}{2}\log\biggl(\frac{h_{j}}{h_{j'}}\biggr)\Bigl(c_1(\overline{L}_j)+c_1(\overline{L}_{j'}) \Bigr),
\]
et
\[
\widetilde{\mathrm{Td}}\bigl(T\p^1,h_{k},h_{k'}\bigr)=-\frac{1}{2}\log\biggl( \frac{h_{k}}{h_{k'}}\biggr)-\frac{1}{12}\log\biggl(\frac{h_{k}}{h_{k'}}\biggr)\Bigl(c_1(\overline{T\p^1}_k)+c_1(\overline{T\p^1}_{k'}) \Bigr).
\]
modulo $\mathrm{Im}\, \overline{\partial}+\mathrm{Im}\,\partial$.\\

Comme
\[
 \overline{T\p^1}_k=\overline{H_1}_k\otimes \overline{H_2}_k^{-1}\quad \forall \, k\in \N
\]
et que $\overline{H_1}_k=(H_1,h_{k,1})$  et $\overline{H_2}_k=(H_2,h_{k,2})$ sont positifs. Par les formules précédentes  on peut écrire que:
\begin{align*}
\biggl|\log h_{Q,{\footnotesize \overline{T\p^1}_k; \overline{L}_j}}-\log h_{Q,{\footnotesize
\overline{T\p^1}_k; \overline{L}_{j'}}} \biggr|&=\biggl| \int_{\p^1}\Bigl(-\log \frac{h_j}{h_{j'}}-\frac{1}{2}\log
\frac{h_{j}}{h_{j'}}\bigl(c_1(\overline{L}_j)+c_1(\overline{L}_{j'}) \bigr)  \Bigr)Td(\overline{H_1}_k\otimes \overline{H_2}_k^{-1} )  \biggr|\\
&=\biggl| \int_{\p^1}-\frac{1}{2}\log \frac{h_j}{h_{j'}}\Bigl(c_1(\overline{H_1}_k)-c_1(\overline{H_2}_k)
\Bigr)-\frac{1}{2}\int_{\p^1}\log
\frac{h_{j}}{h_{j'}}\bigl(c_1(\overline{L}_j)+c_1(\overline{L}_{j'}) \bigr)  \Bigr)\biggl|\\
&\leq\frac{1}{2} \biggl\| \log\biggl(\frac{h_{j,L}}{h_{j',L}}\biggr)
\biggr\|_{\sup}\int_{\p^1}\Bigl(c_1(\overline{H_1}_k)+c_1(\overline{H_2}_k)+c_1(\overline{L}_j)+c_1(\overline{L}_{j'})\Bigr)\\
&=\frac{1}{2} \biggl\| \log\biggl(\frac{h_{j,L}}{h_{j',L}}\biggr)
\biggr\|_{\sup}\int_{\p^1}\Bigl(c_1(H_1)+c_1(H_2)+2c_1(L)\Bigr).
\end{align*}

 Donc, on a  trouvé une constante $M$ qui ne dépend pas de $k$ telle que
\[
\biggl|\log h_{Q,{\footnotesize \overline{T\p^1}_k; \overline{L}_j}}-\log h_{Q,{\footnotesize
\overline{T\p^1}_k; \overline{L}_{j'}}} \biggr|\leq M \biggl\| \log\biggl(\frac{h_{j,L}}{h_{j',L}}\biggr)
\biggr\|_{\sup}\quad \forall \, j,j'\in \N.
\]

On a aussi,
{\allowdisplaybreaks
\begin{align*}
\biggl|\log h_{Q,{\footnotesize \overline{T\p^1}_{k}; \overline{L}_j}}-\log h_{Q,{\footnotesize \overline{T\p^1}_{k'}; \overline{L}_{j}}}\biggr|&=\biggl|\int_{\p^1} ch(\overline{L}_j)\Bigl(-\frac{1}{2}\log \frac{h_{k}}{h_{k'}}-\frac{1}{12}\log \frac{h_{k}}{h_{k'}}\Bigl(c_1(\overline{T\p^1}_k)+c_1(\overline{T\p^1}_{k'}) \Bigr)\Bigr)\biggr|\\
&=\biggl|\frac{1}{2}\int_{\p^1}\log \frac{h_{k}}{h_{k'}} c_1(\overline{L}_j)-\frac{1}{12}\int_{\p^1}\log \frac{h_{k}}{h_{k'}}\Bigl(c_1(\overline{H_1}_k)+c_1(\overline{H_1}_{k'})-c_1(\overline{H_2}_k)+c_1(\overline{H_2}_{k'} \Bigr)\Bigr)\biggr|\\
&\leq \frac{1}{2}\biggl\|\log\biggl(\frac{h_{k,\p^1}}{h_{k',\p^1}}\biggr)
\biggr\|_{\sup}\int_{\p^1}c_1(\overline{L}_j)\\
&+\frac{1}{12}\biggl\|\log\biggl(\frac{h_{k,\p^1}}{h_{k',\p^1}}\biggr)
\biggr\|_{\sup}\int_{\p^1}\Bigl(c_1(\overline{H_1}_k)+c_1(\overline{H_1}_{k'})+c_1(\overline{H_2}_k)+c_1(\overline{H_2}_{k'})\Bigr)\\
&=\frac{1}{2}\biggl\|\log\biggl(\frac{h_{k,\p^1}}{h_{k',\p^1}}\biggr)
\biggr\|_{\sup}\int_{\p^1}c_1(L)+\frac{1}{6}\biggl\|\log\biggl(\frac{h_{k,\p^1}}{h_{k',\p^1}}\biggr)
\biggr\|_{\sup}\int_{\p^1}\Bigl(c_1(H_1)+c_1(H_2)\Bigr)\\
&\leq \biggl\|\log\biggl(\frac{h_{k,\p^1}}{h_{k',\p^1}}\biggr)
\biggr\|_{\sup}\int_{\p^1}\Bigl(c_1(L)+c_1(H_1)+c_1(H_2)\Bigr)\\
&\leq \biggl\|\log\biggl(\frac{h_{k,1}}{h_{k',1}}\biggr)
\biggr\|_{\sup}\biggl\|\log\biggl(\frac{h_{k,2}}{h_{k',2}}\biggr)
\biggr\|_{\sup}\int_{\p^1}\Bigl(c_1(L)+c_1(H_1)+c_1(H_2)\Bigr)
\end{align*}
}

On a trouvé $M'$, une constante qui ne dépend pas de $j$ telle que
\[
\biggl|\log h_{Q,{\footnotesize \overline{T\p^1}_{k}; \overline{L}_j}}-\log h_{Q,{\footnotesize \overline{T\p^1}_{k'}; \overline{L}_{j}}}\biggr|\leq M'\biggl\|\log\biggl(\frac{h_{k,\p^1}}{h_{k',\p^1}}\biggr)
\biggr\|_{\sup}\leq M' \biggl\|\log\biggl(\frac{h_{k,1}}{h_{k',1}}\biggr)
\biggr\|_{\sup}\biggl\|\log\biggl(\frac{h_{k,2}}{h_{k',2}}\biggr)
\biggr\|_{\sup}\quad \forall \, k,k'\in \N.
\]

Par suite, il existe $M''$ une constante réelle telle que
\[
\biggl|\log h_{Q,{\footnotesize \overline{T\p^1}_k; \overline{L}_j}}-\log h_{Q,{\footnotesize
\overline{T\p^1}_{k'}; \overline{L}_{j'}}} \biggr|\leq M'' \biggl(\biggl\| \log\biggl(\frac{h_{j,L}}{h_{j',L}}\biggr)
\biggr\|_{\sup} +\biggl\|\log\biggl(\frac{h_{k,1}}{h_{k',1}}\biggr)
\biggr\|_{\sup}\biggl\|\log\biggl(\frac{h_{k,2}}{h_{k',2}}\biggr)\biggr\| \biggr)\quad \forall \, j,j',k,k'\in \N.
\]
Puisque, par hypothèse, $(h_{j,L})_{j\in \N}$ (resp. $(h_{k,1})_{k\in \N}$, resp. $(h_{k,2})_{k\in \N}$) converge uniformément  vers
$h_{\infty,L}$ (resp. vers $h_{\infty,1}$, resp. $h_{\infty,2}$), alors la suite double,
\[
\Bigl(\log h_{Q,{\footnotesize \overline{T\p^1}_k; \overline{L}_j}}\Bigr)_{j,k\in\N},
\]
converge vers une limite finie lorsque $k\mapsto \infty$ et $j\mapsto \infty$, qui ne dépend pas du choix des suites.\\

Comme $\overline{L}_\infty$ est un fibré en droites admissible sur $\p^1$, il existe $m\in \N$ tel que $L=\mathcal{O}(m)$. Donc pour tout $k\in\N$ et $j\in \N$
\[
h_{L^2,\overline{T\p^1}_k,\overline{L}_j}=\det\bigl((z^l,z^{l'})_{L^2,k,j} \bigr)_{0\leq l,l'\leq m},
\]
où $(\cdot,\cdot)_{L^2,k,j}$ est la métrique $L^2$ induite par $\overline{T\p^1}_k$ et $\overline{L}_j$ et les $1,z^1,\ldots,z^m$ sont les sections globales de $L$.\\

Pour tous $l,l'=0,1,\ldots, m$, la suite double
\[
\bigl(z^l,z^{l'}\bigr)_{L^2,k,j}=\int_{\p^1}h_{j,L}(z^l,z^{l'})\omega_{k,\p^1},
\]
converge vers $(z^l,z^{l'})_{L^2,\infty,\infty}=\int_{\p^1}h_{\infty,L}(z^l,z^{l'})\omega_{\infty,\p^1}$, où  $(\cdot,\cdot)_{L^2,\infty,\infty} $ désigne le produit scalaire  $L^2$ induit par $h_{\infty,\p^1}$ et $h_{\infty,L}$. Par conséquent la suite de matrices
\[
\Bigl(\bigl((z^l,z^{l'})_{L^2,k,j} \bigr)_{0\leq l,l'\leq m}\Bigr)_{k,j\in \N}
\]
converge vers $\bigl((z^l,z^{l'})_{L^2,\infty,\infty} \bigr)_{0\leq l,l'\leq m}
$ pour une norme matricielle arbitraire. Par des arguments de  l'algèbre linéaire, on conclut que
\[
\Bigl(\det\bigl((z^l,z^{l'})_{L^2,k,j} \bigr)_{0\leq l,l'\leq m}\Bigr)_{k,j\in \N}\xrightarrow[k,j\mapsto \infty]{}\det \bigl((z^l,z^{l'})_{L^2,\infty,\infty} \bigr)_{0\leq l,l'\leq m},
\]
c'est à dire
\[
\Bigl(h_{L^2,\overline{T\p^1}_k,\overline{L}_j}\Bigr)_{k,j\in \N}\xrightarrow[k,j\mapsto \infty]{} h_{L^2,\overline{T\p^1}_\infty,\overline{L}_\infty}.
\]
On a donc montré que la suite double suivante:

\[
\biggl(T\Bigl(\bigl(T\p^1,h_{k,T\p^1}\bigr);\bigl(L,h_{j,L}\bigr) \Bigr)\biggr)_{k\in \N,\,j\in
\N},
\]
converge vers une limite finie.
\end{proof}

\begin{definition}
On note cette limite par \[T_g\Bigl((T\p^1,h_{\infty,\p^1});(L,h_{\infty,L}) \Bigr),\]  et  on l'appelle la torsion analytique généralisée. On définit aussi \[h_{Q,g}:=h_{L^2,g}\exp\Bigl(T_g((T\p^1,h_{\infty,\p^1});(L,h_{\infty,L}) \Bigr)\] qu'on appelle la métrique de Quillen généralisée.\\
\end{definition}

On note par $[x_0:x_1]$ les coordonnées homogènes de $\p^1$ avec $z=\frac{x_1}{x_0}$  lorsque $x_0\neq 0$. On munit $\p^1$ de la métrique singulière suivante:
\[
\omega_\infty=\frac{1}{2\pi i} \frac{dz\wedge d\z}{\max(1,|z|^4)},\]
et pour tout $m\geq 0$, le fibré en droites $\mathcal{O}(m)$ sera muni de sa métrique canonique:
\[
h_\infty(s,s)(z)=\frac{|s(z)|^2}{\max(1,|z| )^{2m}},
\]
où $s$ est une section holomorphe locale de $\mathcal{O}(m)$. On montre que ces deux métriques sont admissibles, voir par exemple \cite{Zhang}.\\

On note par $h_{Q,{\footnotesize ( \overline{T\p^1}_\infty; \overline{\mathcal{O}(m)}_\infty)}}$ la métrique de Quillen généralisée associée à $( \overline{T\p^1}_\infty; \overline{\mathcal{O}(m)}_\infty)$. On a
\begin{proposition}
On a  pour tout $m \in \N$,
\[
h_{Q,{\footnotesize ( \overline{T\p^1}_\infty; \overline{\mathcal{O}(m)}_\infty)}}=\exp\bigl(4\zeta'_\Q(-1)-\frac{1}{6}\bigr),
\]
et
\[
T_g\bigl(\overline{T\p^1}_\infty;\overline{\mathcal{O}(m)}_\infty\bigr)=4\zeta'_\Q(-1)-\frac{1}{6}-\log \frac{\,{}(m+2)^{{m+1}}}{((m+1)!)^2}.
\]
\end{proposition}

\begin{proof}

On considère sur $\p^1$ les deux métriques suivantes:
\[
\omega_{FS}=\frac{i}{2\pi}\frac{dz\wedge d\overline{z}}{(1+|z|^2)^2},\quad \text{et}\quad \omega_\infty =\frac{i}{2\pi}\frac{dz\wedge d\overline{z}}{\max(1,|z|^2)^2}.
\]

Par  \eqref{torsiongeneralisée} et les formules des anomalies \cite[théorèmes 0.1, 0.2]{BGS1} on peut écrire que

\[
\log h_{Q,{\footnotesize \overline{T\p^1}_{FS}; \overline{\mathcal{O}(m)}_\infty}}-\log h_{Q,{\footnotesize \overline{T\p^1}_{FS}; \overline{\mathcal{O}(m)}_{FS}}}=
\int_{\p^1}\widetilde{ch}(\mathcal{O}(m),h_{FS}^{\otimes m},h_{\infty}^{\otimes m})\mathrm{Td}(\overline{T\p^1}_{FS})
\]
et
\[
\log h_{Q,{\footnotesize \overline{T\p^1}_{\infty}; \overline{\mathcal{O}(m)}_\infty}}-\log h_{Q,{\footnotesize \overline{T\p^1}_{FS}; \overline{\mathcal{O}(m)}_{\infty}}}=\int_{\p^1}ch(\overline{\mathcal{O}(m)})\widetilde{\mathrm{Td}}(T\p^1,h_{FS},h_{\infty})
\]

On a
\[
\widetilde{ch}\bigl(\mathcal{O}(m),h_{FS}^{\otimes m},h_{\infty}^{\otimes m}\bigr)=-m\log\biggl(\frac{h_{FS}}{h_{\infty}}\biggr)-\frac{m^2}{2}\log\biggl(\frac{h_{FS}}{h_{\infty}}\biggr)\Bigl(c_1(\overline{\mathcal{O}(1)}_\infty)+c_1(\overline{\mathcal{O}(1)}_{FS}) \Bigr),
\]
et
\[
\widetilde{\mathrm{Td}}\bigl(T\p^1,h_{FS}^2,h_{\infty}^2\bigr)=-\frac{m}{2}\log \biggl(\frac{h_{FS}^2}{h_{\infty}^2}\biggr)-\frac{1}{12}\log\biggl(\frac{h_{FS}^2}{h_{\infty}^2}\biggr)\Bigl(c_1(\overline{\mathcal{O}(2)}_\infty)+c_1(\overline{\mathcal{O}(2)}_{FS}) \Bigr).
\]
ce sont deux formes différentielles généralisées au sens de \cite[§ 4.3]{Maillot},  modulo $\mathrm{Im} \overline{\partial}+\mathrm{Im}\partial$.\\

Rappelons le résultat suivant:
\begin{proposition}
Soit $\overline{L}$ un fibré en droites sur $X$ une variété torique lisse de dimension $d$ muni de sa métrique canonique; pour tout fonction $f$ de classe $\mathcal{C}^\infty$ sur $X$, on a:
\[
 \int_Xf c_1(\overline{L})^d=\mathrm{deg}(L)\int_{S_N^+}fd\mu^+
\]
\end{proposition}
\begin{proof}
 C'est un cas particulier du \cite[Corollaire 6.3.5]{Maillot}.
\end{proof}

Par la précédente proposition,
\begin{equation}\label{intinfty}
\begin{split}
-\int_{\p^1}\log\biggl(\frac{h^2_{FS}}{h^2_\infty}\biggr)c_1(\overline{\mathcal{O}(2)}_\infty)&=\frac{\mathrm{deg}(\mathcal{O}(2))}{2\pi}\int_{0}^{2\pi}\log\biggl(\frac{h_\infty}{h_{FS}}\biggr)(e^{i\theta})d\theta \\
&=\frac{\mathrm{2}}{2\pi}\int_{0}^{2\pi}\log\Bigl(\frac{(1+|e^{i\theta}|^2)^2}{\max(1,|e^{i\theta}|^2)}\Bigr)d\theta \\
&=2\log 4.
\end{split}
\end{equation}

On a,
\begin{equation}\label{intFS}
 \begin{split}
-\int_{\p^1}\log\biggl(\frac{h^2_{FS}}{h^2_\infty}\biggr)c_1(\overline{\mathcal{O}(2)}_{FS})&=\frac{4i}{2\pi}\int_{\C}\log\Bigl(\frac{(1+|z|^2}{\max(1,|z|^2)}\Bigr)\frac{dz\wedge d\overline{z}}{(1+|z|^2)^2}\\
&=4\int_{0}^{+\infty}\log\Bigl(\frac{1+r^2}{\max(1,r^2)} \Bigr)\frac{rdr}{(1+r^2)^2}\\
&=4 \int_{0}^{+\infty}\log\Bigl(\frac{1+u}{\max(1,u)} \Bigr)\frac{du}{(1+u)^2}\\
&=4(1-\log 2).
 \end{split}
\end{equation}

Si l'on note par $[\cdots]^{(0)}$ la partie de degré $0$ dans $\oplus_{k\geq 0}\widetilde{\mathcal{D}^{k,k}}(\p^1)$: l'anneau gradué  des courants de degré $(k,k)$ modulo $Im\,\pt+Im\,\overline{\pt}$, alors
\[
 \begin{split}
\biggl[\int_{\p^1}\widetilde{ch}(\mathcal{O}(m),&h_{FS}^{\otimes m},h_{\infty}^{\otimes m})\mathrm{Td}(\overline{T\p^1}_{FS})\biggr]^{(0)}=\biggl[\int_{\p^1}\widetilde{ch}(\mathcal{O}(m),h_{FS}^{\otimes m},h_{\infty}^{\otimes m})(1+c_1(\overline{\mathcal{O}(1)}_{FS}))\biggr]^{(0)} \\
&=\biggl[\int_{\p^1}\widetilde{ch}(\mathcal{O}(m),h_{FS}^{\otimes m}, h_\infty^{\otimes m})\biggr]^{(0)}+m\int_{\p^1}\biggl(-\log \frac{h_{FS}}{h_\infty} \biggr)c_1(\overline{\mathcal{O}(1)}_{FS})\\
&=-\frac{m^2}{2}\int_{\p^1}\log\frac{h_{FS}}{h_{\infty}}\Bigl(c_1(\overline{\mathcal{O}(1)}_\infty)+c_1(\overline{\mathcal{O}(1)}_{FS}) \Bigr)+m\int_{\p^1}\Bigl(-\log \frac{h_{FS}}{h_\infty} \Bigr)c_1(\overline{\mathcal{O}(1)}_{FS})\\
&=\frac{m^2}{2}\bigl(\log 2+(1-\log 2) \bigr)+m(1-\log2) \quad \text{par}\, \eqref{intinfty} \, \text{et}\, \eqref{intFS}\\
&=\frac{m^2}{2}+(1-\log 2)m,
 \end{split}
\]

et

\[
 \begin{split}
\biggl[\int_{\p^1}ch(\overline{\mathcal{O}(m)}_\infty)&\widetilde{\mathrm{Td}}(T\p^1,h_{FS},h_{\infty})\biggr]^{(0)}=
\biggl[\int_{\p^1}\Bigl(1+m c_1(\overline{\mathcal{O}(1)}_\infty)\Bigr)\widetilde{\mathrm{Td}}\bigl(T\p^1,h_{FS},h_{\infty}\bigr)\biggr]^{(0)}\\
&=\biggl[\int_{\p^1}\widetilde{\mathrm{Td}}(T\p^1,h_{FS},h_{\infty})\biggr]^{(0)}-\frac{1}{2}m\int_{\p^1}\Bigl(\log\frac{h_{FS}}{h_{\infty}}\Bigr)c_1(\overline{\mathcal{O}(1)}_\infty)\\
&=-\frac{1}{12}\int_{\p^1}\log\frac{h_{FS}}{h_{\infty}}\Bigl(c_1(\overline{\mathcal{O}(1)}_\infty+c_1(\overline{\mathcal{O}(1)}_{FS}) \Bigr)-\frac{1}{2}m\int_{\p^1}\Bigl(\log\frac{h_{FS}}{h_{\infty}}\Bigr)c_1(\overline{\mathcal{O}(1)}_\infty)\\
&=\frac{1}{3}+m\log 2\quad \text{par}\; \eqref{intinfty}\; \text{et}\; \eqref{intFS}.
 \end{split}
\]

En regroupant tout cela, on obtient
\begin{equation}\label{diffFSCAN}
\log h_{Q,{\footnotesize \overline{T\p^1}_{\infty}; \overline{\mathcal{O}(m)}_\infty}}-\log h_{Q,{\footnotesize \overline{T\p^1}_{FS}; \overline{\mathcal{O}(m)}_{\infty}}}=\Bigl(\frac{m^2}{2}+(1-\log2)m\Bigr)+\Bigl( \frac{1}{3}+m\log 2\Bigr) =\frac{m^2}{2}+m+ \frac{1}{3}.
\end{equation}

Rappelons que pour  calculer $h_{Q,{\footnotesize \overline{T\p^1}_{FS}; \overline{\mathcal{O}(m)}_{FS}}}
$; la métrique de Quillen associée à l'espace
projectif $\p^N$ muni de la métrique de Fubini-Study et $\overline{\mathcal{O}(m)}_{FS}$, l'idée consiste à utiliser $T(
\overline{T\p^N}_{FS}; \overline{\mathcal{O}}_{0})$ qui est déterminé dans \cite{GSZ}, et d'appliquer la formule des immersions de Bismut-Lebeau, (voir \cite{ImmerBismut}) appliquée à l'inclusion naturelle: $i:\p^{N-1}\lra \p^N$ et à la suite exacte suivante:
\[
0\lra \mathcal{O}_{\p^N}(m)\lra \mathcal{O}_{\p^N}(m+1)\lra i_\ast \mathcal{O}_{\p^N}(m)\lra 0.
\]
Par récurrence, on établit que:
\[
\begin{split}
-\log h_{Q,{\footnotesize \overline{T\p^1}_{FS}; \overline{\mathcal{O}(m)}_{FS}}}=\frac{m^2}{2}+m+\frac{1}{2}-4\zeta'_\Q(-1)\quad
\forall \, m\in \N.
\end{split}
\]

 %$-\log h_{Q,{\footnotesize \overline{T\p^1}_{FS}; \overline{\mathcal{O}(m)}_{FS}}}
 %$ en utilisant le théorème de Riemann-Roch arithmétique, cf. \cite{ARR}. On a
%\[
%-\log h_{Q,{\footnotesize \overline{T\p^1}_{FS}; \overline{\mathcal{O}(m)}_{FS}}}=\pi_\ast\bigl(\widehat{ch}(\overline{\mathcal{O}(m)}_{FS})\widehat{\mathrm{Td}}(\overline{T\p^1}_{FS})\bigr)-a\bigl(\pi_\ast(ch(\mathcal{O}(m))\mathrm{Td}(T\p^1)R(T\p^1))\bigr)
%\]
%avec $\pi:\p^1_{\Z}\lra \mathrm{Spec}(\Z)$ et $\pi_\ast:\widehat{CH}^2(\p^1_\Z)\mapsto \R$.\\

%Si note par $\widehat{x}:=c_1(\overline{\mathcal{O}(1)}_{FS})$, alors
%\[
%\begin{split}
%-\log h_{Q,{\footnotesize \overline{T\p^1}_{FS}; \overline{\mathcal{O}(m)}_{FS}}}&= \pi_{\ast}\bigl((1+m\widehat{x}+\frac{m^2}{2}\widehat{x}^2)(1+\widehat{x}+\frac{1}{3}\widehat{x}^2) \bigr)-(4\zeta'_\Q(-1)-\frac{1}{6})\\
%&=\frac{m^2}{2}+m+\frac{1}{2}-4\zeta'_\Q(-1).
%\end{split}
%\]
%avec $\pi_\ast(\widehat{x}^2)=1$, voir \cite[p. 212]{Character2}.\\

 Donc, en utilisant \eqref{diffFSCAN}, on obtient:
 \[
\begin{split}
-\log h_{Q,{\footnotesize \overline{T\p^1}_{\infty}; \overline{\mathcal{O}(m)}_{\infty}}}&=-\Bigl(\frac{m^2}{2}+m+ \frac{1}{3}\Bigr)+\frac{m^2}{2}+m+\frac{1}{2}-4\zeta'_\Q(-1)\\
&=\frac{1}{6}-4\zeta'_\Q(-1).
\end{split}
\]

Si l'on note par $h_{L^2,\infty, \infty}$ le $L^2$-norme naturelle sur $\la(\mathcal{O}(m))$  définie par les métriques canoniques de $T\p^1$ et $\mathcal{O}(m)$, on a
\begin{equation}\label{volumenotation}
h_{L^2,\infty, \infty}=\prod_{k=0}^m \bigl(z^k,z^k\bigr)_{\infty,\infty},
\end{equation}
où $\bigl(z^k,z^k\bigr)_{\infty,\infty}=\int_{\p^1}h_\infty(z^k,z^k)\omega_\infty $. On a pour tout $0\leq k\leq m$,
\[
\begin{split}
(z^k,z^k)_{\infty,\infty}&=\int_{\p^1}h_\infty(z^k,z^k)\omega_\infty \\
&=\int_\C \frac{|z|^{2k}}{\max(1,|z|)^{2m}}\frac{dz\wedge d\z}{2\pi i \max(1,|z|^4)}\\
&=\int_{0}^\infty \frac{r^{2k}}{\max(1,r)^{2m+4}}2rdr\\
&=\int_0^\infty \frac{r^k}{\max(1,r)^{m+2}}dr\\
&=\int_{0}^1 {r^k}dr+\int_1^\infty \frac{r^k}{r^{m+2}}dr\\
&=\frac{1}{k+1}+\frac{1}{m+1-k}\\
&=\frac{m+2}{(k+1)(m+1-k)}.\\
\end{split}
\]
Par suite,
\begin{equation}\label{notevolume}
h_{L^2,\infty, \infty}=\frac{\,{}(m+2)^{{m+1}}}{\bigl((m+1)!\bigr)^2}.
\end{equation}

On obtient donc,
\[
T_g\bigl(\overline{T\p^1}_\infty;\overline{\mathcal{O}(m)}_\infty\bigr)=4\zeta'_\Q(-1)-\frac{1}{6}-\log \frac{\,{}(m+2)^{{m+1}}}{((m+1)!)^2}\quad \forall \,m\in \N.
\]
\end{proof}

\section{Régularisation de familles de fonctions Zêta}\label{partie222}
Nos principaux résultats dans cette section sont    les théorèmes \eqref{repzeta} et \eqref{regularisation}. Dans le
premier nous démontrons  un  résultat proche de \cite[lemme. 1]{Spreafico}. Précisément, nous allons établir un
énoncé plus précis et plus général   donnant une condition suffisante pour qu'une fonction Zêta admette une
représentation intégrale. Le théorème \eqref{regularisation} est un résultat concernant la notion  de famille de
fonctions Zêta régularisable qu'on  développera  dans la suite.

\subsection{Sur la régularisation  des produits infinis d'après Voros}
%\textbf{Rappelons que notre cas sort du cadre classique. Dans \cite[§. 7]{Watson} on trouve une introduction a l'étude des developpements asymptotiques des fonctions de Bessel}.\\

Dans cette première partie on rappelle les principaux résultats de Voros autour de la notion du produit régularisé.

Soit $\Lambda:0<\al_1\leq \al_2\leq \ldots $ une suite croissante non bornée de nombres réels positifs. La fonction \textit{Thêta}  attachée à cette suite est définie comme suit:

\[
 \theta_\Lambda(t):=\sum_{n\geq 1} e^{-\al_nt}.
\]
On suppose que $\theta_\Lambda$ vérifie les deux hypothèses suivantes:
\begin{enumerate}
 \item[\textbf{$\Theta $}1:] $\theta_\Lambda(t)$ converges pour tout $t>0$.
\item[\textbf{$\Theta $}2:] $\theta_{\Lambda}$ admet, pour $t$ assez petit, un développement en série de Laurent:
\[
\theta_\Lambda(t)=\sum_{n=0}^{+\infty} c_{i_n} t^{i_n}, \quad\text{quand}\; t \mapsto 0
\]
avec $\{i_n\}$ une suite croissante non bornée de nombres réels, et que $i_0<0$. On a nécessairement $c_{i_0}\neq 0$.\\
\end{enumerate}

alors la fonction zêta associée à la suite $\Lambda$:
\[
 \zeta_{\Lambda}(s)=\sum_{n\geq 1}\frac{1}{\al_n^s},
\]
vérifie
\[
 \zeta_\Lambda(s)=\frac{1}{\Gamma(s)}\int_{0}^\infty \theta_{\Lambda}(t)t^{s-1}dt
\]
converge pour $\mathrm{Re}(s)>-i_0$ avec un pôle en $s=-i_0$, admet une continuation méromorphe au plan complexe entier qui est sans pôle en $s=0$, avec $\zeta(0)=c_{0}$, voir par exemple \cite[Théorème 1]{Soulé}. Dans ce cas, on pose la définition suivante:
\begin{definition}
Sous ces hypothèses, on appelle produit régularisé et on le note $\al_1\al_2\ldots$ la quantité suivante: \[
 e^{-\zeta_\Lambda'(0)}.
\]
\end{definition}
Si $\al$ est un réel négatif, alors la suite $0<-\al+\al_1\leq -\al +\al_2\leq \ldots$, vérifie $\Theta 1$ et  $\Theta 2$, donc on peut considérer le produit régularisé $(\al_1-\al)(\al_2-\al)\cdots$ comme étant $D(\al):=\exp(-Z'(0,-\al))$,  où $Z(s,-\al):=\sum_{k=0}^\infty \frac{1}{(-\al+\la_k)^s}$.
$D$ fonction peut être définie pour tout nombre complexe, d'après \cite{Voros} c'est l'unique fonction holomorphe en la variable $\al$ ayant $\al_1, \al_2,\ldots$ pour uniques zéros (comptés avec multiplicités), bornée par $\exp(a+b|\al|^N )$ où $a$, $b$ et $N$ des constantes. On montre dans \cite[p.448]{Voros} que $D$  admet au voisinage de $-\infty$, un  développement asymptotique de la forme:
\begin{equation}\label{developpementvoros}
-\log D(\al)\sim_{\al \rightarrow -\infty}\sum_{i_n\neq -m} c_{i_n}\Gamma(i_n)(-\al)^{-i_n}-\sum_{m=0}^{[\mu]}c_{-m}\Bigl(\log(-\al)-\sum_{r=1}^m r^{-1} \Bigr)\frac{\al^m}{m!}\footnote{$[\cdot]$ désigne la partie entière. }
\end{equation}
$\mu$ est, par définition, le réel positif $-i_0$.\\

Supposons que $\mu\neq 1$. Par le théorème de factorisation de Weierstrass, voir par exemple \cite[p. 195]{Hille2}, la fonction holomorphe $D$ se factorise en:
\[
D(\al)=e^{P(\al)}\Delta(\al)
\]
où $P$ est un polynôme de degré $\leq [\mu]$, et $\Delta$ est de la forme
\[
\Delta(\al)=\prod_{k=1}^\infty (1-\frac{\al}{\al_k})\exp(\frac{\al}{\al_k }+\frac{\al^2}{2\al_k}+\ldots+\frac{\al^{[\mu]}}{[\mu]\al_{[\mu]}} ), \quad \text{si}\quad \mu>1
\]
et

\[
\Delta(\al)=\prod_{k=1}^\infty (1-\frac{\al}{\al_k}), \quad \text{si}\quad \mu<1.
\]
L'expression  de $P$ est donnée par  \cite[(4.12)]{Voros}.\\

Le développement asymptotique ci-dessus est une conséquence de la régularisation des zéros de la fonction considérée,
ce qui suppose à priori qu'on dispose d'informations sur ces zéros c'est à dire connaître le comportement de la
fonction $\theta$ au voisinage de zéro. Mais, il existe certaines fonctions analytiques qui possèdent un
développement asymptotique du même type mais obtenus par une autre procédure, par exemple ${\Gamma}^{-1}$, l'inverse
de la fonction Gamma et  son analogue la  $G$-fonction de Barnes, se  déduit de la formule de Binet pour la première
fonction, voir \cite[12.31]{Whittaker} et de l'expression intégrale de $\log \Gamma$, voir \cite[§ 14]{Barnes}. Une
autre classe d'exemples est donnée par les   fonctions de Bessel, qui ont  la particularité d'avoir des zéros non
explicites. Hankel utilise dans \cite{Hankel} une représentation intégrale pour en extraire un développement
asymptotique.\\

Soit   $\Delta$   une fonction analytique non constante sur $\C$ dont  tous ses zéros non nuls sont positifs. On
suppose de plus que $\Delta$ admet un développement asymptotique au voisinage de $-\infty$ de la forme suivante:
\begin{equation}\label{formuleasymptotique}
 -\log \Delta(\al)\sim_{\al\mapsto -\infty}\sum_{j\in \mathcal{J} }c_j(-\al)^{-j}\log(-\al)^{k_j}
+a\log(-\al)+b+d\al,
\end{equation}
où $\mathcal{J}$ est un  ensemble discret minoré, on pose $j_0=\inf \mathcal{J}$ qu on suppose $<0$,  $a,b,d, c_j$
sont des réels et $k_j\in\{0,1\}$ pour tout $j\in \mathcal{J}$.\\

Si l'on se donne $p$, une fonction analytique ayant que des zéros positifs non nuls qu'on
note par $a_1\leq a_2\leq \ldots$ et on suppose que $p$ admet un développement asymptotique de la forme:
\[
\log p(z)=d z^\kappa+a\log z^2+b+\frac{b_{-1}}{z}+\frac{b_{-2}}{z^2}+O\bigl(\frac{1}{z^3}\bigr),\quad \forall
|z|\gg1,\; 0<|\arg(z)|<\pi.
\]
Alors on peut se demander quelle est la relation entre le développement asymptotique de $p$ et celui du  $\Delta$?
 Le résultat suivant décrit cette relation et étend en même temps  \cite[corollaire 1]{Spreafico}.

\begin{theorem}\label{repzeta}
Soit $(a_j)_{j\geq  1}$ une suite strictement croissante de réels. On suppose que
\[
\sum_{j=1}^\infty \frac{1}{a_j^2}<\infty,
\]
alors la fonction $p$ définie par:
\[
p(z)=\prod_{j=1}^\infty \Bigl(1-\frac{z^2}{a_j^2} \Bigr)\quad \forall z\in \C,
\]
est analytique sur $\C$. Si l'on suppose en plus que
\[
\log p(z)=d z^\kappa+a\log z^2+b+\frac{b_{-1}}{z}+\frac{b_{-2}}{z^2}+O\Bigl(\frac{1}{z^3}\Bigr)\quad \forall\,
|z|\gg 1,\; 0<|\arg(z)|<\pi.
\]
avec $\kappa\in \N$, $d,a,b,b_{-1} $ et $b_{-2}$ sont des réels.
Alors,
\[
\zeta(s):=\sum_{j=1}^\infty \frac{1}{a_j^{2s}}=\frac{s^2}{\Gamma(s+1)}\int_0^\infty\frac{t^{s-1}}{2\pi
i}\int_{\Lambda_{c}} \frac{e^{-z^2t}}{z}\log p(z)dz,
\]
pour tout $s\in \C$ tel que $\mathrm{Re}(s)>\frac{1}{2}$, avec un pôle en $s=1$, elle  est prolongeable
analytiquement au voisinage de $s=0$  et on a:
\begin{equation}\label{valeurszeta}
 \zeta(0)=-a, \; \zeta'(0)=-b.
\end{equation}

où $\Lambda_c=\Bigl\{z\in \C\,|\,|\arg(z-c)|=\frac{\pi}{8} \Bigr\}$.
\end{theorem}
\begin{proof}
Supposons que
\[
\sum_{j=1}^\infty \frac{1}{a_j^2}<\infty,
\]
alors par le théorème de factorisation de Weierstrass, la fonction $z\mapsto
\prod_{j=1}^\infty\bigl(1-\frac{z}{\al_j^2} \bigr)$ est analytique sur $\C$ et par conséquent $p$ est analytique sur
$\C$.\\

On vérifie par l'absurde que
\[
 \inf_{z\in \Lambda_c,\,k\in \N^\ast}\frac{|\al_k^2-z^2|}{\al_k^2}\neq 0.
\]
On en déduit que la suite $\Bigl(\sum_{k=1}^n\frac{1}{z^2-\al_k^2}\Bigr)_{n\in \N^\ast}$ converge normalement sur un voisinage ouvert de $\Lambda_c$, donc
\[
 \int_{\Lambda_c}e^{-z^2t}\sum_{k=1}^\infty\frac{2z}{z^2-a_k^2}dz=
\sum_{k=1}^\infty \int_{\Lambda_c}e^{-z^2t}\frac{2z}{z^2-a_k^2}dz=\sum_{k=1}^\infty e^{-a_k^2t}.
\]
Soit $R>0$ et on pose
$\Lambda_{c,R}=\Bigl\{z\in \C\,\bigl|\,|\arg(z-c)|=\frac{\pi}{8},\;|z|\leq R \Bigr\}$. On a

\begin{align*}
 \int_{\Lambda_{c,R}}e^{-z^2t}\sum_{k=1}^\infty\frac{2z}{z^2-\al_k^2}dz&=
 \int_{\Lambda_{c,R}}e^{-z^2t}\frac{d}{dz}\log p(z)dz\\
&=\Bigl[-e^{-z^2 t}\log p(z) \Bigr]_{z_R}^{z'_R}+\int_{\Delta_{c,R}}zt e^{-z^2t}\log p(z)dz
\end{align*}
où $\bigl\{z_R,z'_R\bigr\}=\Lambda_{c,R}\cap\bigl\{|z|=R\bigr\}$ choisis tels que $\mathrm{Im}(z'_R)>0$.\\

On déduit à l'aide du développement asymptotique de $\log p$ que
\[
 \sum_{k=1}^\infty e^{-\al_k^2t}=\lim_{R\mapsto \infty}\int_{\Lambda_{c,R}}zt e^{-z^2t}\log
p(z)dz=\int_{\Lambda_{c}}zt e^{-z^2t}\log p(z)dz.
\]
La suite de la preuve sera une application du théorème \eqref{regularisation}. En effet, si considère la famille
$(G_\nu)_{\nu\in \N^\ast}$ définie comme suit: $G_1(z)=\prod_{k=1}^\infty\bigl(1+\frac{z^2}{\al_{1,k}^2}
\bigr)=p(iz)$ et $G_\nu(z)=\prod_{k=1}^\infty\bigl(1+\frac{z^2}{\al_{\nu,k}^2} \bigr)$ avec la convention
$\al_{\nu,k}=\infty$ si $\nu\geq 2$ et $k\geq 1$, alors cette famille  est régularisable, voir la définition
\eqref{régularisable}.

\end{proof}

\begin{remarque}
\rm{On note par $\mathcal{P}$ l'ensemble des nombres premiers. On pose $\Delta(z):=\prod_{p\in \mathcal{P} }(1-\frac{z}{p^t})$ $\forall z\in \C$, avec $t>2$. Par le théorème de factorisation de Weierstrass, $\Delta$ est une fonction analytique sur $\C$, remarquons  que $\Delta(1)=\zeta_\Q(t)^{-1}$. Si par l'absurde cette fonction admet un développement asymptotique du type  \eqref{formuleasymptotique}, alors la fonction Zêta associée:
\[
\zeta_t(s):=\sum_{p\in \mathcal{P}}\frac{1}{p^{ts}}
\]
admet un prolongement analytique en $s=0$, ce qui n'est pas le cas, puisqu'on montre dans \cite{Landau} que l'axe imaginaire est une frontière naturelle pour la fonction Zêta $\zeta(s):=\sum_{p\in \mathcal{P}}\frac{1}{p^{s}}$.}
\end{remarque}

\subsection{Régularisation de familles de fonctions Zêta }\label{paragrapheG}
Dans ce paragraphe nous allons étendre une méthode de régularisation de familles de fonctions Zêta, introduite en premier dans  \cite{Spreafico}.\\

On considère une famille $\bigl(G_\nu\bigr)_{\nu\in \Omega }$ une famille de fonctions analytiques sur $\C$ indexé par $\Omega$, un sous ensemble infini de $\R$ discret ou connexe. Par exemple solutions d'une équation différentielle dont les coefficients dépendent de $\nu$, comme l'équation de Legendre:
\[
\frac{d}{dz}\Bigl[(1-z^2)\frac{dG_\nu}{dz}  \Bigr]+\nu(\nu+1)G_\nu=0,
\]
Un autre exemple, celui des fonctions de Bessel modifiées, solutions de l'équation:
\[
\frac{d^2 G_\nu}{dz^2}+\frac{1}{z}\frac{dG_\nu}{dz}-\Bigl(1+\frac{\nu^2}{z^2}\Bigr)G_\nu=0.
\]
%\textbf{Déduire $T(X_1\times X_2)$; la torsion analytique pour produit de variétés muni de la métrique produit, dont on dispose d'une formule voir \cite{RaySinger} (mais lorsque tous est $\cl$!)}.\\
Pour les applications, on va se restreindre aux fonctions  $G_\nu$ qui s'écrivent sous la forme
\[
G_\nu(z)=l_\nu z^{e_\nu}\prod_{k=1}^\infty\bigl(1+\frac{z^2}{\la_{\nu,k}^2} \bigr)\quad \forall\, \nu\in \Omega,
\]
avec $l_\nu,e_\nu\in \R$  et $\la_{\nu,k}\in \R, \forall \, k\in \N_{\geq 1}$ ordonnés comme suit: $\la_{\nu,1}^2\leq \la_{\nu,2}^2\leq \ldots$\\

 On pose
\[
p_\nu(z):=\frac{G_\nu(iz)}{l_\nu z^{e_\nu}}=\prod_{k\geq 1}^\infty\Bigl(1-\frac{z^2}{\la_{\nu,k}^2} \Bigr)\quad \forall\, \nu\in \Omega,
\]
On dispose donc d'une famille de fonctions zêta indexée par $\Omega$:
\[\zeta_\nu(s):=\sum_{k\geq 1}\frac{1}{\la_{\nu,k}^{2s}}\quad\forall \, \mathrm{Re}(s)>1,\; \nu\in \Omega.\]

On souhaite étudier la fonction suivante:
\[
\zeta_\Omega(s):=\sum_{\nu\in \Omega}\sum_{k\geq 1}\frac{1}{\la_{\nu,k}^{2s}}=\sum_{\nu\in \Omega}\zeta_\nu(s),
\]
lorsque $\Omega$ est discret et
\[
\zeta_{\Omega}(s)=\int_{\Omega} \zeta_\nu (s)d\nu,
\]
si $\Omega$ est un intervalle.\\

 Il est donc naturel d'imposer  des conditions qui prennent en compte le nouveau paramètre $\nu$, afin d'assurer que $\zeta_\Omega$ possède les propriétés communes à toute fonction Zêta, à savoir convergence pour $\mathrm{Re}(s)\gg1$, prolongement  holomorphe au voisinage de zéro...\\

  On introduit alors la définition suivante:

\begin{definition}\label{régularisable} On dit que la famille $\bigl\{\zeta_\nu,\, \nu\in \Omega \bigr\}$ est régularisable si:

\begin{enumerate}

\item \begin{equation}\label{lowerbound}
c:=\inf_{\nu\in \Omega\cap [1,\infty[}\biggl(\frac{\la_{\nu,1}^2}{\nu^2} \biggr)\neq 0.
\end{equation}

\item Une représentation intégrale de $\zeta_\nu$, $\forall \, \nu\in \Omega$:
\[
\zeta_\nu(s)=\frac{s^2}{\Gamma(s+1)}\int_0^\infty t^{s-1}\int_{\Lambda_{c_\nu}}\frac{e^{-z^2t}}{z}\log  \Bigl(\frac{G_\nu(iz)}{l_\nu z^{e_\nu}}\Bigr)dz,\quad \forall \nu\in \Omega,
\]
avec $c_\nu$ est un réel dans  $]0,\la_{\nu,1}[$.

\item Pour tout $\nu\in \Omega$, il existe $\widetilde{d}_\nu, \widetilde{a}_\nu, \widetilde{b}_\nu$ et $\widetilde{c}_\nu$ des réels tels que
\[
-\log \biggl(\frac{G_\nu(z)}{l_\nu z^{e_\nu}}\biggr)=\widetilde{d}_\nu z+\widetilde{a}_\nu\log (z^2)+\widetilde{b}_\nu+\frac{\widetilde{c}_\nu}{z}+ O\Bigl(\frac{1}{z^2}\Bigr),
\]
pour $|z|\gg1$ et $|\arg(z)|<\frac{\pi}{2}$ \footnote{ Remarquons que les zéros de $G_\nu$ sont situés sur l axe des imaginaires purs}. On suppose en plus que $a_\nu$  est  un polynôme en   $\nu$  pour $\nu\gg1$.\footnote {On a pris une détermination holomorphe du logarithme; en effet: Si $z^2\in -\R$, alors $\arg(z)=\pm \frac{\pi}{2}[\pi]$}.

\item  Il existe  $r$, $w$, $\rho$ et $\eta(\nu,\cdot)$ quatre  fonctions continues sur $|\arg(z)|<\frac{\pi}{2}$ telles que
\[
G_\nu\bigl(\nu z\bigr)=\al_\nu\,  \rho(z)\,\exp\bigl(e'_\nu r(z)\bigr)\biggl(1+\frac{1}{\nu}w(z)+\frac{1}{\nu^2}\eta(\nu,z)  \biggr)\quad \nu\gg1,
\]
avec $\al_\nu, e'_\nu\in \R$.
\begin{enumerate}
\item $\rho(0)=1$ et $\log \rho(z)=\al z+\beta\log (z)+\mu+O(\frac{1}{z})$ pour $|z|\gg1$ avec $\al$, $\beta$ et $\mu$ sont deux réels.
\item $w$ est holomorphe sur $\C\setminus\Bigl(\bigl\{ ict, t\geq 1  \bigr\}\cup \bigl\{ -ict, t\geq 1  \bigr\}\Bigr)$ et $w(z)=O\bigl(\frac{1}{z^\kappa}\bigr)$  pour $|z|\gg1$ pour un certain $\kappa\geq 1$.
%\item Il existe $k\leq 2$ $r(z)=O(|z|^k)$ pour  $|z|\gg1$.
\item $r(z)=z+O\bigl(\frac{1}{z}\bigr)$ pour $|z|\gg1$ et $r(z)=\log (z)+o(1)$ pour $0<|z|\ll1$.\footnote{Par exemple si $r(z)=\sqrt{1+z^2}+\log \frac{z}{1+\sqrt{1+z^2}}$ alors $r(z)=z+O(\frac{1}{z})$ pour $|z|\gg1$.}
%\item $\log(|\al_n|)$ admet un développementasymptotique en $n$ et $\log (n)$.
\item $\eta(\nu,z)$ uniformément borné en $z$ pour $n\gg1$ et que $\eta(\nu,z)=O\bigl(\frac{1}{z}\bigr)$ quand $|z|\gg1$  pour $\nu$ fixé. \\
\end{enumerate}

\end{enumerate}

\end{definition}

Pour simplifier on suppose dans la suite  que $\Omega$ est discret, avec $+\infty$ comme unique point d'accumulation pour $\Omega\cup \{+\infty\}$.\\

Pour $s\in \C$ fixé, on pose
\begin{align*}
&A(s):=\sum_{\nu\in  \Omega}\frac{a_\nu}{\nu^{2s}},\quad  B(s):=\sum_{\nu \in \Omega}\frac{b_\nu}{\nu^{2s}}, \quad P(s):=\sum_{\nu\in \Omega}\frac{w(0)}{\nu^{2s+1}},\\
& \quad \text{et}\quad F(s):=\int_0^\infty \frac{t^{s-1}}{2\pi i}\int_{\Lambda_c}\frac{e^{-z^2}}{z}w(z)dz.
\end{align*}
où $a_\nu$ (resp. $b_\nu$) est par définition le coefficient de $\log(z^2)$ (resp. le terme constant) dans le développement asymptotique de $-\log \frac{G_\nu(\nu z)}{l_\nu (\nu z)^{e_\nu}}$, c'est à dire
\[
-\log \biggl(\frac{G_\nu(\nu z)}{l_\nu (\nu z)^{e_\nu}}\biggr)=\widetilde{d}_\nu \nu z+a_\nu\log (z^2)+b_\nu+\frac{\widetilde{c}_\nu}{\nu z}+ O(\frac{1}{z^2})\quad \forall\, |z|\gg1.
\]
 On vérifie que
\[
 a_\nu=\widetilde{a}_\nu,\quad b_\nu=\widetilde{b}_\nu+2\widetilde{a}_\nu \log\nu.
\]

\begin{lemma}\label{P-B}
On a
\[
P(s)-B(s)
\]
est fini en un  voisinage ouvert de $s=0$, en plus cette fonction est analytique au voisinage de $s=0$.
\end{lemma}
\begin{proof}
Soit $z$, tel que $|\arg(z)|<\frac{\pi}{2}$. Si $\nu\gg1$ alors on a par hypothèse,
\[
\log G_\nu(\nu z)=\log \al_\nu-\log \rho(z)+e'_\nu r(z)+\log\Bigl(1+\frac{w(z)}{\nu}+\frac{1}{\nu^2}\eta(\nu,z) \Bigr).
\]
On peut prendre $z$ grand, et en fixant  $\nu$
\begin{align*}
\log G_\nu(\nu z)&=\log \al_\nu-\al z-\beta\log (z)+O(\frac{1}{z})+e'_\nu \bigl(z+O(\frac{1}{z})\bigr)+\frac{w(z)}{\nu}+\frac{1}{\nu^2}\eta(\nu,z)\\
&= (e'_\nu -\al)z-\beta \log z+\log \al_\nu +O(\frac{1}{z}),
\end{align*}
En le comparant au développement asymptotique de $G_\nu$ pour $|z|\gg1$ c'est à dire à
\begin{align*}
\log G_\nu(\nu z)=-(\widetilde{d}_\nu \nu) z+\bigl(e_\nu-2a_\nu\bigr)\log z+\log (l_\nu \nu^{e_\nu})-b_\nu+O(\frac{1}{z}),
\end{align*}
on obtient pour $\nu\gg1$:
\begin{equation}\label{POLYNOMEEN}
e'_\nu -\al=-d_\nu \nu,\quad
-\beta =e_\nu-2a_\nu,\quad
\log \al_\nu =\log (l_\nu \nu^{e_\nu})-b_\nu.
\end{equation}

Au voisinage de $y=0$, on a
\[
\lim_{y\mapsto 0}\frac{G_\nu(\nu y)}{\exp(e'_\nu r(y))}=\lim_{y\mapsto 0}\frac{l_\nu \nu^{e_\nu}y^{e_\nu}+o(y^{e_\nu})}{y^{e'_\nu}}=\al_\nu\Bigl(1+\frac{w(0)}{\nu}+\frac{1}{\nu^2}\eta(\nu,0)\Bigr),
\]
on en déduit que pour $\nu\gg1$:
\begin{align*}
e_\nu&=e'_\nu,\\
\log (l_\nu \nu^{e_\nu})-\log\al_\nu&=\log\Bigl( 1+\frac{w(0)}{\nu}+\frac{1}{\nu^2}\eta(\nu,0)\Bigr).
\end{align*}
Par suite,
\[
b_\nu=\log\Bigl(1+\frac{w(0)}{\nu}+\frac{1}{\nu^2}\eta(\nu,0)\Bigr)=\frac{w(0)}{\nu}+O\bigl(\frac{1}{\nu^2}\bigr) \quad \forall \nu\gg1.
\]
donc,
\begin{equation}\label{a1a}
b_\nu-\frac{w(0)}{\nu}=O\Bigl(\frac{1}{\nu^2}\Bigr),
\end{equation}
alors la série $P-B$ converge normalement et cela pour tout $s$ assez petit. On conclut que $P-B$  est une fonction analytique au voisinage de $s=0$.\\

\end{proof}

Pour les applications, on va supposer   que $\Omega=\N^\ast$. Par hypothèses, $a_n$ est un polynôme en $n$  et d'après \eqref{POLYNOMEEN}, $e_n$ l'est aussi, soit $d$ son degré.
 On pose:
\[
\zeta_n(s)=\sum_{k\geq 1}\frac{1}{\la_{n,k}^{2s}}\;\forall \, n\in \N^\ast\quad\text{et}\quad \zeta(s)=\sum_{n\geq 1}\zeta_n(s)\quad s\in \C.
\]

\begin{theorem}\label{regularisation}
On considère $\bigl\{G_n,n\in \N^\ast\bigr\}$ une famille de fonctions analytiques non nulles sur $\C$ vérifiant les hypothèses de la définition précédente.

Alors pour tout $s$ tel que $\mathrm{Re}(s)>1$, la fonction suivante:
\[
T(z,s):=-\sum_{n\geq 1}\frac{1}{n^{2s}}\log \biggl(\frac{G_n(inz)}{l_n (n z)^{e_n}}\biggr).
\]
est analytique sur un voisinage ouvert de $\Lambda_c=\bigl\{\la \in\C\,\bigl|\, |\arg(\la-c)|=\frac{\pi}{8}\bigr\}$ et on a la représentation intégrale suivante pour $\zeta$:
\[
 \zeta(s)=\frac{s^2}{\Gamma(s+1)}\int_0^\infty \frac{t^{s-1}}{2\pi i}\int_{\Lambda_c}\frac{e^{-z^2 t}}{z} T(z,s)dz dt,
\]
qui converge pour $\mathrm{Re}(s)>\frac{d+1}{2}$, avec un pôle en $s=1$ et admet un prolongement holomorphe au voisinage de $s=0$, et
\[
\begin{split}
\zeta(s)&=\frac{s}{\Gamma(s+1)}\bigl[\gamma A(s)-\frac{1}{s}A(s)+P(s)-B(s) \bigr]
+\frac{s^2}{\Gamma(s+1)}\zeta_\Q(2s+1)F(s)\\
&+\frac{s^2}{\Gamma(s+1)}h(s).
\end{split}
\]
pour $|s|\ll1$ avec $h$ une fonction analytique au voisinage $s=0$.

\end{theorem}

\begin{proof}
Commençons par noter qu'on a formellement:
\[
T(z,s)=\log \Bigl(\prod_{n=1}^\infty \prod_{k=1}^\infty \bigl(1-\frac{n^2z^2}{\la_{n,k}^2}  \bigr)^{\frac{1}{n^{2s}}}  \Bigr).
\]
Par hypothèse, il existe $N\gg1$ tel que
\[
\begin{split}
T(z,s)&=-\sum_{ n=1}^{ N-1}\frac{1}{n^{2s}}\log \biggl(\frac{G_n(inz)}{l_n (inz)^{e_n}}\biggr)-\biggl(\sum_{n=N}^\infty \frac{\log\al_n-\log (l_n n^{e_n})}{n^{2s}}+\frac{\log \rho(iz)}{n^{2s}} +\frac{e_n r(iz)}{n^{2s}}+\frac{w(iz)}{n^{2s+1}}+\mathrm{O}_{z}\biggl(\frac{1}{n^{2+2s}}\biggr)\biggl)\\
&=-\sum_{1\leq n\leq N-1}\frac{1}{n^{2s}}\log \biggl(\frac{G_n(inz)}{l_n (inz)^{e_n}}\biggr)-\biggl(\sum_{n=N}^\infty -\frac{b_n}{n^{2s}}+\frac{1}{n^{2s}}\log \rho(iz) +\frac{e_n }{n^{2s}}r(iz)+\frac{w(iz)}{n^{2s+1}}+\mathrm{O}_{z}\biggl(\frac{1}{n^{2+2s}}\biggr)\biggr)\\
&=-\sum_{1\leq n\leq N-1}\frac{1}{n^{2s}}\log \biggl(\frac{G_n(inz)}{l_n (inz)^{e_n}}\biggr)-\biggl(\sum_{n=N}^\infty \Bigl(\frac{w(iz)}{n^{2s+1}}-\frac{b_n}{n^{2s}}\Bigr)+\frac{1}{n^{2s}}\log \rho(iz) +\frac{e_n }{n^{2s}}r(iz)+\mathrm{O}_{z}\biggl(\frac{1}{n^{2+2s}}\biggr)\biggl)\\
&=-\sum_{1\leq n\leq N-1}\frac{1}{n^{2s}}\log \biggl(\frac{G_n(inz)}{l_n (inz)^{e_n}}\biggr)-\biggl(\sum_{n=N}^\infty \Bigl(\frac{w(iz)-w(0)}{n^{2s+1}}+O(\frac{1}{n^{2s+2}})\Bigr)+\frac{1}{n^{2s}}\log \rho(iz) +\frac{e_n }{n^{2s}}r(iz)\\
&+\mathrm{O}_{z}\biggl(\frac{1}{n^{2+2s}}\biggr)\biggl)\quad \text{par}\;\eqref{a1a}.\\
\end{split}
\]
(où $\mathrm{O}_{z}(\frac{1}{n^{2+2s}})$ est une fonction bornée en $z$ uniformément en $n$). Ce qui donne que
\begin{equation}\label{Tp}
T(z,s)+\sum_{1\leq n\leq N-1}\frac{1}{n^{2s}}\log \biggl( \frac{G_n(inz)}{l_n z^{e_n}}\biggr)=-\sum_{n=N}^\infty
\Biggl(\frac{w(iz)-w(0)}{n^{2s+1}}+\frac{1}{n^{2s}}\log \rho(iz) +\frac{e_n
}{n^{2s}}r(iz)+\frac{w(iz)}{n^{2s}}+\mathrm{O}_{z}\biggl(\frac{1}{n^{2+2s}}\biggr)\Biggr),
\end{equation}
pour tout $z$ dans un voisinage ouvert  de $\Lambda_c$. Comme la suite $ \bigl(\frac{e_n}{n^{2s}}\bigr)_{n\in \N^\ast}$ est
équivalente à $\bigl(\frac{1}{n^{2s-d}}\bigr)_{n\in \N^\ast}$, on déduit que sur tout ensemble borné en $z$,  $T(z,s)$ converge
normalement pour  tout $\mathrm{Re}(s)>\frac{d+1}{2}$.\\

On a pour tout $n\in \N^\ast$, $\zeta_n$ admet une représentation intégrale (voir \eqref{repzeta}):
\[
\zeta_n(s)=\frac{s^2}{\Gamma(s+1)}\int_0^\infty t^{s-1}\int_{\Lambda_{c_n}}\frac{e^{-z^2t}}{z}\log \frac{G_n(iz)}{l_n z^{e_n}}dzdt\quad \forall\, n\in \N^\ast,
\]
où $c_n$ est un réel  quelconque dans $]0, \la_{n,1}[$. On peut écrire
 \begin{align*}
\zeta_n(s)&=\frac{s^2}{\Gamma(s+1)}\int_0^\infty t^{s-1}\int_{\frac{1}{n}\Lambda_{c_n}}\frac{e^{-n^2z^2t}}{z}\log
\frac{G_n(inz)}{l_n z^{e_n}}dzdt\\ &=\frac{1}{n^{2s}}\frac{s^2}{\Gamma(s+1)}\int_0^\infty t^{s-1}\int_{\frac{1}{n}\Lambda_{c_n}}\frac{e^{-z^2t}}{z}\log
\frac{G_n(inz)}{l_n z^{e_n}}dzdt.
 \end{align*}
En utilisant l'holomorphie de $\log p_n$ loin de l'axe réel (voir \eqref{repzeta}), on obtient:
 \begin{align*}
 \sum_{k\geq 1}\frac{1}{\la_{n,k}^{2s}}
 &=\frac{s^2}{\Gamma(s+1)}\int_0^\infty t^{s-1}\int_{\Lambda_{c_n}}\frac{e^{-z^2t}}{z}\frac{1}{n^{2s}}\log \frac{G_n(inz)}{l_n z^{e_n}}dzdt,
 \end{align*}
(rappelons que $\Lambda_{c_n}=\bigl\{z\,|\, |\arg(z-c_n)|=\frac{\pi}{8} \bigr\}$).\\

 D'après \eqref{lowerbound}, on a
 \begin{equation}\label{zetaNNN}
 \sum_{k\geq 1}\frac{1}{\la_{n,k}^{2s}}
 =\frac{s^2}{\Gamma(s+1)}\int_0^\infty t^{s-1}\int_{\Lambda_{c}}\frac{e^{-z^2t}}{z}\frac{1}{n^{2s}}\log \frac{G_n(inz)}{l_n z^{e_n}}dzdt,\quad \forall n\in \N^\ast.
 \end{equation}

On se propose de montrer que $\zeta$ admet une représentation intégrale en fonction de $T$. Vue l'expression \eqref{Tp}, il suffit d'étudier les fonctions en $s$ suivantes:

\begin{align*}
 &\int_0^\infty t^{s-1}\int_{\Lambda_c}\Bigl|\frac{e^{-z^2 t}}{z}r(iz)dz\Bigr|dt,\\
&\int_0^\infty t^{s-1}\int_{\Lambda_c}\Bigl|\frac{e^{-z^2 t}}{z}\log \rho(iz)dz\Bigr|dt,\\
&\int_0^\infty t^{s-1}\int_{\Lambda_c}\Bigl|\frac{e^{-z^2 t}}{z}w(iz)dz\Bigr|dt,\\
&\int_0^\infty t^{s-1}\int_{\Lambda_c}\Bigl|\frac{e^{-z^2 t}}{z}\mathrm{O}_{z}(\frac{1}{n^{2}})dz\Bigr|dt.\\
\end{align*}
La dernière fonction est bornée pour tout $\mathrm{Re}(s)\geq \frac{d+1}{2}$, puisque par hypothèse $\mathrm{O}_{z}(\frac{1}{n^2} )$ est bornée en $z$ uniformément en $n\gg1$. Notons il suffit d'étudier la première fonction. En effet, par hypothèses, on peut trouver $c_1$ et $c_2$ deux constantes telles que
\[
 |w(z)|\leq c_1|r(z)|,\quad |\log\rho(z)|\leq c_2 |r(z)|\quad \forall |z|\gg1.
\]
On note par $\Lambda_{c,\geq R}=\Lambda_{c}\cap\bigl\{|z|\geq R \bigr\}$ et $\Lambda_{c,\leq R}=\Lambda_{c}\cap\bigl\{|z|\leq R \bigr\}$ et  soit $s>\frac{1}{2}$ fixé. \\

Rappelons que  $r(z)=O(|z|)$ pour $|z|\gg 1$, alors pour
  $R\gg1$ il existe $C>0$ tel que
{\allowdisplaybreaks
\begin{align*}
\int_{0}^1 t^{s-1}\int_{\Lambda_{c,\geq R}}\Bigl| \frac{e^{-tz^2  }}{z}r(iz)dz dt\Bigr|&\leq
C\int_{0}^1 t^{s-1}\int_{v\geq R} e^{-t\mu v^2}  dvdt\\
&=C\int_{0}^1 t^{s-1-\frac{1}{2}}\int_{v\geq \sqrt{t}R} e^{-\mu v^2}  dvdt\\
&\leq C\int_{0}^1 t^{s-1-\frac{1}{2}}\int_{0}^\infty  e^{-\mu v^2}dvdt \\
&\leq \frac{C'}{s-\frac{1}{2}},\quad\text{où}\; C'\in \R,
\end{align*}}
et
\begin{align*}
\int_1^\infty \int_{\Lambda_{c,\geq R}}t^{s-1}\Bigl|\frac{e^{-tz^2 }}{z}r(iz)dz\Bigr| dt&\leq \int_{v\geq R}\int_1^\infty t^{s-1}e^{-t\mu v^2}dv dt\\
&=\int_{v\geq R}\frac{1}{\mu^{2s} v^{2s}}\int_{\mu v^2}^\infty t^{s-1}e^{-t}dt dv\\
&\leq \frac{\Gamma(s)}{s-\frac{1}{2}}\frac{1}{R^{2s-1}},
\end{align*}
par suite,
\[
\int_0^\infty \int_{\Lambda_{c,\geq R}}t^{s-1}\frac{e^{-tz^2 }}{z}r(iz)dz dt=O\biggl(\frac{1}{\mathrm{Re}(s)-\frac{1}{2}}\biggr)+O\biggl(\frac{\Gamma(\mathrm{Re}(s))}{\mathrm{Re}(s)-\frac{1}{2}}\frac{1}{R^{2\mathrm{Re}(s)-1}} \biggr)\quad \forall\, \mathrm{Re}(s)>\frac{1}{2}.
\]

Sur $\Lambda_{c,\leq R}$, la fonction $r$ est bornée puisque continue, donc
\begin{align*}
 \int_{0}^\infty t^{s-1}\int_{\Lambda_{c,\leq R}}\Bigl| \frac{e^{-tz^2  }}{z}r(iz)dz dt\Bigr|&\leq C\int_0^\infty t^{s-1}\int_{c\leq v\leq R} e^{-t\mu  v^2 }dv dt \\
&=C(R-c)\int_0^\infty t^{s-1} e^{-t\mu  c^2 } dt \\
&=\frac{C(R-c)}{(\mu c)^s}\Gamma(s).
\end{align*}
On conclut
\begin{align*}
\int_{0}^\infty t^{s-1}\int_{\Lambda_c}\Bigl| \frac{e^{-tz^2  }}{z}r(iz)dz dt\Bigr|=O\biggl(\frac{\Gamma(\mathrm{Re}(s))}{(\mu c)^{\mathrm{Re}(s)}}\biggr)+O\biggl(\frac{1}{\mathrm{Re}(s)-\frac{1}{2}}\biggr)\quad \forall \, \mathrm{Re}(s)>\frac{1}{2}.
\end{align*}

On montre les mêmes inégalités pour $w$ et $\log \rho$.\\

Rappelons que $d$ est le degré du polynôme $e_n$.
On conclut que pour tout $0<\tau\ll 1$ et $U$  un  ouvert borné en  $s$ vérifiant $\mathrm{Re}(s)>\frac{d+1}{2}+\tau, \forall s\in U$, il existe  une constante $M$ telle que $\forall s\in U$
{\allowdisplaybreaks
\begin{align*}
&\biggl|\int_{0}^\infty \frac{t^{s-1}}{2\pi i}\int_{ \Lambda_c} \frac{e^{-z^2 t}}{-z} \biggl(\sum_{n=N}^\infty \frac{w(iz)-w(0)}{n^{2s+1}}+\frac{1}{n^{2s}}\log \rho(iz) +\frac{e_n r_n(iz)}{n^{2s}}+\frac{w(iz)}{n^{2s}}+\mathrm{O}_{z}\Bigl(\frac{1}{n^{2+2s}}\Bigr)\biggl)dz dt\biggr|\leq\\
&\int_{0}^\infty \frac{t^{\mathrm{Re}(s)-1}}{2\pi i}\int_{ \Lambda_c}\biggl| \frac{e^{-z^2 t}}{-z} \biggl(\sum_{n=N}^\infty  \frac{w(iz)-w(0))}{n^{2\mathrm{Re}(s)+1}}+\frac{1}{n^{2\mathrm{Re}(s)}}\log \rho(iz)+\frac{e_nr_n(iz)}{n^{2\mathrm{Re}(s)}}+\frac{w(iz)}{n^{2\mathrm{Re}(s)}}+\mathrm{O}_{\sqrt{z}}\Bigl(\frac{1}{n^{2+2\mathrm{Re}(s)}}\Bigr)\biggr)dz\biggr| dt\leq\\
&M\biggl(\sum_{n\geq N}\frac{1}{n^{2\mathrm{Re}(s)+1}}+ \frac{1}{n^{2\mathrm{Re}(s)-d}}+\frac{1}{n^{2\mathrm{Re}(s)}}+O\Bigl(\frac{1}{n^{2\mathrm{Re}(s)+2}}\Bigr)\biggr).
\end{align*}}
A partir de cela et de \eqref{Tp}, on obtient
\begin{align*}
&\biggl|\frac{s^2}{\Gamma(s+1)}\int_{0}^\infty \frac{t^{s-1}}{2\pi i}\int_{ \Lambda_c} \frac{e^{-z^2 t}}{-z} \Bigl(T(z,s)+\sum_{n=1}^{N-1}\frac{1}{n^{2s}}\log \frac{G_n(inz)}{l_n z^{e_n}}\Bigr)dz dt\biggr| \\
&=\biggl|\frac{s^2}{\Gamma(s+1)}\int_{0}^\infty \frac{t^{s-1}}{2\pi i}\int_{ \Lambda_c} \frac{e^{-z^2 t}}{-z}T(z,s)dzdt+\sum_{n=1}^{N-1}\frac{s^2}{\Gamma(s+1)}\int_{0}^\infty \frac{t^{s-1}}{2\pi i}\int_{ \Lambda_c} \frac{e^{-z^2 t}}{-z}\frac{1}{n^{2s}}\log \frac{G_n(inz)}{l_n z^{e_n}}dz dt\biggr| \\
&=\biggl|\frac{s^2}{\Gamma(s+1)}\int_{0}^\infty \frac{t^{s-1}}{2\pi i}\int_{ \Lambda_c} \frac{e^{-z^2 t}}{-z}T(z,s)dzdt-\sum_{n=1}^{N-1}\zeta_n(s)\biggr|\quad \text{d'après}\;\eqref{zetaNNN} \\
&\leq M\biggl(\sum_{n\geq N}\frac{1}{n^{2\mathrm{Re}(s)+1}}+ \frac{1}{n^{2\mathrm{Re}(s)-d}}+\frac{1}{n^{2\mathrm{Re}(s)}}+O(\frac{1}{n^{2\mathrm{Re}(s)+2}})\biggr)dt,
 \end{align*}
 C'est à dire, on a montré qu'il existe $E(N,s)$ une fonction en $s$ et $N\gg1$ qui converge normalement vers $0$ uniformément sur tout ouvert de la forme $\{\mathrm{Re}(s)>\frac{d+1}{2}+\eps\}$ et $\eps$ un réel positif non nul quelconque, telle que

\begin{align*}
 \frac{s^2}{\Gamma(s+1)}\int_{0}^\infty \frac{t^{s-1}}{2\pi i}\int_{ \Lambda_c} \frac{e^{-z^2 t}}{-z}T(z,s)dzdt=\sum_{n=1}^{N-1}\zeta_n(s)+ E(N,s) \quad \forall N\gg1\;\forall\, \mathrm{Re}(s)>\frac{d+1}{2}.
\end{align*}
Récapitulons, on a prouvé que:
\begin{enumerate}
\item
\[
\zeta(s)=\sum_{n=1}^\infty \sum_{k=1}^\infty \frac{1}{\la_{n,k}^{2s}}=\frac{s^2}{\Gamma(s+1)}\int_{0}^\infty \frac{t^{s-1}}{2\pi i}\int_{ \Lambda_c} \frac{e^{-z^2 t}}{-z} T(z,s)dz dt\quad \forall\; \mathrm{Re}(s)>\frac{d+1}{2},
\]
\item  $\zeta$ est holomorphe sur $\mathrm{\mathrm{Re}}(s)>\frac{d+1}{2}$, en effet cela découle de  la dernière inégalité qui montre que $\bigl(\sum_{n=1}^N\zeta_n\bigr)_{N\in \N^\ast}$ converge normalement vers $\zeta$ sur tout ouvert de la forme $\{s\in \C\,|\;\mathrm{Re}(s)>\frac{d+1}{2}+\eps\}$, avec $\eps>0$ et  que  $\zeta_n$ est holomorphe sur $\bigl\{s\in \C\,|\; \mathrm{Re}(s)>\frac{1}{2}\bigr\}$ pour tout $n\in \N^\ast$.
\item  $\zeta$ a un pôle en $s=\frac{d+1}{2}$.\\

\end{enumerate}

%\[
%T(\la,s)=-\sum_{1\leq n\leq N-1}\frac{1}{n^{2s}}\log \bigl(\frac{G_n(n\sqrt{-\la})}{c_n (nz)^n}\bigr)-\bigl(\sum_{n=N}^\infty \frac{r_n(\sqrt{-\la})}{n^{2s-1}}+\frac{w(\sqrt{-\la})}{n^{2s}}+\mathrm{O}_{\sqrt{-\la}}(\frac{1}{n^{2+2s}})\bigl)
%\]
On se propose maintenant d'étudier la fonction $\zeta$ au voisinage de $s=0$.  On va scinder l'intégrale par rapport $t$ en deux intervalles $]0,\eps[$ et $]\eps,\infty[$ où $\eps>0$.\\ %On notera que sur $]0,\eps[$ l'étude de cette intégrale est délicate.\\

  On pose,
\[
q_n(z,s)=-\frac{1}{n^{2s}}\log p_n(nz)-\frac{1}{n^{2s+1}}w(z).
\]
et
\[
P(z,s)=\sum_{n=1}^\infty q_n(z,s).
\]

%On a clairement,
%\[
%\sum_{k\geq 1}e^{-t\frac{\la_{n,k}^2}{n^2}}=\int_{\Lambda_c} zt e^{-z^2t }\log \frac{G_n(inz)}{l_n(nz)^{e_n}}dz
%\]

 %aussi
%\[
%\sum_{k\geq 1}\frac{1}{\la_{n,k}^{2s}}=\frac{s^2}{\Gamma(s+1)}\int_0^\infty t^{s-1}\int_{\Lambda_c}\frac{e^{-z^2t}}{z}\frac{1}{n^{2s}}\log \frac{G_n(inz)}{l_n(nz)^{e_n}}dz
%\]
Rappelons qu'on a posé $p_n(z)= \frac{ G_n(iz)}{l_n(nz)^{e_n}}$. On a
\[
\int_{\Lambda_{-c}}  \frac{e^{-z^2 t}}{-z} q_n(z,s)dz=\int_{C_c}\frac{e^{-z^2 t}}{-z} q_n(z,s)dz+\int_{\Lambda_c}\frac{e^{-z^2 t}}{-z} q_n(z,s)dz,
\]
où $\Lambda_{-c}:=\bigl\{\la\in \C\,|\, |\arg(\la+c)|=\eps\bigr\}$ et $C_c$ est le cercle centré en zéro de rayon inférieur à $c$). Cela  résulte de l'holomorphie de $z\mapsto \log G_n(iz)$  sur un ouvert qui ne contient pas  $]-\infty,-c]\cup [c,\infty[$. Par Cauchy, on obtient
\[
\int_{C_c}\frac{e^{-z^2 t}}{-z} q_n(z,s)dz=q(0,s)=\frac{w(0)}{n^{2s+1}}.
\]
On a
\[
\begin{split}
\int_{\Lambda_{-c}}\frac{e^{-z^2t}}{z}\log p_n(nz)dz&=\int_{\sqrt{t}\Lambda_{-c}}\frac{e^{-z^2}}{z}\log p_n\biggl(\frac{nz}{\sqrt{t}}\biggr)dz\\
&=\int_{\Lambda_{-c}}\frac{e^{-z^2}}{z}\log p_n\biggl(\frac{nz}{\sqrt{t}}\biggr)dz.
\end{split}
\]
la deuxième égalité résulte du fait que $\log p_n$ est holomorphe sur l'ensemble connexe de bord $\Lambda_{-c}\cup \sqrt{t}\Lambda_{-c}$ avec
 $0<t<1$.\\

Fixons $n\geq 1$. On décompose $\Lambda_{-c}:=\bigl\{z \in \C \,|\, |\arg(z+c)|=\eps \bigr\}$ en deux ensembles
$\Lambda_{-c}^+=(\R+i\R^+)\cap \Lambda_{-c}$ et $\Lambda_{-c}^-=(\R+i\R^-)\cap \Lambda_{-c}$, donc si $z\in
\Lambda_{-c}^{+}$ alors $\eps-\frac{\pi}{2} <\arg(e^{-i\frac{\pi}{2}}z)<\frac{\pi}{2}$ et si $z\in \Lambda_{-c}^{-}$ on
aura $-\frac{\pi}{2}<\arg(e^{i\frac{\pi}{2}}z)< -\eps -\frac{\pi}{2} $, donc sur chaque branche, le développement
asymptotique de $\log G_n$ est valable.

On choisit $\eps\ll 1$ tel que pour tout $0<t\leq \eps$ on a $\frac{|z|}{|\sqrt{t}|}\gg1$ pour tout $z\in \Lambda_{-c}$, dans ce cas on peut utiliser le développement asymptotique de $G_n$ et obtenir que
{\allowdisplaybreaks
\begin{align*}
\frac{1}{2\pi i}\int_{\Lambda_{-c}}\frac{e^{-z^2}}{-z}\log p_n&\Bigl(\frac{nz}{\sqrt{t}}\Bigr)dz=\frac{1}{2\pi i}\int_{\Lambda_{-c}}\frac{e^{-z^2}}{-z}\log \biggl(\frac{G_n(\frac{e^{-i\frac{\pi}{2}}nz}{\sqrt{t}})}{l_n (\frac{nz}{\sqrt{t}})^{e_n}}\biggr)dz \\
&=\frac{1}{2\pi i}\int_{\Lambda_{-c}^+}\frac{e^{-z^2}}{-z}\log \biggl(\frac{G(\frac{e^{-i\frac{\pi}{2}}nz}{\sqrt{t}})}{l_n (\frac{nz}{\sqrt{t}})^{e_n}}\biggr)dz+\frac{1}{2\pi i}
\int_{\Lambda_{-c}^-}\frac{e^{-z^2}}{z}\log\biggl( \frac{ G(\frac{e^{+i\frac{\pi}{2}}nz}{\sqrt{t}})}{l_n (\frac{nz}{\sqrt{t}})^{e_n}}\biggr)dz\\
&=\frac{1}{2\pi i}\int_{\Lambda_{-c}}\frac{e^{-z^2}}{z}\biggl(\frac{niz}{\sqrt{t}}+a_n \log\bigl((\frac{iz}{\sqrt{t}})^2\bigr)+b_n+\frac{c_n\sqrt{t}}{iz}+O\bigl(\frac{t}{(nz)^2}\bigr) \biggr) dz\\
&=0+ \frac{a_n}{2\pi i}\int_{\Lambda_{-c}}\frac{e^{-z^2}}{z}\log(-z^2)-\frac{a_n}{2\pi i} \int_{\Lambda_{-c}}\frac{e^{-z^2}}{z}\log ({t})\\
&+\frac{b_n}{2\pi i} \int_{\Lambda_{-c}}\frac{e^{-z^2}}{z}dz+\frac{c_n\sqrt{t}}{2\pi i} \int_{\Lambda_{-c}}\frac{e^{-z^2}}{z^2}dz+ \int_{\Lambda_{-c}}\frac{e^{-z^2}}{z}O\Bigl(\frac{t}{n^2z^2}\Bigr)dz\\
&=-{\gamma}a_n-\log(t){a_n}+b_n+0+O\Bigl(\frac{t}{n^2}\Bigr).
\end{align*}
}
%(\textbf{Encore une fois, il faut supposer que $\log G_n$ a un développementasymptotique comme dans Voros (qui est vérifié par les bessel), le fait que le reste est en $\frac{1}{n^2}$ résulte de cette expression et la formule ci-dessous..})\\

où on a utilisé que  $\frac{1}{2\pi i}\int_{\Lambda_{-c}}\frac{e^{-z^2}}{z}dz=1$, $\frac{1}{2\pi
i}\int_{\Lambda_{-c}}\frac{e^{-z^2}}{z^2}dz=0$ et que  $\frac{1}{2\pi
i}\int_{\Lambda_{-c}}\frac{e^{-z^2}}{z}\log(-z^2)dz=-\gamma$. La dernière intégrale  se calcule en dérivant par
rapport à $a$
 l'identité suivante:
\[
\frac{1}{\Gamma(a)}=-\frac{1}{2\pi i}\int_{D}\frac{ze^{-z^2}}{(-z^2)^a}dz, \quad \forall a\in \C,
\]
sachant que
\[
-\Gamma'(1)=\gamma, \quad \text{la constante d'Euler},
\]

avec $D$ est un contour qui entoure strictement la demi-droite positive. On peut  montrer  cette formule en suivant la preuve de la représentation intégrale de l'inverse de la fonction Gamma, voir par exemple  \cite[§ 12.22]{Whittaker}.\\

En regroupant tout cela, on obtient:
{\allowdisplaybreaks
\begin{align*}
\int_0^{\eps} \frac{t^{s-1}}{2\pi i} \int_{\Lambda_c}& \frac{e^{-z^2 t}}{-z} \log p_n(nz)dz dt=-{\gamma}a_n\int_0^\eps t^{s-1}dt-a_n\int_0^\eps t^{s-1}\log(t)dt+b_n\int_0^\eps t^{s-1}dt+\int_0^\eps t^{s-1}O(\frac{t}{n^2})dt\\
&=-\frac{\gamma}{s}\eps^sa_n-\frac{\eps^s\log \eps}{s}a_n+\frac{1}{s^2}\eps^s a_n+\frac{1}{s}\eps^s b_n+\int_0^\eps t^{s}O(\frac{1}{n^2})dt\\
&=-\frac{\gamma}{s}a_n+\frac{1}{s} b_n+\frac{1}{s^2}a_n+\Bigl(\gamma\frac{1-\eps^s}{s}a_n
-\frac{s\eps^s \log \eps-\eps^s+1}{s^2}a_n +\frac{\eps^s-1}{s}b_n+\int_0^\eps t^{s}O(\frac{1}{n^2})dt\Bigr).\\
\end{align*}
}
On vérifie que:
\[
\int_0^\eps\frac{t^{s-1}}{2\pi i}\int_{C_c}\frac{e^{-z^2 t}}{-z} q_n(z,s)dz=q(0,s)=\frac{w(0)\eps^s}{sn^{2s+1}}=\frac{w(0)}{sn^{2s+1}}+(\eps^s-1)\frac{w(0)}{sn^{2s+1}}.
\]
Par suite,
{\allowdisplaybreaks
\begin{align*}
\int_0^{\eps} \frac{t^{s-1}}{2\pi i} \int_{\Lambda_c} &\frac{e^{-z^2 t}}{-z} q_n(z,s)dz dt
=\frac{\gamma}{s}\frac{a_n}{n^{2s}}+\frac{w(0)}{sn^{2s+1}}- \frac{b_n}{sn^{2s}}-\frac{1}{s^2}a_n\\
&-\Bigl(\gamma\frac{1-\eps^s}{s}\frac{a_n}{n^{2s}}
-\frac{s\eps^s \log \eps-\eps^s+1}{s^2}\frac{a_n}{n^{2s}}+\frac{1}{n^{2s}}\frac{\eps^s-1}{s}\bigl(b_n-\frac{w(0)}{n}\bigr)+\int_0^\eps t^{s}O(\frac{1}{n^{2s+2}})dt\Bigr).\\
\end{align*}
}

On note par les mêmes notations les prolongements analytiques de $A$ et $P-B$  sur la bande $0<\mathrm{Re}(s)<\frac{d+1}{2}$, donc

\begin{align*}
\frac{s^2}{\Gamma(s+1)}&\int_0^\eps \frac{t^{s-1}}{2\pi i} \int_{\Lambda_c} \frac{e^{-z^2 t}}{-z}P(z,s)dz dt=\frac{s}{\Gamma(s+1)}\Bigl(\gamma A(s)-B(s)+P(s)-\frac{1}{s}A(s)\Bigr)\\
&+\frac{s^2}{\Gamma(s+1)}\Biggl(\gamma \frac{1-\eps^s}{s}A(s)
-\frac{s\eps^s \log \eps-\eps^s+1}{s^2} A(s)+s(\eps^s-1)\bigl(B(s)-P(s)\bigr)\\
&+\int_0^\eps t^{s}O(\zeta_\Q(2s+2))dt\Biggr)\\
\end{align*}
On pose $g$, la fonction suivante:
\[g(s)=\gamma \frac{1-\eps^s}{s}A(s)
-\frac{s\eps^s \log \eps-\eps^s+1}{s^2} A(s)+s(\eps^s-1)\bigl(B(s)-P(s)\bigr)
+\int_0^\eps t^{s}O(\zeta_\Q(2s+2))dt,\]
alors $g$ est analytique en un voisinage ouvert de $s=0$, puisqu'on sait d'après \eqref{P-B} que $P-B$ est analytique en $0$ et que $A$ l'est aussi (par hypothèses sur les coefficients $a_n$) et on montre par exemple à l'aide d'un développement limité que $s\mapsto \frac{1-\eps^s}{s} $ et $s\mapsto \frac{s\eps^s \log \eps-\eps^s+1}{s^2}$ sont analytique sur $\C$, pour $s\mapsto \int_0^\eps t^{s}O(\zeta_\Q(2s+2))dt$ il suffit de remarquer que c'est une limite  d'une suite de fonctions analytiques pour la convergence normale au voisinage de $s=0$.\\

On écrit donc,
\begin{align*}
\frac{s^2}{\Gamma(s+1)}&\int_0^\eps \frac{t^{s-1}}{2\pi i} \int_{\Lambda_c} \frac{e^{-z^2 t}}{-z}P(z,s)dz dt=\frac{s}{\Gamma(s+1)}\Bigl(\gamma A(s)-B(s)+P(s)-\frac{1}{s}A(s)\Bigr)+\frac{s^2}{\Gamma(s+1)}g(s),
\end{align*}
on conlcut que cette fonction se prolonge en une fonction holomorphe au voisinage de $s=0$.\\

Montrons maintenant que le terme suivant se prolonge en une  fonction analytique en $s$ au voisinage de $s=0$:
\[
\begin{split}
\int_\eps^\infty \frac{t^{s-1}}{2\pi i} \int_{\Lambda_c} \frac{e^{-z^2 t}}{-z} P(z,s)dz dt,
\end{split}
\]
il suffit d'étudier
\[
\int_\eps^\infty \frac{t^{s-1}}{2\pi i} \int_{\Lambda_c} \frac{e^{-z^2 t}}{z}\sum_{ n\geq N}\frac{1}{n^{2s}}\log \Bigl(\frac{G_n(inz )}{l_n (nz)^n}\Bigr),
\]
pour $N\gg1$, puisque la  somme partielle peut être étudier par la théorie de \cite{Voros}. \\

Par définition de $P(z,s)$ et par les propriétés des $G_n$, on a pour tout $\mathrm{Re}(s)>1$
\[
P(z,s)=\sum_{n=1}^Nq_n(z,s)-\Biggl(\sum_{n=N}^\infty O\Bigl(\frac{1}{n^{2s+2}}\Bigr)+\frac{1}{n^{2s}}\log \rho(iz) +\frac{e_n }{n^{2s}}r(iz)+\frac{1}{n^{2s}}\mathrm{O}_{z}\Bigl(\frac{1}{n^{2}}\Bigr)\Biggl)\]
où $\eta(n,iz)=\mathrm{O}_{z}(\frac{1}{n^{2}})$ est par hypothèse bornée en $z$ uniformément en $n>N$. On doit étudier les séries de termes généraux suivants:
\begin{align*}
&\int_\eps^\infty \frac{t^{s-1}}{2\pi i} \int_{\Lambda_c} \frac{e^{-z^2 t}}{z}\frac{r(iz)}{n^{2s-d}}dz,\\
&\int_\eps^\infty \frac{t^{s-1}}{2\pi i} \int_{\Lambda_c} \frac{e^{-z^2 t}}{z}\frac{\log \rho(iz)}{n^{2s}}dz,\\
&\int_\eps^\infty \frac{t^{s-1}}{2\pi i} \int_{\Lambda_c} \frac{e^{-z^2 t}}{z}\frac{\eta(n,iz)}{n^{2s+2}}dz.\\
\end{align*}

On déduit directement  que la série de terme général: $\int_\eps^\infty \frac{t^{s-1}}{2\pi i} \int_{\Lambda_c} \frac{e^{-z^2 t}}{z}\frac{\eta(n,iz)}{n^{2s+2}}dz$ est normalement convergente sur un voisinage de $s=0$. Pour les deux termes restants il suffit d'étudier le premier terme.  On note par $\zeta_\Q(2s-1)$ le prolongement analytique de cette $\sum_{k=1}^\infty\frac{1}{k^{2s-1}}$ au voisinage de $0$.

 Soit  $s$ dans un voisinage de $0$ et on considère
\[
\begin{split}
&\int_\eps^\infty \frac{t^{s-1}}{2\pi i} \int_{\Lambda_c} \frac{e^{-z^2 t}}{z}r(iz)\Bigl(\zeta_\Q(2s-1)-\sum_{n=1}^{N-1}\frac{1}{n^{2s+1}} \Bigr)dzdt,\\
\end{split}
\]
On conclut que c'est une fonction analytique en utilisant l'inégalité suivante valable sur une voisingae de $s=0$:
\begin{align*}
\int_\eps^\infty \int_{\Lambda_{c,\geq R}}\Bigl| t^{s-1}\frac{e^{-tz^2 }}{z}r(iz)dz\Bigr| dt&\leq \int_{v\geq R}\int_\eps^\infty t^{\mathrm{Re}(s)-1}e^{-t\mu v^2}dv dt\\
&=\int_{v\geq R}\frac{1}{\mu^{2\mathrm{Re}(s)} v^{2\mathrm{Re}(s)}}\int_{\eps \mu v^2}^\infty t^{\mathrm{Re}(s)-1}e^{-t}dt dv<\infty.\\
\end{align*}
Par conséquent,
\[
\begin{split}
\int_\eps^\infty \frac{t^{s-1}}{2\pi i} \int_{\Lambda_c} \frac{e^{-z^2 t}}{-z} P(z,s)dz dt
\end{split}
\]
est analytique au voisinage de $s=0$.\\

Etudions maintenant la fonction  suivante:
\[
F(s)=\frac{s^2}{\Gamma(s+1)}\zeta_\Q(2s+1)\int_0^\infty \frac{t^{s-1}}{2\pi i} \int_{\Lambda_c} \frac{e^{-z^2 t}}{z}w(iz)dzdt,
\]
au voisinage de $s=0$. On a
\begin{align*}
F(s)&= \frac{s^2}{\Gamma(s+1)}\zeta_\Q(2s+1)\int_0^\infty t^{s-1}\int_{\Lambda_c}\frac{e^{-z^2t}}{z}w(iz)dzdt\\
&=\frac{s^2}{\Gamma(s+1)}\zeta_\Q(2s+1)\int_{\Lambda_c}\frac{\Gamma(s)}{z^{2s+1}}w(iz)dz\\
&=\frac{s}{\Gamma(s+1)}\zeta_\Q(2s+1)\int_{\Lambda_c}\frac{\Gamma(s+1)}{z^{2s+1}}w(iz)dz.
\end{align*}
puisque on a supposé que $w(z)=O(\frac{1}{z})$ pour $|z|\gg1$. Toujours avec cet argument et le fait que
$\zeta_\Q(2s+1)=\frac{1}{2s}+\gamma+o(1)$ pour $|s|\ll 1$, on conclut que   $F$ s'étend  analytiquement au voisinage de $s=0$.\\

On pose maintenant:
\[
h(s)=g(s)+\int_\eps^\infty \frac{t^{s-1}}{2\pi i} \int_{\Lambda_c} \frac{e^{-z^2 t}}{-z} P(z,s)dz dt
\]
alors $h$ est analytique au voisinage de $s=0$, et

\[
\begin{split}
\zeta(s)&=\frac{s}{\Gamma(s+1)}\bigl[\gamma A(s)-\frac{1}{s}A(s)+P(s)-B(s) \bigr]
+\frac{s^2}{\Gamma(s+1)}\zeta_\Q(2s+1)F(s)\\
&+\frac{s^2}{\Gamma(s+1)}h(s),
\end{split}
\]
c'est à dire que $\zeta$ est prolongeable analytiquement au voisinage de $s=0$, ce qui termine la preuve du théorème.

\end{proof}

\begin{remarque}
\rm{ lorsque $w(z)$ est un polynôme en $q$ avec $q=\frac{1}{\sqrt{1+z^2}}$ alors le calcul de cette intégrale se ramène à ce genre d'intégrales bien connu, $\forall a>0$:
 \begin{align*}
\frac{1}{2\pi i}\int_{\Lambda_c} \frac{e^{-z^2 t}}{-z} \bigl(q(iz)\bigr)^a dz&=\frac{1}{2\pi i}\int_{\Lambda_c} \frac{e^{-z^2 t}}{-z} \frac{1}{(1-z^2)^a}dz\\
&= \frac{1}{2\pi i}\int_{\Lambda_c} \frac{e^{-\lambda t}}{-\la} \frac{1}{(1-\la)^a}d\la\\
 &= - \frac{1}{2\pi i}e^{-t} \int_{\Lambda_c-1}\frac{e^{-u t}}{1+u}\frac{1}{(-u)^a}d\lambda\\
&=\frac{1}{\pi} \sin(\pi a)\Gamma(1-a)\Gamma(a,t).
\end{align*}
Alors,
\[
\begin{split}
\int_0^\infty \frac{t^{s-1}}{2\pi i}\int_{\Lambda_c} &\frac{e^{-z^2 t}}{-z} \frac{1}{(1+(iz)^2)^a}d\la=\frac{1}{\pi} \sin(\pi a)\Gamma(1-a)\int_{0}^\infty t^{s-1}\Gamma(a,t)dt\\
&=\frac{1}{\pi} \sin(\pi a)\Gamma(1-a)\int_0^\infty \int_t^{\infty} u^{a-1}e^{-u}du\\
&=\frac{1}{\pi} \sin(\pi a)\Gamma(1-a)\biggl( \Bigl[\frac{1}{s}t^s \int_t^\infty u^{a-1}e^{-u}du\Bigr]_0^\infty+\frac{1}{s}\int_0^\infty t^{s+a-1} e^{-t}dt \biggr)\\
&=\frac{1}{\pi s} \sin(\pi a)\Gamma(1-a)\Gamma(s+a).
\end{split}
\]
}
\end{remarque}

\begin{remarque}
\rm{
La fonction $\theta(t,s):=\sum_{n=1}^\infty \frac{1}{n^{2s}}\theta_n(t)$ a joué le rôle de fonction thêta associée à la famille $\bigl\{\la^{2}_{n,k}\,\bigl|\, k,n\geq 1 \bigr\}$, mais on peut se demander si la fonction
 \[\theta(t):=\sum_{n\geq 1}\sum_{k\geq 1}e^{-\la_{n,k}^2t}\]
est une fonction Thêta au sens de \cite{Voros}? On en donnera une réponse partielle:  on va  montrer sous une condition supplémentaire que
\[
\lim_{t\mapsto 0} t\theta(t)<\infty.
\]
On suppose donc que $\la^2_{n,1}<\la_{n+1,1}^2<\la_{n,2}^2<\la_{n+1,2}^2<\ldots$, $\forall n\geq 1$, (ce qui est par exemple  le cas pour les zéros de fonctions de Bessel ainsi que  les dérivées premières et secondes) ce qui nous donne  que
\[
\widetilde{\theta}_{n+1}(t)\leq  \widetilde{\theta}_{n}(t)\leq \widetilde{\theta}_{1}(t)<\infty, \quad \forall t>0
\]
Cette  dernière inégalité est conséquence du développement asymptotique de $G_1$.\\

Si $t>\eps>0$, alors
\[
\widetilde{\theta}_n(t)=\sum_{k\geq 1}e^{-\la_{n,k}^2(t-\eps)-\la_{n,k}^2\eps}\leq \widetilde{\theta}_n(\eps) e^{-n^2 c(t-\eps)} \quad \forall n\geq 1
\]
par suite
\[
{\theta}(t)\leq \widetilde{\theta}_1(\eps)\sum_{n\geq 1} e^{-n^2 c(t-\eps)}=\widetilde{\theta}_1(\eps)\Theta\bigl( c(t-\eps)\bigr) \quad \forall \,\eps>0 \quad \forall\, t>\eps.
\]
où  $\Theta(t):=\sum_{n\geq 1} e^{-n^2 t}$ est la fonction thêta habituelle qui converge pour tout $t>0$ voir par exemple \cite[p. 96]{Soulé}. On  en  déduit que $\widetilde{\theta}(t)$ est fini, $\forall t>0$. Si l'on prend $\eps=\frac{t}{2}$,  comme $\sqrt{t}\widetilde{\theta}_1(t)$ et $\sqrt{t}\Theta(t)$ convergent vers une limite finie lorsque $t$ tend vers 0, alors par l'inégalité précédente on obtient que
\[
\underset{t\mapsto 0}{\limsup} \, t \theta(t),
\]
est fini.

}
\end{remarque}

\subsection{Application: régularisation de la fonction zêta de   $\Delta_{\overline{\mathcal{O}(m)}_\infty}$}

Cette section est consacrée à l'étude d'une nouvelle classe de fonctions Zêta associées à des opérateurs Laplaciens
singuliers, en utilisant les  résultats précédents on va montrer que ces fonctions Zêta possèdent des propriétés
analogues aux fonctions Zêta classiques des opérateurs Laplaciens associées à des métriques de classes $\cl$.\\

Soit $(X,h_X)$ une variété kählérienne et $(E,h_E)$ un fibré holomorphe hermitien. Si l'on suppose que $h_X$ et $h_E$
sont  de classe $\cl$ alors   on sait que l'opérateur $\Delta_{\overline{E}}^\bullet$, agissant sur
$A^{0,\bullet}(X,E)$, admet un spectre  infini et discret et possède un noyau de chaleur qu'on note par
$e^{-t\Delta_{\overline{E}}}$, avec $t>0$, voir \cite[Définition 2.15]{heat}. En plus $e^{-t\Delta_{\overline{E}}}$
admet une trace au sens de \cite[§ 2.6]{heat} et  que la fonction Zêta associée à ce spectre s'écrit comme la
transformée de Mellin  de cette  trace. On sait que cette trace possède un développement par prolongement analytique
 en série de Laurent pour $t$ petit, voir \cite[p. 91]{heat} cela permet de  prouver que la fonction Zêta s'étend
méromorphiquement
au plan complexe et qu'elle est holomorphe en zéro, voir \cite[9.35]{heat}.\\

Dans notre cas, c'est à dire $\Bigl((T\p^1,h_\infty); \overline{\mathcal{O}(m)}_\infty\Bigr)$, les métriques
considérées sont singulières. On ne peut pas donc appliquer la théorie classique pour affirmer que
$\Delta_{\overline{\mathcal{O}(m)}_\infty}$ possède un spectre discret infini. Même si l'on arrive  à déterminer le
spectre, par un autre moyen, on ne peut pas utiliser cette théorie   pour étudier la fonction Zêta associée quant au
calcul explicite du déterminant régularisé cela reste une tâche délicate, voir par exemple les calculs du \cite{GSZ}.\\

Dans un précédent texte, on a étudié le spectre du Laplacien singulier:
\[
 \Delta_{\overline{\mathcal{O}(m)}_\infty},\forall m\in \N,
\]
et on a montré que ses valeurs propres et vecteurs propres sont décrits par une famille de fonctions analytiques $(L_n)_{n\in \Z}$.  Ces fonctions ont été définies comme suit:
\[
 L_n(z)=-z^{m}\frac{d}{dz}\Bigl(z^{-m}J_n(z)J_{n-m}(z)\Bigr) \quad \forall z\in \C,
\]
et   on a introduit l'ensemble suivant
\begin{equation}\label{Zeros1}
Z_{n}:=\Bigl\{\lambda \in \C\,\bigl|\,L_n(\la)=0 \Bigr\}.
\end{equation}
Alors, on a montré  dans \cite{Mounir1}:

\begin{lemma}\label{Zeros2}
$\forall n\in \Z$, l'ensemble $Z_n$ est un sous-ensemble infini et discret de $\R^\ast$,

\[
Z_n\cap \{z\in \R|\, J_{n-m}(z)=0 \}=\emptyset,
\]
et
\begin{equation}\label{encoreeq2}
L_n(z)=-J_n(z)J_{n-m-1}(z)+J_{n+1}(z)J_{n-m}(z),\quad \forall z\in \C.
\end{equation}
\end{lemma}

\begin{theorem}\label{deriveeLnormeCopie}
Soit $n\in \Z$ et  $\la\in Z_{n}$, alors $\la$ est un zéro simple de $L_n$  et on a
\[
L'_{n}(\la)=2\frac{J_{n-m}(\la)}{J_{n}(\la)}\bigl(\vf_{{}_{n,\la}},\vf_{{}_{n,\la}}\bigr)_{L^2,\infty}.
\]
où $\vf_{n,\la}$ est un vecteur propre associé à $\la$.
\end{theorem}

\begin{theorem}\label{spectreOm}
On a, pour tout $m\in\N^\ast$
{{} \[
\mathrm{Spec}(\Delta_{\overline{\mathcal{O}(m)}_\infty})=\Bigl\{0\Bigr\}\bigcup \Bigl\{\frac{\la^2}{4} \Bigl|\, \exists n\in \N,\, \la\in Z_n  \Bigr\}.
\]}
\begin{itemize}
\item
Si  $m$ est pair, alors  la multiplicité de $\frac{\la^2}{4}$ est $2$ si $\la\in Z_n$ avec $n\geq m+1 $ ou $0\leq n\leq \frac{m}{2}-1$ et de multiplicité $1$ si $\la\in Z_{\frac{m}{2}}$.
\item
Lorsque $m$ est impair, alors $\frac{\la^2}{4}$ est de multiplicité $2$ si $n\geq m+1$, égale à $1$ si $0\leq n\leq m $.
\end{itemize}
\end{theorem}
Lorsque $m=0$, on avait montré

\begin{theorem}
  \[
\mathrm{ Spec}(\Delta_{{\overline{\mathcal{O}}}_\infty})=\Bigl\{0\Bigr\}\bigcup\Bigl\{\frac{j_{m,k}^2}{4},\frac{j_{n,l}^{'2}}{4}\,\Bigl|\, n,m\in \N,\quad k,l\geq 1\Bigr\},
\]
avec  $\frac{j_{n,k}^2}{4}$ (resp.  $\frac{j_{n,k}^{'2}}{4}$) est   de multiplicité 2 si  $n> 0$, et de multiplicité 1 quand $n=0$, où $j_{\ast,\ast}$ (resp. $j'_{\ast,\ast}$) est un zéro positif de $J_\ast$(resp. $J'_\ast$).
\end{theorem}

Il est donc naturel d'associer à ce spectre une fonction Zêta et d'en étudier les propriétés. On introduit donc la définition suivante:

\begin{definition}\label{definitionzetacan}Soit $m$ un entier positif. On pose
\[
\zeta_{\Delta_{\overline{\mathcal{O}(m)}_\infty}}(s):=\sum_{ \al\in \mathrm{Spec}(\Delta_{\overline{\mathcal{O}(m)}_\infty})\setminus\{0\} }\frac{1}{\al^s},\quad \forall s\in \C,
\]
où chaque valeur propre est comptée avec sa multiplicité. On appelle $\zeta_{\Delta_{\overline{\mathcal{O}(m)}_\infty}}$ la fonction zêta associée à l'opérateur singulier ${\Delta_{\overline{\mathcal{O}(m)}_\infty}}$.
\end{definition}
%\vspace{2cm}

Le principal résultat de ce paragraphe est le théorème résultat ci-dessous. On  distinguera deux cas: le cas $m=0$ et celui lorsque $m\geq 1$. En fait, on peut éviter cela et supposer que $m\geq 0$, mais on préfère  cette séparation de cas par souci de clarté.
\begin{theorem}\label{lavaleurdezeta}
Pour tout $m\geq 0$, la fonction $\zeta_{\Delta_{\overline{\mathcal{O}(m)}_\infty}}$ converge pour tout $\mathrm{Re}(s)>1$, avec un pôle en $1$ et  admet un prolongement analytique au voisinage de $s=0$, de plus, on a

\[
\zeta_{\Delta_{\overline{\mathcal{O}(m)}_\infty}}(0)=-\frac{2}{3}-\frac{m}{2},\footnote{Il est important de noter que d'après la théorie classique de \cite{heat} la quantité $\zeta_{\Delta_{\overline{L}}}(0)$ est un invariant qui ne dépend pas de la métrique $\cl$ sur $L$. On vérifie qu'ici cette valeur correspond  bien à la valeur de cet invariant.}
\]
et

\[\zeta_{\Delta_{\overline{\mathcal{O}(m)}_\infty}}'(0)=4\zeta_\Q'(-1)-\frac{1}{6}-\log \frac{\,{}(m+2)^{{m+1}}}{(m+1)!^2}=T_g(\overline{T\p^1}_\infty, \overline{\mathcal{O}(m)}_\infty).\]
\end{theorem}

\begin{proof} La démonstration est une application du théorème \eqref{regularisation}.
On commencera par le cas $m=0$ qui est relativement facile à traiter.

\subsubsection{Régularisation de la fonction zêta  $\zeta_{\Delta_{\overline{\mathcal{O}}_\infty}}$}

On pose
\[
\zeta_0(s)=\zeta_\mathcal{D}(s)+\zeta_\mathcal{N}(s),
\]
avec
\[
\zeta_\mathcal{D}(s)=\sum_{k\geq 1}\frac{1}{j_{0,k}^{2s}}+ \sum_{n\geq 1}\sum_{k\geq 1}\frac{2}{j_{n,k}^{2s}}\quad\text{et}\quad \zeta_\mathcal{N}(s)= \sum_{k\geq 1}\frac{1}{{j'}_{0,k}^{2s}}+\sum_{n\geq 1}\sum_{k\geq 1}\frac{2}{{j'}_{n,k}^{2s}}, \footnote{On a choisit les lettres $\mathcal{D}$  et $\mathcal{N}$ pour rappeler que ces objets sont liés aux deux problèmes avec condition aux bords classiques: de \textit{Dirichlet} et  de \textit{Neumann}.}
\]
Alors, voir \eqref{definitionzetacan}
\[
\zeta_{\Delta_{\overline{\mathcal{O}}\infty}}(s)=4^s\zeta_0(s).
\]
%{\footnotesize Rappelons que $\mathrm{Spec}(\Delta_{\overline{O}_0})=\{\frac{j^2_{n,k}}{4},\, \frac{c^2_{n,k}}{4}|\, n\in \N, \text{et}\, k\geq 1 \}$}.
On va montrer que
\[
\zeta_0(0)=-\frac{2}{3},\quad \zeta'_0(0)=4\zeta'_\Q(-1)-\frac{1}{6}+\frac{1}{3}\log 2.
\]
ce qui donnera le résultat cherché c'est à dire que $\zeta_{\Delta_{\overline{\mathcal{O}}\infty}}(0)=-\frac{2}{3}$ et $\zeta_{\Delta_{\overline{\mathcal{O}}\infty}}'(0)=4\zeta'_\Q(-1)-\frac{1}{6}-\log 2$.\\

D'après les calculs \cite[corollaire. 1]{Spreafico}, on a
\begin{equation}\label{Dirichlet}
\zeta_\mathcal{D}(0)=\frac{1}{6},\quad \zeta'_\mathcal{D}(0)=2\zeta'_\Q(-1)+\frac{1}{2}\log 2\pi +\frac{5}{12}.
\end{equation}
Il suffit d'étudier $\zeta_\mathcal{N}$. On considère  $\bigl\{ I_n'\,\bigl|\, n\geq 1 \bigr\}$ et la fonction zêta associée:
\begin{equation}\label{zetaN}
z_\mathcal{N}(s)=\sum_{n\geq 1}\sum_{k\geq 1}\frac{1}{{j'}_{n,k}^{2s}}.
\end{equation}

Avec les notations de la proposition \eqref{regularisation}, on considère
\[
p_{\mathcal{N},n}(z)=-\log \biggl(\frac{I_n'(nz)}{\frac{1}{2^n\Gamma(n)}(nz)^{n-1}}\biggr),
\]

Par \eqref{dvlptz2}, \eqref{developpement2} et \eqref{petitlemme}, la famille $\bigl\{I_n'\,|\, n\geq 1\bigr\}$ vérifie les conditions de théorème \eqref{regularisation}. Par un simple calcul, on trouve que:
 \[
 a_{\mathcal{N},n}:=\frac{1}{2}(n-\frac{1}{2}),\, b_{\mathcal{N},n}:=\frac{1}{2}\log 2\pi+(n-\frac{1}{2})\log n-n\log 2-\log \Gamma (n),\,\text{et}\, w(z)=V_1(z).
\]
Par suite
 \[
\begin{split}
 A_{\mathcal{N}}(s)&:=\sum_{n=1}^\infty  \frac{a_{\mathcal{N},n}}{n^{2s}}\\
&=\frac{1}{2}\zeta_\Q(2s-1)-\frac{1}{4}\zeta_\Q(2s),
\end{split}
\]
(en particulier, $A_{\mathcal{N}}'(0)=\zeta'_\Q(-1)-\frac{1}{2}\zeta'_\Q(0)=\zeta'_\Q(-1)+\frac{1}{4}\log (2\pi)
$) voir par exemple,\cite[pp. 94-95]{Soulé}.
\[
 \begin{split}
  B_{\mathcal{N}}(s)&:=\sum_{n=1}^\infty \frac{b_{\mathcal{N},n}}{n^{2s}}\\
&=-\sum_{n=1}^\infty \frac{\log \Gamma(n)}{n^{2s}}+\frac{1}{2}\log (2\pi)\zeta_\Q(2s)-\log( 2) \zeta_\Q(2s-1)-\zeta'_\Q(2s-1)+\frac{1}{2}\zeta'_\Q(2s),
 \end{split}
\]

et
\[
 P_{\mathcal{N}}(s):=\sum_{n=1}^\infty \frac{1}{n^{2s}}w(0)=\sum_{n=1}^\infty \frac{1}{n^{2s}}V_1(0)=-\frac{1}{12}\zeta_\Q(2s+1)
\]
(on a $V_1(0)=-\frac{1}{12}$). \\

La proposition \eqref{regularisation} affirme que $P_{\mathcal{N}}-B_{\mathcal{N}}$ est prolongeable en une fonction analytique  en un voisinage ouvert de $s=0$, on va en donner une autre preuve plus directe et on calculera  sa valeur en $s=0$. On pose
\begin{equation}\label{eta}
\eta(s):=\sum_{n=1}^\infty \frac{\log \Gamma(n)}{n^{2s}}
 -\frac{1}{12}\zeta_\Q(2s+1),
\end{equation}
alors
\begin{equation}\label{mm}
\begin{split}
 -B_\mathcal{N}(s)+P_\mathcal{N}(s)&=\eta(s)-\frac{1}{2}\log(2\pi)\zeta_\Q(2s)\\
 &+\log(2)\zeta_\Q(2s-1)+\zeta_\Q'(2s-1)-\frac{1}{2}\zeta'_\Q(2s)\quad \forall\, \mathrm{Re}(s)>1.
\end{split}
\end{equation}
A l'exception de $\eta$, on déduit par la théorie classique de la fonction Zêta de Riemann que les termes de l'égalité ci-dessus sont prolongeables analytiquement en un voisinage ouvert de $s=0$. Montrons que $\eta$ l'est aussi. On a
\begin{lemma}
 La fonction $\eta$ définie par
\[
\eta(s):=\sum_{n=1}^\infty \frac{\log \Gamma(n)}{n^{2s}}
 -\frac{1}{12}\zeta_\Q(2s+1),
\]
converge pour $\mathrm{Re}(s)>1$ et admet un continuation holomorphe au voisinage de 0 et
\[
\eta(0)= \zeta'_\Q(-1)-\frac{\gamma}{12}-\frac{1}{4}\log (2\pi).
\]

\end{lemma}
\begin{proof}
D'après \cite[8.343.2]{Table}:
\begin{equation}\label{devgamma}
 \log \Gamma(z)=z\log z-z-\frac{1}{2}\log z+\frac{1}{2}\log2\pi+\frac{1}{12z}+R_2(z), \quad \mathrm{Re}(z)>0,
\end{equation}
avec
\[
 |R_2(z)|<\frac{|B_{4}|}{12|z|^{3}\cos^{3}(\frac{1}{2}\arg(z))}\quad \mathrm{Re}(z)>0.
\]
Donc lorsque $z=n$ est  un entier on a
\begin{equation}\label{estimation}
 \sum_{n=1}^\infty \frac{|R_2(n)|}{n^{2s}}\leq \sum_{n=1}^\infty \frac{|B_4|}{12n^{3+2s}}=|B_4|\zeta_\Q(2s+3).
\end{equation}
Par cette estimation et par les propriétés classiques des séries de Bertrand, on a
%\[
 %\begin{split}
%\sum_{n=1}^\infty\frac{\log \Gamma(n)}{n^{2s}}&= \sum_{n=1}^\infty\frac{\log n}{n^{2s-1}}-\sum_{n=1}^\infty\frac{1}{n^{2s-1}}-\sum_{n=1}^\infty\frac{1}{2}\frac{\log n}{n^{2s}}+\sum_{n=1}^\infty\frac{1}{2}\frac{\log2\pi}{n^{2s}}+\sum_{n=1}^\infty\frac{1}{12n^{2s+1}}\\
%&+\sum_{n=1}^\infty \frac{R_2(n)}{n^{2s}}
 %\end{split}
%\]
\[
 \begin{split}
\eta(s)&=\sum_{n=1}^\infty \frac{\log \Gamma(n)}{n^{2s}}-\frac{1}{12n^{2s+1}}\\
&=\sum_{n=1}^\infty\frac{\log n}{n^{2s-1}}-\sum_{n=1}^\infty\frac{1}{n^{2s-1}}-\sum_{n=1}^\infty\frac{1}{2}\frac{\log n}{n^{2s}}+\sum_{n=1}^\infty\frac{1}{2}\frac{\log2\pi}{n^{2s}}+\sum_{n=1}^\infty \frac{R_2(n)}{n^{2s}}\\
&=-\zeta'_\Q(2s-1)-\zeta_\Q(2s-1)+\frac{1}{2}\zeta'_\Q(2s)+\frac{1}{2}\log (2\pi)\zeta_\Q(2s)+\sum_{n=1}^\infty \frac{R_2(n)}{n^{2s}},
 \end{split}
\]
qui converge pour $\mathrm{\mathrm{Re}(s)}>1$ avec un pôle en 1.\\

%\begin{equation}\label{valeureta}
% \eta(s)=\sum_{n=1}^\infty \frac{R_2(n)}{n^{2s}}-\zeta'_\Q(2s-1)-\zeta_\Q(2s-1)+\frac{1}{2}\zeta'_\Q(2s)+\frac{1}{2}\log (2\pi)\zeta_\Q(2s)
%\end{equation}

On a $\zeta_\Q'(2s-1)$, $\zeta_\Q(2s-1)$, $\zeta_\Q'(2s)$ et $\zeta_\Q(2s)$ admettent des prolongements analytiques au voisinage de $ s=0$ et par à l'estimation \eqref{estimation} on conclut que $\eta$ admet un prolongement analytique au voisinage de $s=0$. On notera par la même notation $\eta$ le prolongement de $\eta$ au voisinage de $0$.\\
On a donc
\[
\eta(0)=-\zeta'_\Q(-1)-\zeta_\Q(-1)+\frac{1}{2}\zeta'_\Q(0)+\frac{1}{2}\log (2\pi)\zeta_\Q(0)+\sum_{n=1}^\infty R_2(n).
\]

%De \eqref{estimation}, on déduit que la valeur de $\sum_{n=1}^\infty \frac{\log \Gamma(n)}{n^{2s}}-\frac{1}{12n^{2s+1}}$ en zéro est
%\[
%\begin{split}
%\sum_{n=1}^\infty {R_2(n)}-\zeta_\Q'(-1)-\zeta_\Q(-1)+\frac{1}{2}\zeta'_\Q(0)+\frac{1}{2}\log (2\pi)\zeta_\Q(0)=&-\zeta'_\Q(-1)+\frac{1}{12}-\frac{1}{4}\log (2\pi)\\
%-\frac{1}{4}\log (2\pi)+
%\sum_{n=1}^\infty {R_2(n)}
%\end{split}
%\]

Calculons
\[
 \sum_{n=1}^\infty R_2(n)
\]

%\[
%\lim_{N\mapsto \infty}\bigl(\sum_{n=1}^N{\log \Gamma(n)}-\sum_{n=1}^N n\log n+\frac{N(N+1)}{2}-N\log \sqrt{2\pi}-\sum_{n=1}^N\frac{1}{12n}\bigr)
%\]
On aura besoin des identités suivantes:
\begin{enumerate}
\item

\begin{equation}\label{gamma}
 \gamma=\sum_{n=1}^N\frac{1}{n}-\log N+o(1)
\end{equation}

\item
\begin{equation}\label{simplifie}
 \log(\prod_{n=1}^{N-1}n^{-n})=\sum_{n=1}^{N}\log \Gamma(n)-N\log \Gamma(N)
\end{equation}

%\[
 %\begin{split}
  %\log \bigl(\prod_{n=1}^{N-1} n^{-n}\bigr)&=\log \bigl(\prod_{n=1}^{N-1} n^{N-n-N}\bigr)\\
 %&=\log \bigl(\prod_{n=1}^{N-1}n^{N-n}\bigr)-N\log\bigl( \prod_{n=1}^{N-1} n\bigr)\\
%&=\sum_{n=1}^{N}\log \Gamma(n)-N\log \Gamma(N)
%\end{split}
%\]

\item
\begin{equation}\label{kinkelin}
\log A=\underset{n\to+\infty}{\lim} \log(1^1\cdot 2^2 \cdot N^N)-(\frac{N^2}{2}+\frac{N}{2}+\frac{1}{12})\log N+\frac{N^2}{4},
\end{equation}

et \begin{equation}\label{kinkelinvaleur}
    \log A=-\zeta'(-1)+\frac{1}{12}
   \end{equation}

où $A$ est la constante de \textit{Kinkelin}, voir \cite[p.461-462]{Voros}.\\

\item
\[
 \log \Gamma(N)=(N-\frac{1}{2})\log N-N+\frac{1}{2}\log (2\pi)+\frac{1}{12N}+\mathrm{o}(\frac{1}{N^2})
\]
 voir \cite[p.895]{Table}.

\end{enumerate}
A l'aide de ces dernières formules, on montre que
\[
 \sum_{n=1}^\infty R_2(n)=2\zeta'_\Q(-1)-\frac{1}{12}-\frac{\gamma}{12}+\frac{1}{4}\log (2\pi).
\]
En effet, soit $N\geq 2$, on a
{\allowdisplaybreaks
 \begin{align*}
\sum_{n=1}^N R_2(n)&=\sum_{n=1}^N \Bigr(\log \Gamma(n)-n\log(n) +n+\frac{1}{2}\log n-\log \sqrt{2\pi}\Bigr)-\frac{\gamma}{12}-\frac{\log N}{12}+o(1)\,\, \text{par}\,\eqref{devgamma},\eqref{gamma}\\
&=\sum_{n=1}^N{\log \Gamma(n)}-\sum_{n=1}^N n\log n+\frac{N(N+1)}{2}+\frac{1}{2}(\log \Gamma(N)+\log N)-N\log \sqrt{2\pi}\\
&-\frac{\gamma}{12}-\frac{\log N}{12}+o(1)\\
&=\sum_{n=1}^N{\log \Gamma(n)}+\log \bigl(\prod_{n=1}^{N-1} n^{-n}\bigr)-N\log N+\frac{N(N+1)}{2}+\frac{1}{2}\log \Gamma(N)\\
&-N\log \sqrt{2\pi}-\frac{\gamma}{12}+\frac{5}{12}\log N+o(1)\\
&=2\log \bigl(\prod_{n=1}^{N} n^{-n}\bigr)-N\log N-N\log \Gamma(N)+\frac{N(N+1)}{2}+\frac{1}{2}\log \Gamma(N)-N\log \sqrt{2\pi}\\
&-\frac{\gamma}{12}+\frac{5}{12}\log(N)+o(1),\,\, \text{par}\, \eqref{simplifie}\\
&=-2\log A-({N^2}+{N}+\frac{1}{6})\log N+\frac{N^2}{2}+N\log N+N\log \Gamma(N)\\
&+\frac{N(N+1)}{2}+\frac{1}{2}\log \Gamma(N)-N\log \sqrt{2\pi}-\frac{\gamma}{12}+\frac{5}{12}\log(N)+o(1),\,\, \text{par}\,\eqref{kinkelin}\\
&=-2\log A-({N^2}+{N}+\frac{1}{6})\log N+\frac{N^2}{2}+N\log N+N(N-\frac{1}{2})\log N-N^2\\
&+\frac{1}{2}\log(2\pi)N+\frac{1}{12}+\frac{N^2+N}{2}+\frac{1}{2}(N-\frac{1}{2})\log N-\frac{N}{2}+\frac{1}{4}\log(2\pi)-N\log(\sqrt{2\pi})\\
&+\frac{\gamma}{12}+\frac{5}{12}\log N+\mathrm{o}(\frac{1}{N^2})+o(1),\,\, \text{par}\,\eqref{devgamma}\\
&=-2\log A+\frac{1}{12}-\frac{\gamma}{12}+\frac{1}{4}\log (2\pi)+o(1)\\
&=2\zeta'(-1)-\frac{1}{12}-\frac{\gamma}{12}+\frac{1}{4}\log(2\pi)+o(1),\,\, \text{par}\, \eqref{kinkelinvaleur}.
\end{align*}
}

On en déduit la valeur de $\eta$ en zéro:
\begin{equation}\label{valeureta2}
\begin{split}
\eta(0)&=2\zeta'_\Q(-1)-\frac{1}{12}-\frac{\gamma}{12}+\frac{1}{4}\log (2\pi)\\
&-\zeta'_\Q(-1)-\zeta_\Q(-1)+\frac{1}{2}\zeta'_\Q(0)+\frac{1}{2}\log (2\pi)\zeta_\Q(0)\\
&=\zeta'_\Q(-1)-\frac{\gamma}{12}-\frac{1}{4}\log (2\pi)
\end{split}
\end{equation}

\end{proof}

%+\frac{1}{2}\log (2\pi)\zeta_\Q(2s)-\log( 2) \zeta_\Q(2s-1)-\zeta'_\Q(2s-1)+\frac{1}{2}\zeta'_\Q(2s)

On en déduit que
\[
\begin{split}
 \lim_{s\mapsto 0}\bigl(-B_\mathcal{N}(s)+P_\mathcal{N}(s)\bigr)&=\eta(0)-\frac{1}{2}\log (2\pi)\zeta_\Q(0)+\log( 2) \zeta_\Q(-1)+\zeta'_\Q(-1)-\frac{1}{2}\zeta'_\Q(0)\\
&=2\zeta_\Q'(-1)-\frac{\gamma}{12}+\frac{1}{4}\log (2\pi)-\frac{1}{12}\log 2.
\end{split}
\]

On pose, voir les notations avant  \eqref{regularisation}:
\[
F_\mathcal{N}(s):=-\frac{s^2}{\Gamma(s+1)}\zeta_\Q(2s+1)\int_0^\infty \frac{t^{s-1}}{2\pi i}\int_{\Lambda_c}\frac{e^{-z^2 t}}{-z}V_1(iz)dz dt, \quad \mathrm{Re}(s)>1,
\]
et on montre le lemme suivant,
\begin{lemma}
On a
\[
F_\mathcal{N}(s)=\frac{1}{\Gamma(s+1)}\Bigl(\frac{1}{2}+\gamma s+so(1) \Bigr)\frac{\Gamma(s+\frac{1}{2})}{12\sqrt{\pi}}(1-7s ),
\]
pour tout $|s|\ll1$ et on a
\[
F_{\mathcal{N}}(0)=\frac{1}{24}\quad\text{et} \quad F_{\mathcal{N}}'(0)=\frac{1}{12}\Bigl(\gamma-\log 2-\frac{7}{2}\Bigr).
\]
\end{lemma}
\begin{proof}

Soit $\mathrm{Re}(s)>1$, on a
\[
\begin{split}
F_{\mathcal{N}}(s)&=\frac{s^2}{\Gamma(s+1)}\zeta_\Q(2s+1)\biggl(\frac{3}{8}\int_0^\infty \frac{t^{s-1}}{2\pi i}\int_{\Lambda_c}\frac{e^{-\lambda t}}{-\lambda}\frac{1}{(1-\lambda)^\frac{1}{2}}d\lambda dt\\
&-\frac{7}{24}\int_0^\infty \frac{t^{s-1}}{2\pi i}\int_{\Lambda_c}\frac{e^{-\lambda t}}{-\lambda}\frac{1}{(1-\lambda)^\frac{3}{2}}d\lambda dt\biggr)\\
&=\frac{s^2}{\Gamma(s+1)}\zeta_\Q(2s+1)\biggl(\frac{3}{8}\int_0^\infty t^{s-1}\frac{1}{\pi} \sin(\pi \frac{1}{2})\Gamma(\frac{1}{2})\Gamma(\frac{1}{2},t)dt\\
&-\frac{7}{24} \int_0^\infty t^{s-1}\frac{1}{\pi} \sin(\pi \frac{3}{2})\Gamma(-\frac{1}{2})\Gamma(\frac{3}{2},t)dt\biggr)\\
&=\frac{s^2}{\Gamma(s+1)}\zeta_\Q(2s+1)\biggl(\frac{3}{8\sqrt{\pi}}\frac{1}{s}\Gamma(s+\frac{1}{2})-\frac{7}{12\sqrt{\pi}}\frac{1}{s}\Gamma(s+\frac{3}{2})\biggr)\\
&=\frac{s}{\Gamma(s+1)}\zeta_\Q(2s+1)\Bigl(\frac{3}{8\sqrt{\pi}}- \frac{7}{12\sqrt{\pi}}(s+\frac{1}{2}) \Bigr)\Gamma(s+\frac{1}{2})\\
&= \frac{1}{\Gamma(s+1)}\Bigl(\frac{1}{2}+\gamma s+so(1) \Bigr)\frac{\Gamma(s+\frac{1}{2})}{12\sqrt{\pi}}(1-7s ).
\end{split}
\]
où l'on a utilisé les formules suivantes:

\[
 \begin{split}
\frac{1}{2\pi i} \int_{\Lambda_c}\frac{e^{-\lambda t}}{-\lambda}\frac{1}{(1-\lambda)^a}d\lambda &= - \frac{1}{2\pi i}e^{-t} \int_{\Lambda_c-1}\frac{e^{-u t}}{1+u}\frac{1}{(-u)^a}du\\
&=\frac{1}{\pi} \sin(\pi a)\Gamma(1-a)\Gamma(a,t),
\end{split}
\]
et
\[
 \begin{split}
\int_{0}^\infty t^{s-1}\Gamma(a,t)dt &=\Bigl[\frac{t^s}{s}\Gamma(a,t) \Bigr]_0^\infty+\int_{0}^\infty \frac{t^s}{s}e^{-t} t^{a-1}dt\\
&=\frac{1}{s}\Gamma(a+s).
 \end{split}
\]

avec  $\Gamma(a,t)$ est la fonction Gamma incomplète. (Pour montrer cette formule, on suit la méthode  pour la représentation de l'inverse de $\Gamma$ comme une intégrale le long d'un contour, voir \cite[12.22]{Whittaker}  et on utilise la formule de \cite[8.353.3]{Table})).\\

%\[
 %\begin{split}
%F'(s)&=(o(1)+s\cdot o'(1))\frac{1}{\Gamma(s+1)}\bigl(\frac{3}{8\sqrt{\pi}}- \frac{7}{12\sqrt{\pi}}(s+\frac{1}{2}) \bigr)\Gamma(s+\frac{1}{2})\\
%&-(\frac{1}{2}+\gamma+s\cdot o(1) )\frac{\Gamma'(s+1)}{\Gamma(s+1)^2}\bigl(\frac{3}{8\sqrt{\pi}}- \frac{7}{12\sqrt{\pi}}(s+\frac{1}{2}) \bigr)\Gamma(s+\frac{1}{2})\\
%&-\frac{7}{12\sqrt{\pi}}(\frac{1}{2}+\gamma+s\cdot o(1) )\frac{1}{\Gamma(s+1)}\Gamma(s+\frac{1}{2})\\
%&+(\frac{1}{2}+\gamma+s\cdot o(1) )\frac{1}{\Gamma(s+1)}\bigl(\frac{3}{8\sqrt{\pi}}- \frac{7}{12\sqrt{\pi}}(s+\frac{1}{2}) \bigr)\Gamma'(s+\frac{1}{2})
% \end{split}
%\]

On conclut que $F_{\mathcal{N}}$ est analytique en un voisinage ouvert de  $s=0$,  et on a
\[
 F_{\mathcal{N}}(0)=\frac{1}{24}.
\]

%On note par la même lettre le prolongement analystique de $\zeta_\Q$ au voisinage de 0.\\

%Il existe une fonction de classe $\cl$ au voisinage de 0 qui soit un petit o de 1... on la notera $o(s)$. telle que

%\[
 %F(s)=\frac{1}{\Gamma(s+1)}(\frac{1}{2}+\gamma s+so(1) )\frac{\Gamma(s+\frac{1}{2})}{12\sqrt{\pi}}(1-7s )
%\]
Soit $s\neq 0$, on a
\[
 \begin{split}
F_{\mathcal{N}}'(s)=&  -\frac{\Gamma'(s+1)}{\Gamma(s+1)^2}\Bigl(\frac{1}{2}+\gamma s +s o(1) \Bigr)\frac{\Gamma(s+\frac{1}{2})}{12\sqrt{\pi}}(1-7s)\\
&+\frac{1}{\Gamma(s+1)}\Bigl(\gamma+o(1)+s o'(1)\Bigr)\frac{\Gamma(s+\frac{1}{2})}{12\sqrt{\pi}}(1-7s)\\
&+\frac{1}{\Gamma(s+1)}\Bigl(\frac{1}{2}+\gamma s +s o(1)\Bigr)\frac{\Gamma'(s+1)}{12\sqrt{\pi}}(1-7s)\\
&+\frac{1}{\Gamma(s+1)}\Bigl(\frac{1}{2}+\gamma s
+s o(1)\Bigr)\frac{\Gamma(s+1)}{12\sqrt{\pi}}(-7).
 \end{split}
\]
Rappelons, voir par exemple \cite[8.366.1, 8.366.2 ]{Table}, que
\[
 \Gamma'(1)=-\gamma,\quad \frac{1}{2}\frac{\Gamma'(\frac{1}{2})}{\Gamma(\frac{1}{2})}=-\log 2-\frac{\gamma}{2},
\]
 Par suite
\[
 F'_{\mathcal{N}}(0)=\frac{1}{12}(\gamma-\log 2-\frac{7}{2} ).
\]

\end{proof}

D'après \eqref{regularisation},  la fonction  $z_{\mathcal{N}}$ (voir \eqref{zetaN}) s'écrit au voisinage de $s=0$ sous la forme suivante:
\[
 z_{\mathcal{N}}(s)=\frac{s}{\Gamma(s+1)}\Bigl(\gamma A_{\mathcal{N}}(s)-B_\mathcal{N}(s)-\frac{1}{s}A_{\mathcal{N}}(s)+P_\mathcal{N}(s)\Bigr)+F_{\mathcal{N}}(s)+s^2 g(s),
\]
où $g$ est une fonction analytique au voisinage  zéro. Par les calculs précédents, on obtient:
\[
\begin{split}
 z_{\mathcal{N}}(0)&=-A_{\mathcal{N}}(0)+F_{\mathcal{N}}(0)\\
&=-\frac{1}{2}\zeta_\Q(-1)+\frac{1}{4}\zeta_\Q(0)+\frac{1}{24}\\
&=\frac{1}{24}-\frac{1}{8}+\frac{1}{24}\\
&=-\frac{1}{24}.
\end{split}
\]
et
\[
 \begin{split}
  z'_{\mathcal{N}}(s)&=
\frac{1}{\Gamma(s+1)}\Bigl(\gamma A_{\mathcal{N}}(s)-B_{\mathcal{N}}(s)-\frac{1}{s}A_{\mathcal{N}}(s)+P_{\mathcal{N}}(s)+\gamma sA_{\mathcal{N}}'(s)-sB_{\mathcal{N}}'(s)+\frac{1}{s}A_{\mathcal{N}}(s)\\
&-A'(s)+sP'(s)  \Bigr)-\frac{\Gamma'(s+1)}{\Gamma(s+1)^2}s\bigl( \gamma A_{\mathcal{N}}(s)-B_{\mathcal{N}}(s)-\frac{1}{s}A_{\mathcal{N}}(s) +P_{\mathcal{N}}(s)\bigr)\\
&+F'_{\mathcal{N}}(s)+\frac{d}{ds}(s^2 g(s)),
 \end{split}
\]
donc
\[
\begin{split}
z'_{\mathcal{N}}(0)&=-A'_{\mathcal{N}}(0)+\lim_{s\mapsto 0}(-B_{\mathcal{N}}(s) +P_{\mathcal{N}}(s))+ \lim_{s\mapsto 0}s(-B'_{\mathcal{N}}(s)+P'_{\mathcal{N}}(s) )+\frac{1}{12}\bigl(\gamma-\log 2-\frac{7}{2} \bigr)\\
&=-\zeta'_\Q(-1)-\frac{1}{4}\log (2\pi)
+2\zeta_\Q'(-1)-\frac{\gamma}{12}+\frac{1}{4}\log (2\pi)-\frac{1}{6}\log 2+\frac{1}{12}\bigl(\gamma-\log 2-\frac{7}{2} \bigr)\\
&=\zeta'_\Q(-1)-\frac{1}{6}\log 2-\frac{7}{24}.
\end{split}
\]

Rappelons que
\[
\zeta_{\mathcal{N}}(s)=2z_\mathcal{N}(s)+\sum_{k\geq 1}\frac{1}{{j'}_{0,k}^{2s}}.
\]
Si l'on pose $
 \zeta_{J'_0}(s)=\sum_{k\geq 1}\frac{1}{{j'}_{0,k}^{2s}}$ alors la contribution de $J'_0$ vaut
\[
 \zeta_{J'_0}(0)=-\frac{1}{2}(1+\frac{1}{2})=-\frac{3}{4},\quad \zeta'_{J'_0}(0)=\frac{1}{2}\log 2-\frac{1}{2}\log \pi.
\]
(C'est une application de \eqref{valeurszeta}, sachant que $I_0'(z)=\frac{z}{4}+o(z)$, pour $|z|\ll1$ et $I_0'(z)=\frac{e^z}{\sqrt{2\pi z}}(1+O(\frac{1}{z}))$ si $|z|\gg1$).\\

On obtient:
\[
 \zeta_{\mathcal{N}}(0)=-(2)\frac{1}{24}-\frac{3}{4}=-\frac{5}{6},
\]
et
\[
\begin{split}
 \zeta_{\mathcal{N}}'(0)&=2z'_N(0)+\bigl( \frac{1}{2}\log 2-\frac{1}{2}\log \pi \bigr)\\
&=2\zeta'_\Q(-1)+\frac{1}{6}\log 2 -\frac{1}{2}\log \pi -\frac{7}{12}.
\end{split}
\]
On conclut que,  voir \eqref{Dirichlet}
\[
\zeta_0(0)=-\frac{2}{3},\quad \zeta'_0(0)=4\zeta'_\Q(-1)-\frac{1}{6}+\frac{1}{3}\log 2.
\]
Ce qui termine la preuve du résultat cherché.

\end{proof}

\subsubsection{Régularisation de la fonction zêta  $\zeta_{\Delta_{\overline{\mathcal{O}(m)}_\infty}}$ pour $m\geq 1$}
Ce paragraphe est consacré à l'étude et au calcul explicite du déterminant régularisé de la fonction Zêta:
\[
 \zeta_{\Delta_{\overline{\mathcal{O}(m)}_\infty}},\quad \forall m\geq 1.
\]
L'idée de la démonstration du théorème \eqref{lavaleurdezeta} consiste en premier temps à introduire une suite de fonctions analytiques $ (G_n)_{n\in \Z}$ qui vérifie les hypothèses du \eqref{paragrapheG} et d'établir, tenant compte des notations de \eqref{paragrapheG} et du théorème \eqref{spectreOm}, que :
\begin{align*}
  \zeta_{\Delta_{\overline{\mathcal{O}(m)}_\infty}}&=4^s z_m(s),
\end{align*}
où
\begin{align*}
 z_m(s)&=2\sum_{n\geq 1}\zeta_{G_{n+m}}(s)+ 2\sum_{k=0}^{\frac{m}{2}-1}\zeta_{G_k}(s)+\zeta_{G_\frac{m}{2}}(s),\quad \text{si}\; m\;\text{pair},\\
 z_m(s)&=2\sum_{n\geq 1}\zeta_{G_{n+m}}(s)+ 2\sum_{k=0}^{\frac{m-1}{2}}\zeta_{G_k}(s),\quad\quad \quad\quad \quad \text{si}\; m\;\text{impair}.
\end{align*}
Puis de démontrer que $\bigl(\zeta_{G_{n+m}}\bigr)_{n\in \N^\ast}$ est une famille régularisable, voir définition \eqref{régularisable} et d'appliquer ensuite le théorème \eqref{regularisation} pour conclure.\\

On introduit, pour tout $n\in \Z$, la fonction suivante
\[
 G_n(z):=I_{n+1}(z)I_{n-m}(z)+I_n(z)I_{n-m-1}(z)\quad \forall\, z\in \C.
\]
où $I_\ast$ désigne la fonction de Bessel modifiée d'ordre $\ast$. \\

Etablissons d'abord la formule précédente reliant $\zeta_{\Delta_{\overline{\mathcal{O}(m)}_\infty}}$ et $z_m$. Cela résulte de la définition $\zeta_{\Delta_{\overline{\mathcal{O}(m)}_\infty}}$ et de celle $z_m$ et du lemme suivant:
\begin{lemma}
On a
\begin{enumerate}
\item Pour tout $n\in \Z$,
\[
G_n(z)=\frac{1}{i^{2n-m-1}}L_n(iz)\quad \forall\, z\in \C
\]
%\item Les zéros de la fonction $G_n(\sqrt{-\la})$ est l'ensemble des réels positifs dont la racine est un élement de $Z_n$, en effet, \[
%\begin{split}
%G_n(\sqrt{-\la})&=G_n(\pm i\sqrt{|\la|})\\
%&=\frac{1}{i^{2n-m-1}}\bigl(J_n(\mp \sqrt{|\la|})J_{n-m-1}(\mp \sqrt{|\la|})-J_{n+1}(\mp \sqrt{|\la|})J_{n-m}(\mp \sqrt{|\la|}) \bigr).
%\end{split}
%\]
et donc, $\{z\in \C\, |\, G_n(iz)=0\}=Z_n$ (voir \eqref{Zeros2}).
\item
\begin{equation}\label{symetrieG}
 G_n(z)=G_{m-n}(z)q\quad \forall\, z\in \C,
\end{equation}
\item
\begin{equation}\label{symetrieG}
 G_n(-z)=(-1)^{m+1}G_n(z)\quad \forall\, z\in \C.
\end{equation}
\item Si $n\geq m+1$,
\[
G_n(z)=\frac{z^{2n-m-1}}{2^{2n-m-1}\Gamma(n+1)\Gamma(n-m)}+o(z^{2n-m-1})\quad |z|\ll1.
\]
\end{enumerate}
\end{lemma}
\begin{proof}
 Cela découle directement des propriétés des fonctions de Bessel. En effet,  1) résulte de la définition de $I_n$ et on a
\[
 \begin{split}
  G_{n}(z)&=I_{n+1}(z)I_{n-m}(z)+I_{n}(z)I_{n-m-1}(z)\\
&=I_{-n-1}(z)I_{-n+m}(z)+I_{-n}(z)I_{-n+m+1}(z),\quad \cite[9.6.6 ]{Table2}\\
&=I_{(m-n)+1}(z)I_{(m-n)-m}(z)+I_{(m-n)}(z)I_{(m-n)-m-1}(z)\\
&=G_{m-n}(z).
 \end{split}
\]

\end{proof}

Avec les notations du paragraphe \eqref{paragrapheG}, on considère la famille de fonctions Zêta suivante \[\bigl\{\zeta_{\widetilde{G}_{n}}\,\bigl|\, n\geq 1 \bigr\},\]
 où $\widetilde{G}_{n}:=G_{n+m}$.\\%{\footnotesize Rappelons que $G_n$ vérifie la propriété \eqref{symetrieG}}.\\

  On se propose  d'étudier la fonction
\[
\zeta_{\geq m+1}(s):=\sum_{n\geq 1}\zeta_{\widetilde{G}_{n}}.
\]
Plus précisément, on va montrer que  $\{\zeta_{G_{n+m}}\,|\, n\geq 1 \}$ est régularisable, voir définition \eqref{régularisable}, et on déterminera les valeurs  $\zeta_{\geq m+1}(0)$ et $\zeta_{\geq m+1}'(0)$ ce qui nous permettra de  compléter la démonstration de \eqref{lavaleurdezeta}.

Commençons par montrer que cette famille de fonctions  vérifie  les hypothèses de la définition \eqref{régularisable}.
C'est à dire, on doit vérifier que
\[
p_n(z):=\frac{\widetilde{G}_n(iz)}{l_n z^{e_n}}=\prod_{k\geq 1}^\infty\Bigl(1-\frac{z^2}{\la_{n,k}^2} \Bigr)\quad \forall\, z\in \C,
\]
en plus des conditions de la définition \eqref{régularisable} où $l_n, e_n, a_n,b_n,w,r$ seront calculés explicitement (Notons que $l_n$ et $e_n$ sont déjà donnés dans le lemme précédent).\\

Montrons le premier point, c'est à dire que
\[
\widetilde{G}_n(iz)=l_n z^{e_n}\prod_{k\geq 1}^\infty\Bigl(1-\frac{z^2}{\la_{n+m,k}^2} \Bigr),
\]
 Il suffit, par le théorème de factorisation de Weierstrass,  de montrer que
\[
\sum_{k=1}^\infty\frac{1}{\la_{n,k}^2}<\infty\quad\forall\, n\in \N^\ast.
\]
On va montrer que c'est une conséquence de
\[
\sum_{k=1}^\infty\frac{1}{j_{n,k}^2}<\infty\quad \text{et}\quad \sum_{k=1}^\infty\frac{1}{j_{n+m,k}^2}<\infty,
\]
où $j_{\ast,k}$ désigne le $k$-ème zéro positif de la fonction Bessel $J_\ast$ (Rappelons que la convergence de ces
sommes est un résultat classique de la théorie des fonctions de Bessel, voir \cite[§ 15]{Watson} ).\\

Pour cela, on a besoin du lemme technique suivant:
\begin{lemma}
Soient $f$ et $g$ deux fonctions réelles de classe $\mathcal{C}^1$ définies sur un intervalle ouvert non vide $I$ tel que $0\notin I$.

On suppose que les zéros de $f$ et $g$ sont simples (c'est à dire, si$f(\al)=0$ alors $f'(\al)\neq 0$). Notons par $Z(f)$ (resp. $Z(g)$) l'ensemble des zéros de $f$ (resp. de $g$) sur $I$ et suppose en plus que
\[
Z(f)\cap Z(g)=\emptyset.
\]
Soit $m\in \N$ et on pose
\[
L(x)=x^m\frac{d}{dx}\bigl(x^{-m}f(x)g(x) \bigr)\quad \forall\, x\in I.
\]
Si $\la_1<\la_2$ sont deux zéros consécutifs de $L$ alors
\[
\mathrm{Card}\Bigl(]\la_1,\la_2[ \cap\bigl(Z(f)\cup Z(g) \bigr)\Bigr)\leq 1.
\]
\end{lemma}
\begin{proof}
Commençons par remarquer que $L$ s'annule entre deux éléments de $Z(f)\cup Z(g)$ et que $Z(L)\cap\bigl(Z(f)\cup Z(g) \bigr)=\emptyset$.\\

Supposons qu'il existe $\al\in \bigl(Z(f) \cup Z(g)\bigr)\cap ]\la_1,\la_2[$ qui soit minimal pour cette propriété et que $f(\al)=0$. Supposons par l'absurde qu'il existe $\beta\in Z(g)$ tel que $\la_1<\al<\beta<\la_2$ minimal pour cette propriété, donc  $g$ est de signe constant sur $[\la_1,\beta]$. On a
\begin{align*}
L(\al)&=-m\al^{-1}f(\al) g(\al)+f'(\al)g(\al)+f(\al)g'(\al)=f'(\al)g(\al),\\
L(\beta)&=-m\al^{-1}f(\beta) g(\beta)+f'(\beta)g(\beta)+f(\beta)g'(\beta)=g'(\beta)f(\beta),
\end{align*}
Comme $L$ est de signe constant sur $]\la_1,\la_2[$ (puisque continue), alors
\[
f'(\al)g(\al)g'(\beta)f(\beta)>0,
\]
ce qui est impossible puisqu'on montre que
\begin{align*}
f'(\al)f(\beta)&=\lim_{x\mapsto 0^+}\frac{f(\al+x)f(\beta)}{x}>0,\\
g'(\beta)g(\al)&=\lim_{x\mapsto 0^+}\frac{g(\beta-x)g(\al)}{-x}<0.\\
\end{align*}
On conclut que soit $\la_1<\al<\la_2<\beta$ ou bien il existe $\al'\in Z(f)$ tel que $\la_1<\al<\al'<\la_2$. Si l'on est dans la dernière situation, on obtient comme avant:
\[
f'(\al)g(\al)f'(\al')g(\al')>0,
\]
Puisque $g$ est de signe constant sur $[\al,\al']$ alors
\[
f'(\al)f'(\al')>0,
\]
ce qui impossible puisque
\begin{align*}
f'(\al)f'(\al')&=\lim_{x\mapsto 0^+}\frac{f(\al+x)f(\al'-x)}{-x^2}<0.
\end{align*}

\end{proof}

On va appliquer ce lemme à $f=J_n$ et $g=J_{n+m}$, dans ce cas
\[
 L=-L_{n+m}.
\]
 On obtient donc avec les notations introduites avant:
\[
\mathrm{Card}\Bigl( [\la_{{}_{{n+m},k}},\la_{{}_{{n+m},k+1}}\bigr]\cap \bigl(Z(J_n)\cup Z(J_{n+m})\bigr)\Bigr)\leq 1,\quad \forall k\geq 1.
\]
Mais d'après \eqref{deriveeLnormeCopie} on a
\[
L'_{n+m}(\la)=\frac{J_{n}(\la)}{J_{n+m}(\la)}\|\vf_\la\|^2,\quad \forall \la\in Z_{n+m}.
\]
qui donne que
\[
0>L'_{n+m}(\la_{{}_{n+m,k}})L'_{n+m}(\la_{{}_{n+m,k+1}})=\frac{J_{n+m}(\la_{{}_{n+m,k}})}{J_{n}(\la_{{}_{n+m,k}})}\frac{J_{n+m}(\la_{{}_{n+m,k+1}})}{J_{n}(\la_{{}_{n+m,k+1}})}\|\vf_{\la_{{}_{n+m,k+1}}}\|^2\|\vf_{\la_{{}_{n+m,k+1}}}\|^2.
\]
(l'inégalité à droite résulte du fait que $L_{n+m}$ est $\mathcal{C}^1$ et de signe constant sur $]\la_{{}_{n+m,k}} ,\la_{{}_{n+m,k+1}}[$).\\

Par suite
\[
\mathrm{Card}\Bigl( [\la_{{}_{n+m,k}} ,\la_{{}_{n+m,k+1}}]\cap \bigl(Z(J_n)\cup Z(J_{n+m})\bigr)\Bigr)= 1\quad
\forall\, k\geq 1.
\]
Maintenant on peut montrer que
\[
\sum_{k=1}^\infty \frac{1}{\la_{n,k}^2}<\infty,
\]
En posant \[\widetilde{\la}_{n,k}:=\max\Bigl\{\bigl(Z(J_n)\cup Z(J_{n+m})\bigr)\cap [0,\la_{n,k}] \Bigr\}\quad \forall\, k\in \N,\] alors,  par le
lemme précédent, on a
\[
\widetilde{\la}_{n,k}<\la_{n,k} <\widetilde{\la}_{n,k+1}<\la_{n,k+1}\quad \forall\, k\geq 1.
\]
et
\[
\sum_{k=1}^\infty \frac{1}{\widetilde{\la}_{n,k}^{2}}\leq \sum_{k=1}^\infty \frac{1}{j_{n,k}^2}+\sum_{k=1}^\infty
\frac{1}{j_{n+m,k}^2}.
\]
Or le terme de droite dans cette inégalité est fini, voir \cite[§. 15]{Watson}.
On conclut que\[
\sum_{k=1}^\infty \frac{1}{\la_{n,k}^2}<\infty\quad \forall \,n\in \Z.
\]

%Comme la famille $\{G_n\,|\, n\in \Z\}$ vérifie  la propriété \eqref{symetrieG}, on va considérer dans la suite de fonctions suivante

Vérifions maintenant  \eqref{lowerbound}. On note par $\la_{n+m,1}$ la première valeur positive non nulle de $Z_{n+m}$,
montrons que
\[
\inf_{n\in \N^\ast}\frac{\la_{n+m,1}^2}{n^2}\neq 0.
\]

Montrons d'abord le lemme suivant:
\begin{lemma}\label{petitlemme}
Soit $a$ un réel non nul. On a
\[
\int_0^1 xJ_n(ax)^2 dx=\frac{1}{2}\bigl(J_n'(a)^2+(1-\frac{n^2}{a^2})J_n(a)^2 \bigr)
\]
en particulier, soit $n\neq 0$, si l'on note par $j_{n,1}$ (resp. $j'_{n,1}$) le premier zéro positif non nul de $J_n$
(resp. de $J'_n$) alors on a
\[
j_{n,1}>j'_{n,1}>n.
\]
%{\footnotesize Rappelons qu'on a noté dans le texte $j_{n,k}'$ par $c_{n,k}$. }
\end{lemma}
\begin{proof}
Soit $b$ un réel non nul, on a
\[
\begin{split}
\int_0^1xJ_n(ax)J_n(bx)dx&=\frac{1}{b^2-a^2}\bigl(aJ_n(b)J_n'(a)-bJ_n(a)J_n'(b)  \bigr)\\
&=\frac{1}{a+b}\Bigl(a\frac{J_n(b)-J_n(a)}{b-a}J_n'(a)+J_n(a)J_n'(b)-aJ_n(a)\frac{J_n'(b)-J_n'(a)}{b-a}  \Bigr).\\
\end{split}
\]
La première égalité résulte de l'équation de Bessel, ou consulter \cite[§ 5.11 (8)]{Watson}.
 En faisant tendre $b$ vers $a$, on obtient:
\[
\int_0^1x J_n(ax)^2dx=\frac{1}{2a}\bigl(-aJ_n'(a)^2+J_n(a)J_n'(a)+aJ_n(a)J_n''(a)  \bigr),
\]
mais comme $J_n$ vérifie l'équation de Bessel, alors
\begin{equation}\label{encoreeq}
\int_0^1x J_n(ax)^2dx=\frac{1}{2}\Bigl(J_n'(a)^2+(1-\frac{n^2}{a^2})J_n(a)^2  \Bigr).
\end{equation}

Si $J_n'(a)=0$ alors on a de la formule précédente,  $|a|>n$. En particulier $j'_{n,1}>n$. Par définition $j_{n,1}>0$ et $J_n(j_{n,1})=0$, par suite    $J'_n$ s'annule sur l'intervalle ouvert $]0,j_{n,1}[$ par le théorème des accroissements finis, ce qui termine la preuve du lemme.
\end{proof}

Soit $\la >0$, on pose
\begin{equation}
f_{n,\la}(r) =
\begin{cases}
 J_{n+m}(\la r)  & \text{si } r \leq 1,\\
\frac{J_{n+m}(\la)}{J_{n}(\la)}r^{m}J_{n}(\frac{\la}{r}) & \text{si } r>1, \\
\end{cases}
\end{equation}
%où $Z_{n+m}\cap\R^+=\{\la_{n,1}<\la_{n,2}<\ldots\}$.

\begin{proposition}
Avec les notations précédentes,
\[
\begin{split}
\mathcal{I}_{n,\la}:=\int_0^\infty |f_{n,\la}(r)|^2\frac{rdr}{\max(1,r)^{m+4}}&=\frac{1}{2}\Bigl(J_{n+m}'(\la)^2+(1-\frac{(n+m)^2}{\la^2})J_{n+m}(\la)^2 \Bigr)\\
&+\frac{J_{n+m}(\la)^2}{2J_{n}(\la)^2}\bigl(J'_{n}(\la)^2+(1-\frac{n^2}{\la^2})J_{n}(\la)^2 \bigr),
\end{split}
\]
et on a
\[
\la_{n+m,1}^2\geq n^2\quad\forall\, n\in\N^\ast.
\]
\end{proposition}
\begin{proof}
On a
\[
\begin{split}
\int_0^\infty \bigl|f_{n,i}(r)\bigr|^2\frac{rdr}{\max(1,r)^{m+4}}&=\int_0^1 \bigl|f_{n,i}(r)\bigr|^2 rdr+\int_1^\infty \bigl|f_{n,i}(r)\bigr|^2\frac{rdr}{r^{m+4}}\\
&=\int_0^1 J_{n+m}(\la r) ^2 rdr+\biggl(\frac{J_{n+m}(\la)}{J_{n}(\la)}\biggr)^2\int_1^\infty J_{n}\Bigl(\frac{\la}{r}\Bigr)  ^2d\Bigl(\frac{1}{r^{3}}\Bigr).
\end{split}
\]
En utilisant \eqref{encoreeq}, on montre la première assertion.\\

 Comme $Z_{n+m}$ est symétrique, il suffit de montrer que $\inf Z_{n+m}\cap \R^+ > n$.\\

Si $\la \in Z_{n+m}\cap \R^+$, alors
\begin{equation}\label{encoreeq3}
-\frac{m}{\la} +\frac{J'_{n+m}(\la)}{J_{n+m}(\la)}+\frac{J_{n}'(\la)}{J_{n}(\la)}=0
\end{equation}
 puisque $\frac{d}{dr}(r^{-m}J_{n+m}(r)J_{n}(r) )_{|_\lambda}=0$,  voir \eqref{Zeros1} et que $\frac{d}{dr}(r^{-m}J_n(r)J_{n-m}(r) )=$\\
 $-mr^{-m-1}J_{n+m}(r)J_{n}(r)+r^{-m}\bigl(J_{n+m}'(r)J_n(r)+J_{n+m}(r)J'_{n}(r) \bigr)$.\\

Supposons par l'absurde que $\la< n$. On a $J_{n+m}(0)=J_{n}(0)=0$ et puisque $J_n$ et $J_{n+m}$ sont analytiques sur $\C$ alors il existe $0<\eps\ll 1$ tel que $J_nJ'_n>0$ et $ J_{n+m}J'_{n+m}>0$ sur $]0,\eps[$ et par conséquent $J_n^2$ et $J_{n+m}^2$ sont croissantes sur $[0,\eps[$, mais d'après \eqref{petitlemme}, $J_nJ'_n$ et $ J_{n+m}J'_{n+m}$  sont de signe constant  sur $]0,\min( n[$ donc on peut prendre $\eps\geq  n$.

Comme on a $\la< \min\bigl( |{n+m}|,|n| \bigr)\leq \eps$, alors
\[
\begin{split}
\frac{m^2}{\la^2}&=\biggl(\frac{J'_n(\la)}{J_n(\la)}+\frac{J_{n+m}'(\la)}{J_{n+m}(\la)} \biggr)^2\quad \text{par}\quad \eqref{encoreeq3}\\
&\geq \biggl(\frac{J'_n(\la)}{J_n(\la)}\biggr)^2+\biggl(\frac{J_{n+m}'(\la)}{J_{n+m}(\la)}\biggr)^2.
\end{split}
\]

Cela donne
{\allowdisplaybreaks
\begin{align*}
0& <\frac{\mathcal{I}_{n,i}}{J_{n+m}(\la)^2}\\
&=\frac{1}{2}\Biggl( \frac{J_{n+m}'(\la)^2}{J_{n+m}(\la)^2}+\Bigl(1-\frac{{(n+m)}^2}{\la^2}\Bigr) \Biggr)+\frac{1}{2}\Biggl(\frac{J'_{n}(\la)}{J_{n}(\la)^2}+\Bigl(1-\frac{n^2}{\la^2}\Bigr) \Biggr)\\
&=\frac{1}{2}\biggl(\frac{J_{n+m}'(\la)^2}{J_{n+m}(\la)^2}+\frac{J'_{n}(\la)}{J_{n}(\la)^2} \biggr)+1-\frac{{(n+m)}^2+n^2}{2\la^2}\\
&\leq \frac{1}{2}\biggl(\frac{J'_{n+m}(\la)}{J_{n+m}(\la)}+\frac{J_{n}'(\la)}{J_{n}(\la)} \biggr)^2+1-\frac{{(n+m)}^2+n^2}{2\la^2}\\
&=\frac{m^2}{2\la^2}+1-\frac{n^2+(n+m)^2}{2\la^2}\\
&=\frac{-n^2-nm}{\la^2}+1\\
\end{align*}}
donc,
\[
\la^2>n^2+nm=n(n+m).
\]
Ce qui est impossible puisque $\la<n$ et $m\geq 0$.
\end{proof}

On conclut que
\[
\inf_{n\in \N}\frac{\la_{n+m,1}^2}{n^2}\geq 1.
\]

\begin{proposition}
On a pour tout $|z|\gg 1$   avec $|\arg(z)|<\frac{\pi}{2}$,
\begin{equation}\label{developpementGn}
\widetilde{G}_n(z)=G_{n+m}(z)=\frac{e^{2z}}{2\pi z}\biggl(1-\frac{4n^2+2m^2+4nm+2m+1}{2z}+O(\frac{1}{z^2})\biggr),
\end{equation}

\end{proposition}

\begin{proof}
C'est une conséquence de la formule \eqref{dvlptz1}.
\end{proof}

%\begin{remarque}
%Comme $|G_n(-z)|=|G_n(z)|$ voir \eqref{symetrieG} et par la proposition précédente, $G_n$ vérifie   \eqref{G}.
%\end{remarque}

\begin{proposition}\label{developpementG}
On a pour tout $|\arg(z)|<\frac{\pi}{2}$, et $n\gg1$
\[
 \begin{split}
\widetilde{G}_n(nz)=G_{n+m}(nz)&= \frac{e^{(2n+m-1)\eta(z)}}{2\pi n(1+z^2)^\frac{1}{2}}\Bigl(e^{2(\eta(z)-z\eta'(z))}+1\Bigr)\Biggl(1+\frac{1}{n}w_m(z)+\frac{1}{n^2}\eta_m(z,n) \Biggr),
 \end{split}
\]
avec $w_m(z)=-\frac{1}{4}\bigl(2m^2+2m+1\bigr)\frac{1}{(1+z^2)^{\frac{1}{2}}}+\frac{1}{12}\frac{1}{(1+z^2)^{\frac{3}{2}}}$ et $\eta_m$ une fonction bornée uniformément en $n$.
\end{proposition}
Donc, on obtient avec les notations du \eqref{régularisable}:
\begin{align*}
 r(z)&=\eta(z)=\sqrt{1+z^2}+\log \frac{z}{1+\sqrt{1+z^2}},\\
\rho(z)&=\frac{1}{(1+z^2)^\frac{1}{2}}\Bigl(e^{2(\eta(z)-z\eta'(z))}+1\Bigr),\\
w(z)&=w_m(z).
\end{align*}
et on vérifie à l'aide de developpements limités convenables que
\begin{align*}
\eta(z)&=\log(z)+O(1)\quad \text{pour}\;|z|\ll1,\\
 \eta(z)&=z+O(\frac{1}{z})\quad\text{pour}\;|z|\gg1,\\
\eta(z)-z \eta'(z)&=\log (z)\quad\text{pour}\;|z|\ll1 \\
\eta(z)-z \eta'(z)&=z+1+O(\frac{1}{z})\quad\text{pour}\;|z|\gg1\\
\rho(0)&=1\\
\log\rho(z)&=2z-\log z+2+O\bigl(\frac{1}{z}\bigr)\quad \text{pour}\; |z|\gg 1.
\end{align*}
La preuve de cette proposition sera faite en deux étapes; on commence par montrer les résulats suivants:
\begin{lemma}
Soit $n>0$ et $l$ un entier fixé. Il existe des fonctions $v_{1,l},v_{2,l},\ldots $ telles que pour $n$ grand et $|\arg(z)|\leq \frac{\pi}{2}-\eps$, $\eps$ assez petit:
\begin{equation}\label{besselasym}
I_{n}((n+l)z)=\frac{1}{\sqrt{2\pi n}}\frac{e^{n\eta(z)}}{ (1+z^2)^\frac{1}{4}}e^{lz\eta'(z)}\Bigl(1+\sum_{k=1}^\infty\frac{v_{k,l}(z)}{n^k}\Bigl)
\end{equation}
avec $ v_{1,l}(z)=u_1(z)+zl(1+z^2)^\frac{1}{4}\frac{\pt}{\pt z}\bigl(\frac{1}{(1+z^2)^\frac{1}{4}} \bigr)+\frac{l^2z^2}{2}\eta''(z)$.

\end{lemma}
\begin{proof}
Dans \cite[9.7.7]{Table2}, on a

\begin{equation}\label{bessel1}
I_{n}(nz)=\frac{1}{\sqrt{2\pi n}}\frac{e^{n\eta(z)}}{ (1+z^2)^\frac{1}{4}}\Bigl(1+\sum_{k=1}^\infty\frac{u_{k}(z)}{n^k}\Bigl)
\end{equation}
pour tout $z$ vérifiant $|\arg(z)|\leq \frac{\pi}{2}-\eps$. \\

On cherche à établir une formule analogue pour $I_n((n+l)z)$ pour $n\gg1$ et $l$ un entier fixé, c'est à dire montrer qu'il existe des fonctions continues $v_{l,1},v_{l,2},\ldots$ telles que
\[
I_{n}((n+l)z)=\frac{1}{\sqrt{2\pi n}}\frac{e^{n\eta(z)}}{ (1+z^2)^\frac{1}{4}}e^{lz\eta'(z)}\Bigl(1+\sum_{k=1}^\infty\frac{v_{k,l}(z)}{n^k}\Bigl)
\]
%Pour les applications on explicitera le terme $v_{1,l}$.\\

%Pour simplifier les calculs on écrira \eqref{bessel1} sous la forme:
%\begin{equation}\label{bessel2}
%I_{n}(nz)=\frac{1}{\sqrt{2\pi n}}\frac{e^{n\eta(z)}}{ (1+z^2)^\frac{1}{4}}\Bigl(1+\frac{1}{n}u_{1}(z) +\frac{1}{n^2}\mathrm{O}(1)\Bigl)
%\end{equation}
Pour simplifier les notations, on notera par $\mathrm{O}(1)$ une fonction en $n$ bornée uniformément en $z$.\\

Soit $n\gg1$.\\

Soit $z\in \C$ tel que $|\arg(z)|\leq \frac{\pi}{2}-\eps$, remarquons que $ \arg((1+\frac{l}{n})z)=\arg(z)$, donc de  \eqref{bessel1} on tire
\[
 \begin{split}
  &I_n\bigl((n+l)z\bigr)=I_n\bigl(n(1+\frac{l}{n})z\bigr)\\
&=\frac{1}{\sqrt{2\pi n}}\frac{e^{n\eta((1+\frac{l}{n})z)}}{ (1+((1+\frac{l}{n})z)^2)^\frac{1}{4}}\biggl(1+\frac{1}{n}u_1((1+\frac{l}{n})^2z)+\sum_{k=2}^\infty\frac{u_{k}((1+\frac{l}{n})z)}{n^k}\biggl)\\
&=\frac{1}{\sqrt{2\pi n}}e^{n(\eta(z)+\frac{lz}{n}\eta'(z)+\frac{z^2l^2}{2n^2}\eta''(z)+o(\frac{1}{n^2} ))}\biggl(\frac{1}{(1+z^2)^\frac{1}{4}}+\frac{lz}{n}\frac{\pt }{\pt z}\bigl( \frac{1}{(1+z^2)^\frac{1}{4}})+\frac{1}{(n+l)^2}O(1)  \biggr)\\
&\biggl(1+\frac{1}{n}\Bigl(u_1(z)+\frac{l}{n}\frac{\pt}{\pt z}u_1(z)+\frac{1}{n^2}\mathrm{O}(1)\Bigr)+\frac{1}{n^2}\mathrm{O}(1)  \biggr)\\
&=\frac{e^{n\eta(z)}}{\sqrt{2\pi n}}e^{lz\eta'(z)}\biggl(1+\frac{l^2 z^2}{2n}\eta''(z)+\frac{1}{n^2}\mathrm{O}(1) \biggr)\biggl(\frac{1}{(1+z^2)^\frac{1}{4}}+\frac{lz}{n}\frac{\pt }{\pt z}\bigl( \frac{1}{(1+z^2)^\frac{1}{4}}\bigr)+\frac{1}{(n+l)^2}O(1)  \biggr)\\
&\biggl(1+\frac{1}{n}\Bigl(u_1(z)+\frac{l}{n}\frac{\pt}{\pt z}u_1(z)+\frac{1}{n^2}\mathrm{O}(1)\Bigr)+\frac{1}{n^2}\mathrm{O}(1)  \biggr)\\
&=\frac{e^{n\eta(z)}}{\sqrt{2\pi n}}e^{lz\eta'(z)}\biggl(\frac{1}{(1+z^2)^\frac{1}{4}}+\frac{1}{n}\Bigl(\frac{u_1(z)}{(1+z^2)^\frac{1}{4}}+zl\frac{\pt}{\pt z}\bigl(\frac{1}{(1+z^2)^\frac{1}{4}} \bigr)+\frac{l^2z^2}{2}\frac{\eta''(z)}{(1+z^2)^\frac{1}{4}} \Bigr) +\frac{1}{n^2}\mathrm{O}(1)\biggr)
 \end{split}
\]

En posant
\[
 v_{1,l}(z)=u_1(z)+zl(1+z^2)^\frac{1}{4}\frac{\pt}{\pt z}\bigl(\frac{1}{(1+z^2)^\frac{1}{4}} \bigr)+\frac{l^2z^2}{2}\eta''(z)
\]
on trouve la formule cherchée.
\end{proof}

\begin{Corollaire}
Soient $l$ et $l'$ deux entiers fixés. On a
\begin{equation}\label{bessel2}
I_{n-l}(nz)I_{n-l'}(nz)=\frac{e^{2n\eta(z)}}{2\pi n(1+z^2)^\frac{1}{2}}e^{-(l+l')(\eta(z)-z\eta'(z))}\Bigl(1+\frac{1}{n}\bigl(v_{1,l}(z)+v_{1,l'}(z)+\frac{l+l'}{2}\bigr) +\frac{1}{n^2}\mathrm{O}(1) \Bigr)
\end{equation}
\end{Corollaire}

\begin{proof}
On a pour $n\gg1$
\[
 \frac{\sqrt{n}}{\sqrt{n-l}}=1+\frac{l}{2n}+\frac{1}{n^2}\mathrm{O}(1),
\]
et par ce qui précède, on a
\[
 \begin{split}
&I_{n-l}(nz)=I_{n-l}\bigl(((n-l)+l)z \bigr)\\
&=\frac{e^{{(n-l)}\eta(z)}}{\sqrt{2\pi(n-l)}(1+z^2)^\frac{1}{4}}e^{zl\eta'(z)}\biggl(1+\frac{1}{n-l} v_{1,l}(z)+\frac{1}{(n-l)^2}\mathrm{O}(1) \biggr)\\
&=\frac{e^{{(n-l)}\eta(z)}}{\sqrt{2\pi n}(1+z^2)^\frac{1}{4}}e^{zl\eta'(z)}\frac{\sqrt{n}}{\sqrt{n-l}}\biggl(1+\frac{1}{n}v_{1,l}(z)+\frac{1}{n^2}\mathrm{O}(1)\biggr)\\
&=\frac{e^{{(n-l)}\eta(z)}}{\sqrt{2\pi n}(1+z^2)^\frac{1}{4}}e^{zl\eta'(z)}\biggl(1+\frac{l}{2n}+\frac{1}{n^2}\mathrm{O}(1)\biggr)\biggl(1+\frac{1}{n}v_{1,l}(z)+\frac{1}{n^2}\mathrm{O}(1)\biggr)\\
&=\frac{e^{{(n-l)}\eta(z)}}{\sqrt{2\pi n}(1+z^2)^\frac{1}{4}}e^{zl\eta'(z)}\biggl(1+\frac{1}{n}\Bigl(v_{1,l}(z)+\frac{l}{2}\Bigr) +\frac{1}{n^2}\mathrm{O}(1) \biggr).
\end{split}
\]

Si l'on prend un autre entier $l'$ alors on déduit directement  de la dernière formule que
\begin{equation}
I_{n-l}(nz)I_{n-l'}(nz)=\frac{e^{2n\eta(z)}}{2\pi n(1+z^2)^\frac{1}{2}}e^{-(l+l')(\eta(z)-z\eta'(z))}\biggl(1+\frac{1}{n}\bigl(v_{1,l}(z)+v_{1,l'}(z)+\frac{l+l'}{2}\bigr) +\frac{1}{n^2}\mathrm{O}(1) \biggr).
\end{equation}
\end{proof}

On pose
\[
 w_{l,l'}(z):=v_{1,l}(z)+v_{1,l'}(z)+\frac{l+l'}{2}.
\]

De \eqref{bessel2}, on a

\[
 I_{n+m+1}(nz)I_{n}(nz)=\frac{e^{2n\eta(z)}}{2\pi n(1+z^2)^\frac{1}{2}}e^{(m+1)(\eta(z)-z\eta'(z))}\biggl(1+\frac{1}{n}\Bigl(v_{1,-(m+1)}+v_{1,0}-\frac{m+1}{2}\Bigr)+\mathrm{O}(\frac{1}{n^2}) \biggr)
\]
et

\[
I_{n+m}(nz)I_{n-1}(nz)=\frac{e^{2n\eta(z)}}{2\pi n(1+z^2)^\frac{1}{2}}e^{(m-1)(\eta(z)-z\eta'(z))}\biggl(1+\frac{1}{n}\Bigl(v_{1,-m}+v_{1,1}-\frac{m-1}{2}\Bigr)+\mathrm{O}(\frac{1}{n^2}) \biggr)
\]

donc
\[
 \begin{split}
&G_{n+m}(nz)=\frac{e^{2n\eta(z)}}{2\pi n(1+z^2)^\frac{1}{2}}e^{(m+1)(\eta(z)-z\eta'(z))}\biggl(1+e^{-2(\eta(z)-z\eta'(z))}+\frac{1}{n}\Bigl(v_{1,-(m+1)}+v_{1,0}-\frac{m+1}{2}\Bigr)\\
&+\frac{e^{-2(\eta(z)-z\eta'(z))}}{n}(v_{1,-m}+v_{1,1}-\frac{m+1}{2})+\mathrm{O}(\frac{1}{n^2}) \biggr)\\
&= \frac{e^{2n\eta(z)}}{2\pi n(1+z^2)^\frac{1}{2}}e^{(m+1)(\eta(z)-z\eta'(z))}(1+e^{-2( \eta(z)-z\eta'(z)})\biggl(1+\\
&\frac{1}{n}\Bigl(\frac{v_{1,-(m+1)}(z)+v_{1,0}(z)-\frac{m+1}{2}+(v_{1,-(m+1)}(z)+v_{1,0}(z)-\frac{m-1}{2})e^{-2(\eta(z)-z\eta'(z))}}{1+e^{-2(\eta(z)-z\eta'(z))}} \Bigr)+\mathrm{O}(\frac{1}{n^2}) \biggr)
 \end{split}
\]
Rappelons que
\[
\begin{split}
v_{1,l}(z)&=u_1(z)+lz(1+z^2)^\frac{1}{4}\frac{\pt }{\pt z}\bigl( \frac{1}{(1+z^2)\frac{1}{4}}\bigr)+\frac{l^2z^2}{2}\eta''(z).\\
\end{split}
\]
avec $u_1(z)=\frac{1}{8(1+z^2)^\frac{1}{2}}-\frac{5}{24(1+z^2)^\frac{3}{2}
}$.\\
Si l'on pose
\[
 p:=\frac{1}{(1+z^2)^\frac{1}{2}},
\]
on trouve  que
\[
\begin{split}
&\frac{v_{1,-(m+1)}(z)+v_{1,0}(z)-\frac{m+1}{2}+(v_{1,-(m+1)}(z)+v_{1,0}(z)-\frac{m-1}{2})e^{-2(\eta(z)-z\eta'(z))}}{1+e^{-2(\eta(z)-z\eta'(z))}}=\\
&=-\frac{1}{4}\bigl(2m^2+2m+1\bigr)p+\frac{1}{12}p^3,
\end{split}
\]
qu'on notera par $w_m$, ce qui termine la preuve de la proposition \eqref{developpementG}.\\

\begin{remarque}
\rm{ Si l'on prend $m=0$, c'est à dire le cas correspondant au fibré trivial, on retrouve bien le résultat précédemment calculé.}
\end{remarque}

\begin{remarque}\rm{\[
w_0(z)=V_1(z)+U_1(z),
\]
voir notations \eqref{developpement1} et \eqref{developpement2}.\\
}
\end{remarque}

On considère la fonction zêta suivante:
\[
\zeta_{\geq m+1}:=\sum_{n\geq 1}\sum_{\la\in Z_{n+m}}\frac{1}{\la^{2s}},
\]
Rappelons que c'est la fonction zêta associée à la famille $\bigl\{G_{n+m}\bigl|\; n\geq 1 \bigr\}$. Avec les notations du théorème \eqref{regularisation} et en utilisant les développement asymptotiques \eqref{dvlptz1}, on a pour tout $n\geq 1$:
\[
\begin{split}
 p_{m,n}(z)&:=-\log G_{n+m}(nz)+\bigl(2n+m-1\bigr)\log (nz)-\bigl(2n+m-1\bigr)\log 2
-\log \Gamma(n)\\
&-\log \Gamma(n+m+1)+\frac{w_m(z)}{n}\\
&=-2nz+ \log(\pi nz)+\bigl(2n+m-1\bigr)\log (nz)-\bigl(2n+m-1\bigr)\log 2-\log \Gamma(n)\\
&-\log \Gamma(n+m+1)+ \mathrm{O}(\frac{1}{z})\quad \text{par}\, \eqref{developpementGn}.
\end{split}
\]
Par suite
\[
 a_{m,n}=n+\frac{m}{2},
\]
et
\[
 b_{m,n}:=-\log\Gamma(n+m+1)-\log \Gamma(n)+(2n+m)\log n-(2n+m-1)\log 2+\log \pi.
\]

Donc
\[
 A_m(s)=\zeta_\Q(2s-1)+\frac{m}{2}\zeta_\Q(2s),
\]
et
\[
\begin{split}
B_m(s)&=-\sum_{n\geq 1}\bigl(\frac{\Gamma(n+m+1)}{n^{2s}}+\frac{\log \Gamma(n)}{n^{2s}} \bigr)+2\sum_{n\geq 1}\frac{\log n}{n^{2s-1}}+m\sum_{n\geq 1}\frac{\log n}{n^{2s}}-2\log 2\sum_{n\geq 1}\frac{1}{n^{2s-1}}\\
&-(m-1)\log 2\sum_{n\geq 1}\frac{1}{n^{2s}}+\log \pi \sum_{n\geq 1}\frac{1}{n^{2s}}\\
&=-\sum_{n\geq 1}\frac{\log[(n+m)(n+m-1)\cdots n]}{n^{2s}}-2\sum_{n\geq 1} \frac{\log \Gamma(n)}{n^{2s}}-2\zeta_\Q'(2s-1)-m\zeta_\Q'(2s)\\
&-2\log 2 \;\zeta_\Q(2s-1)-(m-1)\log 2\;\zeta_\Q(2s)+\log \pi\; \zeta_\Q(2s)\\
&=-\sum_{n\geq 1}\sum_{k=1}^m \frac{\log (n+k)}{n^{2s}}-2\sum_{n\geq 1}\frac{\log \Gamma(n)}{n^{2s}}-2\zeta_\Q'(2s-1)-(m-1)\zeta_\Q'(2s)-2\log 2\; \zeta_\Q(2s-1)\\
&-(m-1)\log 2\; \zeta_\Q(2s)+\log \pi\; \zeta_\Q(2s).
\end{split}
\]
et
\[
P_m(s)=\sum_{n\geq 1}\frac{w_m(0)}{n^{2s+1}}=-\sum_{n\geq 1}\frac{1}{n^{2s+1}}\Bigl(\frac{m(m+1)}{2}+\frac{1}{6}\Bigr).
\]

et on note  par $F_m$,
\[
F_m(s):=\frac{s^2}{\Gamma(s+1)}\zeta_\Q(2s+1)\int_0^\infty \frac{t^{s-1}}{2\pi i}\int_{\Lambda_c}\frac{e^{-\la t}}{-\la}w_m(\la)d\la dt
\]
On montre que
\[
 F_m(s)=\frac{s^2}{\Gamma(s+1)}\zeta_\Q(2s+1)\Bigl(\frac{2+6m+6m^2}{s}-2 \Bigr)\frac{\Gamma(s+\frac{1}{2})}{12\sqrt{\pi}} \forall \,|s|\ll1,
\]

et que
\[
 F'_m(0)=\frac{1}{6}\gamma +\frac{1}{2}m(m+1)\gamma -\frac{1}{6}\log 2-\frac{1}{2}m(m+1)\log 2-\frac{1}{12}. \\
\]

On a besoin d'expliciter dans la formule de $\zeta_{\geq m+1}$, la quantité $P_m-B_m$. C'est l'objet du lemme suivant:
\begin{lemma}
Pour $|s|\ll1$, on a\[
\begin{split}
 (-B_m+P_m)(s)&=\sum_{k=1}^m \gamma_k(s)+2\eta(s)-m\zeta_\Q'(2s)+2\zeta_\Q'(2s-1)+(m-1)\zeta_\Q'(2s)\\
&+2\log 2 \,\zeta_\Q(2s-1)+(m-1)\log 2\,\zeta_\Q(2s)-\log \pi \,\zeta_\Q(2s),
\end{split}
\]
avec $\gamma_k(s)=\sum_{n\geq 1}\frac{1}{n^{2s}}\bigl( \log(1+\frac{k}{n})-\frac{k}{n}\bigr)$, qui est une fonction analytique au voisinage de $s=0$.
\end{lemma}

\begin{proof}

Par application du théorème des accroissements finis à la fonction $x\mapsto \log(1+x)-x$ sur  $\R^+$ on a
\[
 \Bigl|\log(1+\frac{k}{n})-\frac{k}{n}\Bigr|\leq \frac{k^2}{n^2}\quad  \forall\, n\geq 1.
\]
Donc la série numérique suivante:
\[
 \gamma_{k}(s)=\sum_{n\geq 1}\frac{1}{n^{2s}}\bigl( \log(1+\frac{k}{n})-\frac{k}{n}\bigr),
\]
converge pour tout complexe fixé $s$ vérifiant $\mathrm{Re}(s)>-\frac{1}{2}$.\\

De plus, si $s$ est tel que $|\mathrm{Re}(s)|\leq \frac{1}{4}$ alors

\[
\Bigl| \sum_{n\geq N+1} \frac{1}{n^{2s}}(\log(1+\frac{k}{n})-\frac{k}{n}) \Bigr|\leq k^2\sum_{n\geq N+1}\frac{1}{n^\frac{3}{2}} \quad \forall\, N\in \N^\ast,
\]
d'où la convergence normale et comme les somme partielles sont continues en $s$ sur $|\mathrm{Re}(s)|<\frac{1}{2}$ alors $\gamma_k$ est continue et
\[
 \gamma_k(0)=\sum_{n\geq 1}\log(1+\frac{k}{n})-\frac{k}{n}
\]

On sait que,  voir \cite[6.1.3]{Table2},
\[
 -\log \Gamma(z)=\log z +\gamma z +\sum_{n\geq 1} \log (1+\frac{z}{n})-\frac{z}{n}\quad \mathrm{Re}(z)>0,
\]
donc
\[
 \gamma_k(0)=-\log \Gamma(k+1)-\gamma k.
\]

On a
\[
 \begin{split}
  \sum_{n\geq 1}\sum_{k=1}^m\frac{\log(n+k)}{n^{2s}}\;-&\frac{m(m+1)}{2}\sum_{n\geq 1}\frac{1}{n^{2s+1}}= \sum_{n\geq 1}\sum_{k=1}^m \Bigl(\frac{\log(n+k)}{n^{2s}}-\frac{k}{n^{2s+1}} \Bigr) \\
&=\sum_{n\geq 1}\Bigl(m\frac{\log n}{n^{2s}}+\frac{1}{n^{2s}}\sum_{k=1}^m\bigl(\log(1+\frac{k}{n})-\frac{k}{n} \bigr)\Bigr)\\
&=-m\zeta_\Q'(2s)+\sum_{n\geq 1}\frac{1}{n^{2s}}\sum_{k=1}^m\bigl(\log(1+\frac{k}{n})-\frac{k}{n}\bigr).
 \end{split}
\]

En $s=0$, la fonction précédente est égale à
\[
 \begin{split}
-m\zeta_\Q'(0)+\sum_{k=1}^m\sum_{n\geq 1}\Bigl(\log(1+\frac{k}{n})-\frac{k}{n} \Bigr)&=-m\zeta_\Q'(0)+\sum_{k=1}^m \bigl(-\log \Gamma(k)-\log k -\gamma k\bigr)\\
&=-m\zeta_\Q'(0)-\gamma \frac{m(m+1)}{2}-\sum_{k=1}^m\log \Gamma(k+1).
 \end{split}
\]

On déduit que
{\allowdisplaybreaks
\begin{align*}
 (-B_m+P_m)(s)&=\sum_{n\geq 1}\sum_{k=1}^m \frac{1}{n^{2s}}\Bigl(\log (1+\frac{k}{n})-\frac{k}{n}\Bigr)-m\zeta_\Q'(2s)+2\sum_{n\geq 1}\Bigl(\frac{\log \Gamma(n)}{n^{2s}}-\frac{1}{n^{2s+1}} \Bigr)\\
&+2\zeta_\Q'(2s-1)+(m-1)\zeta_\Q'(2s)+2\log 2\, \zeta_\Q(2s-1)+(m-1)\log 2\,\zeta_\Q(2s)\\
&-\log \pi\, \zeta_\Q(2s)\\
&=\sum_{k=1}^m \gamma_k(s)+2\eta(s)-m\zeta_\Q'(2s)+2\zeta_\Q'(2s-1)+(m-1)\zeta_\Q'(2s)\\
&+2\log 2 \,\zeta_\Q(2s-1)+(m-1)\log 2\,\zeta_\Q(2s)-\log \pi\, \zeta_\Q(2s).
\end{align*}
}

\end{proof}

On peut donc évaluer les valeurs de $\zeta_{\geq m+1}(0)$ et $\zeta_{\geq m+1}'(0)$. On a
\[
 \begin{split}
  \zeta_{n\geq m+1}(0)&=-A_{m}(0)+F_m(0)\\
&=\frac{1}{12}+\frac{m}{4}+\frac{1}{12}(1+3m+3m^2)\\
&=\frac{1}{6}+\frac{m}{2}+\frac{m^2}{4}
 \end{split}
\]
et
\[
 \zeta'_{n\geq m+1}(0)=2\zeta_\Q'(-1)-\sum_{k=1}^m \log \Gamma(k+1)+\Bigl(\frac{1}{6}-\frac{m(m+1)}{2}\Bigr)\log 2+\frac{m+1}{2}\log \pi-\frac{1}{12}.
\]

 Notons par $z_m$,
\[
z_m(s):=\sum_{n\geq 0}\sum_{\la\in Z_n}\frac{1}{\la^{2s}},
\]
(ici les zéros sont comptés avec leurs multiplicités).\\

Rappelons  que
\[\zeta_{\Delta_{\overline{\mathcal{O}(m)_\infty}}}(s)=4^s z_m(s).\]

Il reste à déterminer la contribution de $\bigl\{G_{n},\, 0\leq n\leq m  \bigr\}$:
\begin{enumerate}
\item
\textbf{Si $m$ est pair}. Remarquons que $G_n=G_{n-m}$. On a
\[
z_m(s)=2\zeta_{\geq m+1}(s)+2\zeta_{0\leq n\leq \frac{m}{2}-1}(s)+\zeta_{\frac{m}{2}}(s),
\]
avec
\[
\zeta_{0\leq n \leq \frac{m}{2}-1}(s)=\sum_{0\leq n \leq \frac{m}{2}-1}\sum_{\la\in Z_{n}}\frac{1}{\la^{2s}},
\]
et
\[
\zeta_{\frac{m}{2}}(s)=\sum_{\la\in Z_{\frac{m}{2}}}\frac{1}{\la^{2s}}.
\]
En appliquant \eqref{valeurszeta}, on a
\[
\zeta_{0\leq n \leq \frac{m}{2}-1}(0)=-\sum_{0\leq n \leq \frac{m}{2}-1}\frac{m+2}{2},
\]
et
\[
\begin{split}
\zeta_{0\leq n \leq \frac{m}{2}-1}'(0)&=-\sum_{0\leq n \leq \frac{m}{2}-1}b_n\\
&=-\sum_{0\leq n \leq \frac{m}{2}-1}\Bigl( \log \pi -(m+1)\log 2-\log (n+1)!-\log (m-n+1)!\\
&+\log (m+2)\Bigr)\\
&=-\log \pi \frac{m}{2}+\frac{m(m+1)}{2}\log 2-\frac{m}{2}\log(m+2)+\sum_{0\leq n \leq \frac{m}{2}-1}\Bigl(\log (n+1)!\\
&+\log (m-n+1)!\Bigr)\\
&=-\log \pi \,\frac{m}{2}+\frac{m(m+1)}{2}\log 2-\frac{m}{2}\log(m+2)+\sum_{0\leq n\leq m, n\neq \frac{m}{2}}\log (n+1)!,
\end{split}
\]
On a aussi,
\[
\zeta_{\frac{m}{2}}(0)=\frac{m+2}{2},
\]
et
\[
\zeta_{\frac{m}{2}}'(0)=\log \pi -(m+1)\log 2-\log \bigl(\frac{m}{2}+1\bigr)!-\log \bigl(\frac{m}{2}+1\bigr)!+\log (m+2).
\]
En regroupant tout cela,  on obtient:
\[
z_m(0)=2\Bigl(\frac{1}{6}+\frac{m}{2}+\frac{m^2}{4}\Bigr)- 2\frac{m(m+2)}{4}-\frac{m+2}{2}=-\frac{2}{3}-\frac{m}{2},
\]
 et
 \[
\begin{split}
 z_m'(0)&=2\Bigl(2\zeta_\Q'(-1)-\sum_{k=1}^m \log \Gamma(k+1)+(\frac{1}{6}-\frac{m(m+1)}{2})\log 2+\frac{m+1}{2}\log \pi-\frac{1}{12}
\Bigr)\\
&+2\Bigl(-\log \pi \frac{m}{2}+\frac{m(m+1)}{2}\log 2-\frac{m}{2}\log(m+2)+\sum_{0\leq n\leq m, n\neq \frac{m}{2}}\log (n+1)!\Bigr)\\
&-\log \pi +(m+1)\log 2+\log \bigl(\frac{m}{2}+1\bigr)!+\log \bigl(\frac{m}{2}+1\bigr)!-\log (m+2)\\
&=4\zeta'_\Q(-1)+2\log (m+1)!+\bigl(m+\frac{4}{3}\bigr)\log 2-\frac{1}{6}-(m+1)\log(m+2)\\
&= 4\zeta'_\Q(-1)-\frac{1}{6}+\bigl(m+\frac{4}{3}\bigr)\log 2-2\log \frac{(m+2)^{\frac{m+1}{2}}}{(m+1)!}
\end{split}
 \]

\item \textbf{Si $m$ est impair}.
On rappelle que:
\[\quad z_m(s)=2\zeta_{\geq m+1}+2\zeta_{0\leq n\leq [\frac{m}{2}]}(s).\]
 Il suffit d'étudier la fonction
\[
\zeta_{0\leq n\leq [\frac{m}{2}]}(s)=\sum_{0\leq n\leq [\frac{m}{2}]}\sum_{\la \in Z_n}\frac{1}{\la^{2s}},
\]

 et on vérifie que
\[
 \zeta_{0\leq n\leq [\frac{m}{2}]}(0)=-\frac{(m+1)(m+2)}{4},
\]
et que
\[
\zeta'_{0\leq n\leq [\frac{m}{2}]}(0)=-\frac{m+1}{2}\log \pi +\frac{(m+1)^2}{2}\log 2-\frac{1}{2}\log \Bigl(\frac{(m+2)^{m+1}}{((m+1)!)^2}\Bigr) +\sum_{k=1}^m \log \Gamma(k+1).
\]

\end{enumerate}

Donc,
\[
\begin{split}
 z_m(0)&=2\Bigl(\frac{1}{6}+\frac{m}{2}+\frac{m^2}{4}-\frac{m^2+3m+2}{4}  \Bigr)\\
&=-\frac{2}{3}-\frac{m}{2},
\end{split}
\]
et

\[
\begin{split}
 z'_{ m}(0)&= 2\zeta'_{n\geq m+1}+2\zeta'_{0\leq n\leq [\frac{m}{2}]}(0)\\
&=4\zeta_\Q'(-1)+\bigl(\frac{4}{3}+m\bigr)\log 2-\frac{1}{6}-\log \Bigl(\frac{(m+2)^{m+1}}{((m+1)!)^2}\Bigr).
\end{split}
\]

On conclut que pour tout $m\geq 1$ on a:

\[
\zeta_{\Delta_{\overline{\mathcal{O}(m)_\infty}}}(0)=z_m(0)=-\frac{2}{3}-\frac{m}{2},\]
et
\[
\begin{split}
 \zeta_{\Delta_{\overline{\mathcal{O}(m)_\infty}}}'(0)&=z_m'(0)+\log 4 z_{m}(0)\\
 &=4\zeta_\Q'(-1)-\frac{1}{6}-2\log \frac{(m+2)^{\frac{m+1}{2}}}{(m+1)!}+\bigl(\frac{4}{3}+m\bigr)\log 2-\bigl(\frac{4}{3}+m\bigr)\log 2 \\
 &=4\zeta_\Q'(-1)-\frac{1}{6}-\log h_{L^2,\infty,\infty}+\bigl(\frac{4}{3}+m\bigr)\log 2-(\frac{4}{3}+m)\log 2, \quad \text{voir}\,\eqref{notevolume} \\
&=4\zeta_\Q'(-1)-\frac{1}{6}-\log h_{L^2,\infty,\infty},
\end{split}
\]
qu'on écrit sous la forme

\begin{equation}\label{valeurtorsion}
 \zeta_{\Delta_{\overline{\mathcal{O}(m)_\infty}}}'(0)+\log h_{L^2,\infty,\infty}=4\zeta_\Q'(-1)-\frac{1}{6}.
\end{equation}

Par les notations et les calculs du paragraphe \eqref{tttt}, on constate que

\[\zeta_{\Delta_{\overline{\mathcal{O}(m)_\infty}}}'(0)=T_g\bigl(\overline{T\p^1}_\infty; \overline{\mathcal{O}(m)}_\infty\bigr).\]

\section{Annexe}
\subsection{Sur un calcul de classes de Bott-Chern}

Soit $X$ une variété analytique  complexe. Pour tout $p$ et $q$ deux entiers positifs, on note par $A^{(p,q)}(X)$ l'ensemble   des formes différentielles de degré $(p,q)$ sur $X$, par $\widetilde{A}^{(p,p)}(X)$ le quotient de $A^{(p,p)}(X)$ par $\pt A^{(p-1,p)}(X)+\overline{\pt}A^{(p,p-1)}(X)$ et on pose

\[
\widetilde{A}(X)=\oplus_{p\in \N}\widetilde{A}^{(p,p)}(X).
\]
Si $\omega\in \widetilde{A}(X)$, on note pour tout $k\in \N$,  $[\omega]^{(k)}$ l'élément de $\widetilde{A}^{(k,k)}(X)$ tel que $\omega=\sum_{k\in \N }[\omega]^{(k)}$.\\

A toute série formelle symétrique $\phi$ on peut associer une classe caractéristique encore notée $\phi$ et une
classe secondaire de Bott-Chern $\widetilde{\phi}$ comme dans \cite[§ 1]{Character}. Si  $\phi(x)=\sum_{k\in \N}a_k
x^k$ est une série formelle en une seule variable $x$, on a pour tout fibré en
droites holomorphe $L$,
\[
 \phi(L)=\sum_{k\in \N}a_k c_1(L)^k.
\]
Soit $\overline{\mathcal{E}}$ une suite exacte de fibrés en droites hermitiens sur $X$:
\[
 \overline{\mathcal{E}}:0\lra (L,h_1) \lra (L,h_2)\lra 0,
\]
et on note par $ \widetilde{\phi}(L,h_1,h_2)$ la classe caractéristique associée. D'après \cite[Proposition
1.3.1]{Character}, on a:
\[
 \widetilde{\phi}(L,h_1,h_2)=\sum_{k\in \N}a_k \,\widetilde{c_1^k}(L,h_1,h_2).
\]
Donc, pour déterminer $ \widetilde{\phi}(L,h_1,h_2)$, il suffit de calculer les
$\widetilde{c_1^k}(L,h_1,h_2)$ pour tout $k\in \N$. C'est le but de la proposition suivante:

\begin{proposition}\label{ch}
Soit $X$ une variété  analytique  complexe et $L$ un fibré en droites holomorphe sur $X$.  Si $L$ est muni  de deux métriques hermitiennes de classes $\mathcal{C}^\infty$, $h_1$ et $h_2$, alors
\begin{equation}\label{bc1}
 \widetilde{\mathrm{ch}}(L,h_1,h_2)=\sum_{k\geq 1} -\frac{1}{k!}\log \frac{h_2}{h_1} \Bigl(dd^c(-\log
\frac{h_2}{h_1})\Bigr)^{{k-1}} \mathrm{ch}(L,h_1),
\end{equation}
et
\begin{equation}
 \widetilde{c_1^k}(L,h_1,h_2)=k!\bigl[  \widetilde{\mathrm{ch}}(L,h_1,h_2)\bigr]^{(k)},
\end{equation}

dans $\widetilde{A}(X)$, où $f$ est  une  fonction  est de classe $\cl$ sur $X$ définie localement par $-\log \frac{h_2(s,s)}{h_1(s,s)}$,  avec $s$  est une section holomorphe locale de $L$.
\end{proposition}

\begin{proof}    Soit  $L$ un fibré en droites holomorphe sur $X$ qu'on munit de  deux métriques hermitiennes $h_{1}$ et $h_2$ de classe $C^{\infty}$. Si $h$ est une autre métrique hermitienne $C^\infty$ sur $L$, on a:
\begin{enumerate}
\item $\widetilde{\mathrm{ch}}(L,h,h)=0$, voir \cite[Théorème 1.2.2 (iii)]{Character},
\item $\widetilde{\mathrm{ch}}(L,h_1,h_2)=\widetilde{ch}(L,h_1,h)+\widetilde{\mathrm{ch}}(L,h,h_2)$ voir \cite[Corollaire 1.3.4]{Character},
\item $\widetilde{\mathrm{ch}}(L,h_1,h_2)=\widetilde{\mathrm{ch}}(\mathcal{O},h_1\otimes h_1^{-1},h_2\otimes
h_1^{-1})\mathrm{ch} (L,h_1)$,
, voir \cite[(1.3.5.2)]{Character}.
\end{enumerate}

Il suffit donc  de déterminer  $\widetilde{\mathrm{ch}}(\mathcal{O},h_1,h_2)$ avec $h_1$ est la métrique triviale et $h_2$ est $\cl$ quelconque sur $\mathcal{O}$. \\

Sur $X\times \mathbb{P}^{1}$, on considère le fibré en droites trivial $\mathcal{O}$ muni de la métrique $h_\rho$ définie par la fonction $\rho f$, où $f$ est  la fonction qui définit la métrique $h_2$ (c.à.d $h_2(1,1)= e^{f} $) et $\rho$ est l'image réciproque par la projection de $X\times \p^1$ sur $\p^1$ d'une fonction positive $C^\infty$ sur $\mathbb{P}^1$ telle que $\rho (0)=1$ et $\rho (\infty)=0$.

Avec ces notations, on a d'après \cite{Character}:
\[
\widetilde{\mathrm{ch}}\bigl(\mathcal{O},h_1,h_2 \bigr)=-\int_{\mathbb{P}^1}\mathrm{ch} (\mathcal{O},\tilde{h})\log|z|^2
\]
dans $\widetilde{A}(X)$.
On a
\[
\begin{split}
\mathrm{ch}(\mathcal{O},\tilde{h})&=\mathrm{ch}\bigl(dd^{c}(\log \widetilde{h}(1,1)\bigr) \\
&=\mathrm{ch}\bigl(dd^{c}(\rho f)\bigr)\\
&=\sum_{j\in \mathbb{N}}\frac{1}{j!}\bigl(dd^{c}(\rho f)\bigr)^j.
\end{split}
\]
Montrons  que :
\[
\begin{split}
-\int_{\mathbb{P}^1}\mathrm{ch} (\mathcal{O},\tilde{h})\log|z|^2= \sum_{j\in\mathbb{N}}\frac{1}{j!} f (dd^{c}f)^{j}.
\end{split}
\]
Dans $A(X\times \p^1)=A(X)\otimes A(\p^1)$, on a pour tout $j\in \N^\ast$,
\[
\begin{split}
\bigl(dd^{c}(\rho f)\bigr)^j&=\Bigl(f dd^{c}(\rho)+d^{c}\rho d f+d\rho d^{c} f +\rho\, dd^{c} f\Bigr)^j\\
&=\omega_{1}^j + j((dd^{c}\rho)f) \omega_1^{j-1} + 0,\\
\end{split}
\]
où on a posé $\omega_1:=d^{c}\rho d f+d\rho d^{c} f +\rho dd^{c} f$ et les autres termes sont nuls puisque  $\bigl(dd^c \rho\bigr)^k=0$ pour $k\geq 2$. %$ \omega_1\in A(X)\otimes A^{(0,1)}(\p^1)+A(X)\otimes A^{(1,0)}(\p^1)$.\\
Par le même  argument, on a
\[
\begin{split}
\bigl(f\,dd^{c}\rho\bigr) \omega_1^{j-1}&=f \bigl(dd^{c}\rho\bigr)\Bigl(d^c \rho\, df + d\rho \, d^c f +\rho dd^cf \Bigr)^{j-1}\\
&=f(dd^c\rho)\Bigl(\rho^{j-1}(dd^cf)^{j-1}+\cdots\Bigr)\\
&=f \rho^{j-1} dd^c\rho\,(dd^cf)^{j-1}+ 0\\
&=(\rho^{j-1}dd^c\rho) f (dd^cf)^{j-1},
\end{split}
\]
et
\[\begin{split}
\omega_1^j&=\Bigl(d^{c}\rho \,d f+d\rho\, d^{c} f +\rho \,dd^{c}f\Bigr)^j\\
&=\rho^j (dd^cf)^j+j\rho^{j-1}\bigl(d^{c}\rho\, d f+d\rho\, d^{c} f\bigr)(dd^cf)^{j-1}+j
(j-1)\rho^{j-2}\bigl(d^{c}\rho d f+d\rho d^{c} f\bigr)^2 (dd^cf)^{j-2}+0\\
&=\rho^j (dd^cf)^j+j\bigl(d^{c}\rho d f+d\rho d^{c} f\bigr)(\rho^{j-1}(dd^cf)^{j-1})-2j (j-1)\rho^{j-2}(d\rho\,
d^c\rho)d^{c}f\, df (dd^cf)^{j-2}.
\end{split}
\]
Si $\alpha\in A(X\times \p^1)$, on va noter par $\lbrace\alpha \rbrace$ la composante de $\alpha$ contenant $dz \wedge d^c z$ (où $z$ une coordonnée locale sur $\mathbb{P}^{1}$).  On a donc,
\[\lbrace \omega_1^j\rbrace=-2j (j-1)\rho^{j-2}d\rho \,d^c\rho\,d^{c}f \,df (dd^cf)^{j-2}.
\]
Par suite,
\begin{align*}
\lbrace(dd^{c}&(\rho f))^j\rbrace=j\bigl(\rho^{j-1}dd^c\rho\bigr) f (dd^c f)^{j-1}-2j (j-1)\rho^{j-2}(d\rho d^c\rho)d^{c}f df (dd^cf)^{j-2}\\
&=j\rho^{j-1}dd^c\rho\,  f (dd^cf)^{j-1}+2j (j-1)\rho^{j-1}\,d\rho\, d^c\rho\,\Bigl(d(f d^{c}f)
-f dd^{c}f\Bigr)(dd^c f)^{j-2}\\
&=\rho^{j-1}\Bigl(j\,dd^c\rho-2j (j-1)d\rho\, d^c\rho\, \Bigr)f (dd^cf)^{j-1}+2j (j-1)\rho^{j-1}\,d\rho\,
d^c\rho\,d(f d^{c}f)
(dd^c f)^{j-2} \\
&= \alpha_{j}(z)f (dd^cf)^{j-1}+\beta_j(z) d(f d^{c}f)
(dd^c f)^{j-2}
\end{align*}
avec  $\alpha_{j}:=\rho^{j-1}\Bigl(j\,dd^c\rho-2j (j-1)d\rho\, d^c\rho\, \Bigr)$ et $\beta_j:= 2j
(j-1)\rho^{j-1}\,d\rho\, d^c\rho$ sont deux formes différentielle de degré maximal sur $\mathbb{P}^1$. \\

Donc, on a dans $\widetilde{A}(X)$,
\begin{align*}
\tilde{\mathrm{ch}}(\mathcal{O},h_1,h_2)&= -\int_{\mathbb{P}^{1}}\mathrm{ch}(dd^{c}(\rho f))\log |z|^2\\
&= -\sum_{j\geq 1}\frac{1}{j!} \Bigl(\int_{\mathbb{P}^{1}}\alpha_{j}(z) \log |z|^{2}\Bigr)f (dd^cf)^{j-1}-\sum_{j\geq 1}\frac{1}{j!}\Bigl(\int_{\p^1}\beta_j(z) \log |z|^2\Bigr) d(f d^{c}f)
(dd^c f)^{j-2} \\
&=\sum_{j\geq 1}\frac{1}{j!} f (dd^cf)^{j-1}.
\end{align*}

On a donc montré que:
\[
\begin{split}
\widetilde{ch}\bigl(L,h_1,h_2\bigr)&= \widetilde{ch}\bigl(\mathcal{O},h_1\otimes h_1^{-1},h_2\otimes h_1^{-1}\bigr) ch(L,h_1)\\
&=\sum_{j\geq 1} \frac{1}{j!}\Bigl(-\log \frac{h_2}{h_1}\Bigr) \Bigl(dd^c\bigl(-\log \frac{h_2}{h_1}\bigr)\Bigr)^{{j-1}} ch(L,h_1).
\end{split}
\]
\end{proof}

\subsection{Un rappel sur la théorie des fonctions de Bessel}

Dans cette introduction à la théorie des fonctions de Bessel, la principale référence est le chapitre 17 du \cite{Whittaker}. On s'intéressera à une sous-classe de fonctions de Bessel à savoir les fonctions de Bessel d'ordre un entier.\\

Pour tout $z$ fixé, la fonction
\[
t\mapsto e^{\frac{1}{2}z(t-\frac{1}{t})},\quad t\neq 0
\]
admet un développement en une série de Laurent en $t$. Soit $n\in \Z$. Par définition la fonction de Bessel d'ordre $n$ est la fonction qui pour tout $z\in \C$ associe $J_n(z)$ le coefficient de $t^n$ dans ce développement.\\

\begin{itemize}
\item[$\bullet$]
%\item[$\bullet$][$\bullet$]
 $J_n$ est une fonction analytique sur $\C$ telle que $J_{-n}=(-1)^nJ_n$ et
\[
J_n(z)=\sum_{r=0}^\infty \frac{(-1)^r(\frac{1}{2}z)^{n+2r}}{\Gamma(r+1)\Gamma(n+r+1)}\quad\forall\, z\in \C \quad\text{si}\quad n\in \N.
\]
\item[$\bullet$]
$J_n$ est une solution de l'équation différentielle linéaire suivante:
\[
\frac{d^2 y}{d^2 z}+\frac{1}{z}\frac{dy}{dz}+\bigl(1-\frac{n^2}{z^2} \bigr)y=0
\]

\textbf{Relations de récurrence}: On a pour tout $z\in \C$
\item[$\bullet$]
%\item[$\bullet$][$\bullet$]
\begin{equation}\label{B1}
\frac{d}{dz}\bigl(z^{-n}J_n(z)\bigr)=-z^{-n}J_{n+1}(z),
\end{equation}
\item[$\bullet$]
%\item[$\bullet$][$\bullet$]
\begin{equation}\label{B2}
J_n'(z)=\frac{n}{z}J_n(z)-J_{n+1}(z),
\end{equation}
\item[$\bullet$]
%\item[$\bullet$][$\bullet$]
\begin{equation}\label{B3}
J'_n(z)=\frac{1}{2}(J_{n-1}(z)-J_{n+1}(z) ),
\end{equation}

\item[$\bullet$]
%\item[$\bullet$][$\bullet$]
\begin{equation}\label{B4}
J_n'(z)=J_{n-1}(z)-\frac{n}{z}J_n(z),
\end{equation}
\item[$\bullet$]
%\item[$\bullet$][$\bullet$]
\begin{equation}\label{B5}
\frac{d}{dz}(z^n J_n(z))=z^n J_{n-1}(z),
\end{equation}
\item[$\bullet$]
%\item[$\bullet$][$\bullet$]
Les fonctions de Bessel modifiées d'ordre entier sont  les fonctions $I_n$  données par
\[
I_n(z):=i^{-n}J_n(iz)=\sum_{r=0}^\infty \frac{(\frac{1}{2}z)^{n+2r}}{\Gamma(r+1)\Gamma(n+r+1)}\quad \forall\,z\in \C,
\]

On a
\item[$\bullet$]
%M\item[$\bullet$][$\bullet$]
\begin{equation}\label{auvoisinagezero}
I_n(z)=\frac{1}{2^n \Gamma(n+1)}z^n+o(z^n), \quad |z|\ll1,
\end{equation}

\item[$\bullet$]
%\item[$\bullet$][$\bullet$]
\begin{equation}\label{B6}
I_{n-1}(z)-I_{n+1}(z)=\frac{2n}{z}I_n(z),
\end{equation}

\item[$\bullet$]
%\item[$\bullet$][$\bullet$]
\begin{equation}\label{B7}
\frac{d}{dz}(z^n I_n(z))=z^n I_{n-1}(z),
\end{equation}

\item[$\bullet$]

%\item[$\bullet$][$\bullet$]
\begin{equation}\label{B8}
\frac{d}{dz}(z^{-n}I_n(z))=z^{-n}I_{n+1}(z),
\end{equation}

\item[$\bullet$]
%\item[$\bullet$][$\bullet$]
\begin{equation}\label{B9}
\frac{d^2 I_n(z)}{dz^2}+\frac{1}{z}\frac{dI_n(z)}{dz}-(1+\frac{n^2}{z^2})I_n(z)=0.
\end{equation}

\end{itemize}

\begin{itemize}

\item[$\bullet$]
\begin{equation}\label{dvlptz1}
I_n(z)=\frac{e^z}{\sqrt{2\pi z}}\Bigl(1-\frac{4n^2-1}{8z}+O(\frac{1}{z^2}) \Bigr),
\end{equation}
avec $(|\arg(z)|<\frac{\pi}{2})$, cf. \cite[9.7.1]{Table2}.

\item[$\bullet$]
\begin{equation}\label{dvlptz2}
I_n'(z)=\frac{e^z}{\sqrt{2\pi z}}\bigl(1-\frac{4n^2+3}{8z}+O(\frac{1}{z^2}) \bigr),
\end{equation}
avec $(|\arg(z)|<\frac{\pi}{2})$,  cf. \cite[9.7.3]{Table2}.

\item[$\bullet$]

\begin{equation}\label{developpement1}
I_n(nz)=\frac{1}{(1+\eta_{k,1}(n,\infty))}\frac{e^{n\xi(z)}}{(2\pi n)^{\frac{1}{2}}(1+z^2)^\frac{1}{4}}\sum_{i=0}^\infty\frac{1}{n^i}U_i(z),
\end{equation}
$\forall z$ vérifiant $\mathrm{arg}(z)\leq \frac{\pi}{2}-\eps$, $0<\eps\ll1$, voir \cite[7.18 p.378]{Olver} ou \cite[9.7.7]{Table2}. Une estimation d'erreur est donnée par la formule \cite[7.14, p. 377]{Olver}.

\item[$\bullet$]

\begin{equation}\label{developpement2}
I_n'(nz)=\frac{(1+z^2)^\frac{1}{4}}{1+\eta_{2,1}(n,\infty)}\frac{e^{n\xi(z)}}{(2\pi n)^\frac{1}{2}z}\Bigl(1+\frac{1}{n}V_1+\frac{1}{n^2}(V_2-U_2) +\kappa_{2,1}(n,z)-\frac{z^2 p^3}{2n}\eta_{2,1}(n,z)  \Bigr),
\end{equation}
où
\begin{equation}\label{dev4}
V_1(z)=-\frac{3}{8}\frac{1}{(1+z^2)^\frac{1}{2}}+\frac{7}{24}\frac{1}{(1+z^2)^\frac{3}{2}},
\end{equation}
voir \cite[9.7.9]{Table2} ou \cite[ex\, 7.2\, p.378]{Olver}.

\end{itemize}

\bibliographystyle{plain}
\bibliography{biblio}

\begin{thebibliography}{10}

\bibitem{Table2}
Milton Abramowitz and Irene~A. Stegun, editors.
\newblock {\em Handbook of mathematical functions with formulas, graphs, and
  mathematical tables}.
\newblock Dover Publications Inc., New York, 1992.
\newblock Reprint of the 1972 edition.

\bibitem{Barnes}
E.~W. Barnes.
\newblock {T}he {T}heory of the {G}-funtion.
\newblock {\em Quarterly Journ. Pure and Appl. Math.}, (31):264--314, 1900.

\bibitem{heat}
Nicole Berline, Ezra Getzler, and Mich{\`e}le Vergne.
\newblock {\em Heat kernels and {D}irac operators}.
\newblock Grundlehren Text Editions. Springer-Verlag, Berlin, 2004.
\newblock Corrected reprint of the 1992 original.

\bibitem{BGS1}
J.-M. Bismut, H.~Gillet, and C.~Soul{\'e}.
\newblock Analytic torsion and holomorphic determinant bundles. {I}.
  {B}ott-{C}hern forms and analytic torsion.
\newblock {\em Comm. Math. Phys.}, 115(1):49--78, 1988.

\bibitem{ImmerBismut}
Jean-Michel Bismut and Gilles Lebeau.
\newblock Complex immersions and {Q}uillen metrics.
\newblock {\em Inst. Hautes \'Etudes Sci. Publ. Math.}, (74):ii+298 pp. (1992),
  1991.

\bibitem{GSZ}
H.~Gillet and C.~Soul{\'e}.
\newblock Analytic torsion and the arithmetic {T}odd genus.
\newblock {\em Topology}, 30(1):21--54, 1991.
\newblock With an appendix by D. Zagier.

\bibitem{Character}
Henri Gillet and Christophe Soul{\'e}.
\newblock Characteristic classes for algebraic vector bundles with {H}ermitian
  metric. {I}.
\newblock {\em Ann. of Math. (2)}, 131(1):163--203, 1990.

\bibitem{Table}
I.~S. Gradshteyn and I.~M. Ryzhik.
\newblock {\em Table of integrals, series, and products}.
\newblock Academic Press Inc., San Diego, CA, sixth edition, 2000.
\newblock Translated from the Russian, Translation edited and with a preface by
  Alan Jeffrey and Daniel Zwillinger.

\bibitem{Mounir1}
Mounir Hajli.
\newblock {S}pectre du {L}aplacien singulier associé aux métriques canoniques
  sur la droite projective complexe.
\newblock {\em Arxiv}.

\bibitem{these}
Mounir Hajli.
\newblock \emph{{T}héorie spectrale pour certaines métriques singulières et
  {G}éométrie d'{A}rakelov}.
\newblock {\em \rm{PhD thesis}}, Université Pierre et Marie Curie, 2012.

\bibitem{Hankel}
Hermann Hankel.
\newblock Die {C}ylinderfunctionen erster und zweiter {A}rt.
\newblock {\em Math. Ann.}, 1(3):467--501, 1869.

\bibitem{Hille2}
Einar Hille.
\newblock {\em Analytic function theory. {V}ol. {II}}.
\newblock Introductions to Higher Mathematics. Ginn and Co., Boston, Mass.-New
  York-Toronto, Ont., 1962.

\bibitem{Spectra}
Akira Ikeda and Yoshiharu Taniguchi.
\newblock Spectra and eigenforms of the {L}aplacian on {$S^{n}$} and
  {$P^{n}({\bf C})$}.
\newblock {\em Osaka J. Math.}, 15(3):515--546, 1978.

\bibitem{Landau}
E~Landau and A~Walfisz.
\newblock {Ü}ber die {N}ichtfortsetzbarkeit einiger durch {D}irichletsche
  {R}eihen definierte {F}unktionen.
\newblock {\em Rend. Circ. Mat. Palermo}, 44:82--86, 1920.

\bibitem{Maillot}
Vincent Maillot.
\newblock {G}\'eom\'etrie d'{A}rakelov des vari\'et\'es toriques et fibr\'es en
  droites int\'egrables.
\newblock {\em M\'em. Soc. Math. Fr. (N.S.)}, 80:vi+129, 2000.

\bibitem{Olver}
F.~W.~J. Olver.
\newblock {\em Asymptotics and special functions}.
\newblock Academic Press [A subsidiary of Harcourt Brace Jovanovich,
  Publishers], New York-London, 1974.
\newblock Computer Science and Applied Mathematics.

\bibitem{Quillen}
D.~Quillen.
\newblock Determinants of {C}auchy-{R}iemann operators on {R}iemann surfaces.
\newblock {\em Funct. Anal. Appl.}, pages 31--34,, 1985.

\bibitem{RaySinger}
D.~B. Ray and I.~M. Singer.
\newblock Analytic torsion for complex manifolds.
\newblock {\em Ann. of Math. (2)}, 98:154--177, 1973.

\bibitem{Soulé}
C.~Soul{\'e}.
\newblock {\em Lectures on {A}rakelov geometry}, volume~33 of {\em Cambridge
  Studies in Advanced Mathematics}.
\newblock Cambridge University Press, Cambridge, 1992.
\newblock With the collaboration of D. Abramovich, J.-F. Burnol and J. Kramer.

\bibitem{Spreafico}
M.~Spreafico.
\newblock Zeta function and regularized determinant on a disc and on a cone.
\newblock {\em J. Geom. Phys.}, 54(3):355--371, 2005.

\bibitem{Voros}
A.~Voros.
\newblock Spectral functions, special functions and the {S}elberg zeta
  function.
\newblock {\em Comm. Math. Phys.}, 110(3):439--465, 1987.

\bibitem{Watson}
G.~N. Watson.
\newblock {\em A {T}reatise on the {T}heory of {B}essel {F}unctions}.
\newblock Cambridge University Press, Cambridge, England, 1944.

\bibitem{Whittaker}
E.~T. Whittaker and G.~N. Watson.
\newblock {\em A course of modern analysis. {A}n introduction to the general
  theory of infinite processes and of analytic functions: with an account of
  the principal transcendental functions}.
\newblock Fourth edition. Reprinted. Cambridge University Press, New York,
  1962.

\bibitem{Zhang}
Shouwu Zhang.
\newblock Small points and adelic metrics.
\newblock {\em J. Algebraic Geom.}, 4(2):281--300, 1995.

\end{thebibliography}

\vspace{1cm}

\begin{center}
{\sffamily \noindent National Center for Theoretical Sciences, (Taipei Office)\\
 National Taiwan University, Taipei 106, Taiwan}\\

 {e-mail}: {hajli@math.jussieu.fr}

\end{center}

\end{document}